\documentclass[oneside,reqno]{amsart}

\usepackage[english]{babel}
\usepackage{enumitem}      
\usepackage{amssymb}
\usepackage{hyperref}
\usepackage{cleveref} 
\usepackage{geometry}
\usepackage{tikz}
\usepackage{textcomp}
\usepackage{mathtools}
\usepackage{verbatim}
\usepackage{tabularx}
\newcolumntype{C}[1]{>{\centering}m{#1}} 
\usepackage{hhline}
\usepackage{longtable}

\usepackage{csquotes}
\usepackage[style=alphabetic, url=false,  isbn=false, maxnames=4,  eprint=false, giveninits=true, doi=false]{biblatex} 
\newtheorem{theorem}{Theorem}[section]
\newtheorem{lemma}[theorem]{Lemma}
\newtheorem{proposition}[theorem]{Proposition}
\newtheorem{corollary}[theorem]{Corollary}

\theoremstyle{definition}
\newtheorem{definition}[theorem]{Definition}

\theoremstyle{remark}
\newtheorem{remark}[theorem]{Remark}
\newtheorem{example}[theorem]{Example}

\newcommand{\Rep}{\ensuremath{\mathop{\mathrm{Rep}}}}
\newcommand{\uRep}{\ensuremath{\mathop{\mathrm{\underline{Rep}}}}}
\newcommand{\Homm}{\ensuremath{\mathrm{Hom}}}
\newcommand{\End}{\ensuremath{\mathrm{End}}}
\newcommand{\tr}{\ensuremath{\mathrm{tr}}}
\newcommand{\id}{\ensuremath{\mathrm{id}}}
\newcommand{\ot}{\ensuremath{\otimes}}
\newcommand{\NN}{\mathbb N}
\newcommand{\ZZ}{\mathbb Z}

\newcommand{\CC}{\mathbb C}

\newcommand{\cC}{\mathcal C}
\renewcommand\L{\mathcal{L}}

\newcommand{\QR}[2]{\raisebox{1ex}{\ensuremath{#1}}\ensuremath{\mkern-3mu}\big/\ensuremath{\mkern-3mu}\raisebox{-1ex}{\ensuremath{#2}}}
\newcommand\RepCt{{\uRep(\cC,t)}}
\newcommand{\Projk}{{\operatorname{Proj}_{\cC}(k)}}
\newcommand\cP{\mathcal{P}}
\newcommand\ProjC{\operatorname{Proj}_\cC}
\newcommand\Projl{\operatorname{Proj}_\cC(l)}

\newcommand\Irr{\operatorname{Irr}}

\newcommand\even{{\text{even}}}

\def\La{\LPartition{0.6:1}{}}

\def\Laa{\LPartition{}{0.6:1,2}}

\def\Ua{\UPartition{0.4:1}{}}

\def\Uaa{\UPartition{}{0.4:1,2}}

\def\Paa{\Partition{
\Pline (1,0) (1,1)
}}

\def\Pab{\Partition{
\Psingletons 0to0.3:1
\Psingletons 1to0.7:1
}}

\def\Paaaa{\Partition{
\Pblock 0to0.3:1,2
\Pblock 1to0.7:1,2
\Pline (1.5,0.3) (1.5,0.7)
}}
\def\Paaab{\Partition{
\Pblock 1to0.7:1,2
\Psingletons 0to0.3:2
\Pline (1,0) (1,0.7)
}}

\def\Paabb{\Partition{
\Pblock 0to0.3:1,2
\Pblock 1to0.7:1,2
}}

\def\Pabab{\Partition{
\Pline (1,0) (2, 1)
\Pline (1,1) (2, 0)
}}

\def\Pabcb{\Partition{
\Psingletons 0to0.3:1
\Psingletons 1to0.7:2
\Pline (2,0) (1,1)
}}

\def\Pabcabc{\Partition{
\Pline (1,0) (3, 1)
\Pline (2,0) (2, 1)
\Pline (3,0) (1, 1)
}}

\def\singleton{{\mathord{\uparrow}}}
\def\upsingleton{{\mathord{\downarrow}}}

\def\Partition#1{%
\ifmmode
\mathchoice{
\,\PartitionA{0.25em}{0.7em}{1em}{#1}\,%
}{
\,\PartitionA{0.25em}{0.7em}{1em}{#1}\,%
}{
\,\PartitionA{0.175em}{0.55em}{0.7em}{#1}\,%
}{
\,\PartitionA{0.125em}{0.4em}{0.5em}{#1}\,%
}\else%
\thinspace\PartitionA{0.25em}{0.7em}{1em}{#1}\thinspace%
\fi%
}

\def\BigPartition#1{%
\ifmmode
\mathchoice{
\,\PartitionA{1em}{0.9em}{2.5em}{#1}\,%
}{
\,\PartitionA{0.75em}{0.7em}{2em}{#1}\,%
}{
\,\PartitionA{0.25em}{0.5em}{1em}{#1}\,%
}{
\,\PartitionA{0.20em}{0.4em}{0.8em}{#1}\,%
}
\else%
\thinspace\PartitionA{0.75em}{0.7em}{2em}{#1}\thinspace%
\fi%
}

\def\PartitionA#1#2#3#4{%
\pgfpicture
\pgfsetbaseline{#1}
\pgfsetxvec{\pgfpoint{#2}{0em}}
\pgfsetyvec{\pgfpoint{0em}{#3}}
#4
\endpgfpicture
}

\def\fihere#1\fi{\fi#1}

\def\Pline(#1,#2)#3(#4,#5){
\pgfpathmoveto{\pgfpointxy{#1}{#2}}
\pgfpathlineto{\pgfpointxy{#4}{#5}}
\pgfusepath{stroke}
}

\def\Pcurveto(#1,#2)(#3,#4)(#5,#6)(#7,#8){
\pgfpathmoveto{\pgfpointxy{#1}{#2}}
\pgfcurveto{\pgfpointxy{#3}{#4}}{\pgfpointxy{#5}{#6}}{\pgfpointxy{#7}{#8}}
\pgfusepath{stroke}
}

\def\Psingletons #1to#2:#3 {\PsingletonsA{#1}{#2}#3,,}
\def\PsingletonsA#1#2#3,{\ifx,#3,\else
\pgfpathmoveto{\pgfpointxy{#3}{#1}}
\pgfpathlineto{\pgfpointxy{#3}{#2}}
\pgfusepath{stroke}
\fihere\PsingletonsA{#1}{#2}\fi}

\def\Pblock #1to#2:#3 {\PblockA{#1}{#2}#3,,}
\def\PblockA#1#2#3,#4,{\ifx,#4,\else
\pgfpathmoveto{\pgfpointxy{#3}{#1}}
\pgfpathlineto{\pgfpointxy{#3}{#2}}
\pgfpathlineto{\pgfpointxy{#4}{#2}}
\pgfpathlineto{\pgfpointxy{#4}{#1}}
\pgfusepath{stroke}
\fihere\PblockA{#1}{#2}#4,\fi}

\def\Ptext(#1,#2)#3{%
\pgftext[at={\pgfpointxy{#1}{#2}}]{#3}
}

\def\Ppoint #1 #2:#3 {\PpointA{#1}{#2}#3,,}
\def\PpointA#1#2#3,{\ifx,#3,\else #2{#3}{#1}\fihere\PpointA{#1}{#2}\fi}

\def\Lsingletons #1#2 {%
	\ifcat\noexpand#1\noexpand\ \Ppoint0 #1#2 \else\Psingletons 0to#1#2 \fi
}
\def\Usingletons #1#2 {%
	\ifcat\noexpand#1\noexpand\ \Ppoint1 #1#2 \else\Psingletons 1to#1#2 \fi
}

\def\LPartition#1#2{\Partition{
\LPartitionB#2;;
\ifx,#1,\else\LPartitionA#1;;\fi
}}

\def\LPartitionA#1;{\ifx;#1;\else
\Lsingletons #1 
\fihere\LPartitionA\fi}

\def\LPartitionB#1;{\ifx;#1;\else
\Pblock 0to#1 
\fihere\LPartitionB\fi}

\def\UPartition#1#2{\Partition{
\UPartitionB#2;;
\ifx,#1,\else\UPartitionA#1;; \fi
}}

\def\UPartitionA#1;{\ifx;#1;\else
\Usingletons #1 
\fihere\UPartitionA\fi}

\def\UPartitionB#1;{\ifx;#1;\else
\Pblock 1to#1 
\fihere\UPartitionB\fi}

\def\Pcurveto(#1,#2)(#3,#4)(#5,#6)(#7,#8){
\pgfpathmoveto{\pgfpointxy{#1}{#2}}
\pgfcurveto{\pgfpointxy{#3}{#4}}{\pgfpointxy{#5}{#6}}{\pgfpointxy{#7}{#8}}
\pgfusepath{stroke}
}
\newcommand{\cross}{\Pabab}
\newcommand{\bbar}{\Pab}
\newcommand{\vierpart}{\Paaaa}
\newcommand{\twoblocks}{\Paabb}
\newcommand{\starpart}{\Pabcabc}
\newcommand{\legpart}{\Pabcb}
\newcommand{\twoblockspart}{\Partition{\Pblock 0 to 0.25:1,2 \Pblock 1 to 0.75:2,3 \Pline (3,0) (1,1)}}

\newcommand{\primarypart}{\Partition{\Pblock 0 to 0.25:1,2 \Pblock 1 to 0.75:2,3 \Pline (3,0) (1,1) \Pline (1.5,0.25) (2.5,0.75)}}

\newcommand{\JWiTikz}{\scalebox{0.8}{\begin{tikzpicture}
\coordinate [label=left:{\scalebox{1.25}{$e_k=$}}](O) at (0,0);
\draw (0.3,1) -- (0.3,0.5);
\draw (0.3,-1) -- (0.3,-0.5);
\coordinate [label=right:{$\cdots$}](O) at (0.6,0.7);
\draw (0,-0.5) -- (0,0.5) -- (2,0.5) -- (2,-0.5) -- (0,-0.5);
\coordinate [label=right:{\scalebox{1.25}{$e_{k-1}$}}](O) at (0.5,0);
\draw (2.3,1) -- (2.3,-1);
\draw (1.7,1) -- (1.7,0.5);
\draw (1.7,-1) -- (1.7,-0.5);
\coordinate [label=right:{$\cdots$}](O) at (0.6,-0.7);
\coordinate [label=right:{\scalebox{1.25}{$-~a_k$}}](O) at (2.5,0);
\draw (4.3,2) -- (4.3,1.5);
\draw (5.7,2) -- (5.7,1.5);
\coordinate [label=right:{$\cdots$}](O) at (4.6,1.7);
\draw (4,0.5) -- (4,1.5) -- (6,1.5) -- (6,0.5) -- (4,0.5);
\coordinate [label=right:{\scalebox{1.25}{$e_{k-1}$}}](O) at (4.5,1);
\draw (4.3,0.5) -- (4.3,-0.5);
\draw (5.1,0.5) -- (5.1,-0.5);
\coordinate [label=right:{$\cdots$}](O) at (4.3,0);
\draw (5.7,0.5) -- (5.7,0.3) to [bend right=90] (6.3,0.3) -- (6.3,2);
\draw (5.7,-0.5) -- (5.7,-0.3) to [bend left=90] (6.3,-0.3) -- (6.3,-2);
\draw (4,-0.5) -- (4,-1.5) -- (6,-1.5) -- (6,-0.5) -- (4,-0.5);
\coordinate [label=right:{\scalebox{1.25}{$e_{k-1}$}}](O) at (4.5,-1);
\draw (4.3,-2) -- (4.3,-1.5);
\draw (5.7,-2) -- (5.7,-1.5);
\coordinate [label=right:{$\cdots$}](O) at (4.6,-1.7);
\coordinate [label=right:{\scalebox{1.25}{$\in NC_2(k,k).$}}](O) at (6.5,0);
\end{tikzpicture}}}

\begin{document}

\hyphenation{theo-re-ti-cal group-theo-re-ti-cal
semi-sim-ple al-geb-ras di-men-sions sim-ple ob-jects
equi-va-lent}

\title[Interpolating Partition Categories]{Semisimplicity and Indecomposable Objects in Interpolating Partition Categories}

\author{Johannes Flake}
\address{Algebra and Representation Theory,
RWTH Aachen University,
Pontdriesch 10--16,
52062 Aachen, Germany}
\email{flake@art.rwth-aachen.de}

\author{Laura Maa\ss{}en}
\address{Lehrstuhl f\"ur Algebra und Zahlentheorie,
RWTH Aachen University,
Pontdriesch 10--16,
52062 Aachen, Germany}
\email{laura.maassen@rwth-aachen.de}

%
%

\begin{abstract} 

We study Karoubian tensor categories which interpolate representation categories of families of so-called easy quantum groups in the same sense in which Deligne's interpolation categories $\uRep(S_t)$ interpolate the representation categories of the symmetric groups. As such categories can be described using a graphical calculus of partitions, we call them interpolating partition categories. They include $\uRep(S_t)$ as a special case and can generally be viewed as subcategories of the latter. Focusing on semisimplicity and descriptions of the indecomposable objects, we prove uniform generalisations of results known for special cases, including $\uRep(S_t)$ or Temperley--Lieb categories. In particular, we identify those values of the interpolation parameter which correspond to semisimple and non-semisimple categories, respectively, for all so-called group-theoretical easy quantum groups. A crucial ingredient is an abstract analysis of certain subobject lattices developed by Knop, which we adapt to categories of partitions. We go on to prove a parametrisation of the indecomposable objects in all interpolating partition categories for non-zero interpolation parameters via a system of finite groups which we associate to any partition category, and which we also use to describe the associated graded rings of the Grothendieck rings of these interpolation categories. 
\end{abstract}

\maketitle

\addtocontents{toc}{\protect\setcounter{tocdepth}{1}}
\tableofcontents

\section{Introduction}
In this article, we link the representation theory of easy quantum groups with interpolating categories of the kind studied by Deligne. This provides many new examples for the latter theory. Each of these examples is a subcategory of one of Deligne's categories $\uRep(S_t)$ with the same objects, but restricted morphism spaces. We will start reviewing some background material on (easy) quantum groups in order to put our results into context.

\medskip

There are various settings in which the term quantum group is used. Originally, quantum groups were introduced by Drinfeld \cite{Di87} and Jimbo \cite{Ji85} as Hopf algebra deformations of the universal enveloping algebras of semisimple Lie algebras. In this article we consider topological quantum groups in the sense of Woronowicz \cite{Wo87}. A compact matrix quantum group is a deformation of the algebra of continuous complex-valued functions on a compact matrix group. In such a non-commutative setting, Woronowicz proved a Tannaka--Krein type result \cite{Wo88} showing that any compact matrix quantum group can be fully recovered from its representation category. 

This was the starting point for Banica and Speicher \cite{BS09} to introduce (orthogonal) easy quantum groups. These form a subclass of compact matrix quantum groups, which can be built up from sets of partitions closed under certain operations, called categories of partitions (we should warn the reader, however, that despite their name, the latter are not thought of as categories themselves, but as purely combinatorial structures; see \Cref{ssec::categories-of-partitions}). For any category of partitions $\cC$ and any non-negative integer $n$, Banica and Speicher defined a monoidal category whose composition is given by a relatively simple graphical calculus. In the present article, we denote these categories $\uRep(\cC,n)$, $n\in\NN_0$. An easy quantum group is then a compact matrix quantum group whose representation category is the image of some category $\uRep(\cC,n)$ under a certain fiber functor. An example of an easy quantum group is the $n$-th symmetric group $S_n$ induced by the category of all partitions $\cC=P$. An honest quantum group example, where the underlying algebra is non-commutative, is Wang's \cite{Wa98} free symmetric quantum group $S_n^+$ induced by the category of all non-crossing partitions. In 2016, Raum and Weber \cite{RW16} completed the classification of all categories of partitions, and we will use this classification throughout the paper. 

\medskip

In \cite{De07}, Deligne introduced and studied categories $\uRep(S_t)$ interpolating the representation categories of all symmetric groups. Deligne's categories depend on a complex interpolation parameter $t$, they are always Karoubian (pseudo-abelian) and monoidal. However, for $t\not\in\NN_0$, they turn out to be semisimple, while for $t\in\NN_0$, they are not. Instead, there is a unique semisimple quotient category, the semisimplification in the sense of Barrett--Westbury (\cite{BW99}, see also \cite{EO18}). Its defining tensor ideal is formed by all negligible morphisms, that is, morphisms whose compositions with other morphisms have trace $0$ whenever they are endomorphisms. The semisimplification of $\uRep(S_t)$ in the case $t=n\in\NN_0$ is equivalent to $\Rep(S_t)$, the ordinary category of representations of the $n$-th symmetric group, whose finitely many irreducible objects have a well-known parametrisation by a finite set of Young diagrams, depending on $n$. This description extends to a parametrisation of the indecomposable objects in $\uRep(S_t)$ by Young diagrams of arbitrary size, independent of $t$ (see \cite[Thm.~3.7]{CO11}).

An intriguing feature of Deligne's categories is their combinatorial definition via set partitions, which looks very much like the calculus used for easy quantum groups. In fact, we have $\uRep(S_n)=\uRep(P,n)$ for $n\in \NN_0$, and more generally, the categories $\Rep(\cC,n)$ can always be regarded as subcategories of $\uRep(S_n)$. Hence, it is natural to consider interpolation categories $\uRep(\cC,t)$ such that $\uRep(S_t)$ is recovered as a special case for $\cC=P$. The definition of such interpolation categories can be found in \cite{Fr17}. However, they have never been studied systematically within the framework of Deligne's interpolating categories, and we intend to initiate such an endeavour.

\medskip

In particular, we want to study the semisimplicity and the indecomposable objects in such interpolating partition categories. The table in \Cref{fig:table-intro} summarises some known results about special cases, together with some results obtained in this paper which are new to our knowledge (more examples are considered in the last section, \Cref{sec-examples}).

\begin{figure}[ht]
    \centering
    \caption{Special cases of interpolating partition categories. Indecomposable objects are computed for seven more interpolating partition categories in \Cref{sec-examples}.}
    \label{fig:table-intro}
\begin{longtable}{|C{2.4cm}||C{1.5cm}|C{3cm}|C{2.5cm}|C{3.1cm}|}
\hline
$\mathcal{C}$ & $\uRep(\mathcal{C},t)$ & Non-semisimple & Indecomposable objects up to isomorphism & Reference \tabularnewline
\hhline{=====}
$P=$ \\all partitions & $\uRep(S_t)$ & $t\in \NN_0$ & Young diagrams of arbitrary size & \cite[Thm.~3.7]{CO11}\tabularnewline
\hline
$P_2=$ \\partitions with block size two & $\uRep(O_t)$ & $t\in \ZZ$ & Young diagrams of arbitrary size & \cite[Thm.~3.5]{CH17} \tabularnewline
\hline
$P_\even=$ partitions with even block size & $\uRep(H_t)$ & $t\in \NN_0$ & bipartitions of arbitrary size & Thm.~\ref{thm-grouptheo-semisimple}, Prop.~\ref{thm::indecomp_obj_Hn} \tabularnewline
\hline
$NC=$ non-crossing partitions & $\uRep(S_t^+)$ & $t=2\cdot \text{cos}(j\pi/{l} )$, \\ $l\in \NN_{\geq 2}, j\in \NN_{\leq l-1}$ & modified Jones--Wenzl idempotents & Lem.~\ref{lem::St+_semisimple}, Lem.~\ref{lem::indecomposables_St+} \tabularnewline
\hline
$NC_2=$ non-crossing partitions with block size two & $\uRep(O_t^+)$ & $t=4\cdot \text{cos}(j\pi/{l})^2$, \\ $l\in \NN_{\geq 2}, j\in \NN_{\leq l-1}$& Jones--Wenzl idempotents & \cite[Cor.~3.2,~Thm.~3.3]{GW02}, \\ \cite[Thm.~4.0.8]{Ch14} \tabularnewline
\hline
$NC_\even=$ non-crossing partitions with even block size & $\uRep(H_t^+)$ & $t=4\cdot \text{cos}(j\pi/{l})^2$, \\ $l\in \NN_{\geq 2}, j\in \NN_{\leq l-1}$& finite binary sequences of arbitrary length & Lem.~\ref{lem::Ht+_semisimple}, Prop.~\ref{thm::indecomp_obj_Hn} \tabularnewline
\hline
\end{longtable}
\end{figure}

More systematically, it turns out that many results on the semisimplicity and indecomposable objects can be derived for general interpolating partition categories $\RepCt$. We find that, as semisimplicity can be encoded in polynomial conditions, such categories will be semisimple for generic values of the deformation parameter $t$, that is, for all values outside a set of algebraic complex numbers depending on $\cC$. We recall these special values for $t$ for several known special cases and use the concept of a positive $*$-operation on a category to provide a criterion for recognising semisimplification functors, before proving a general result for so-called \emph{group-theoretical} categories of partitions, an uncountable family covering all but countably many cases of categories of partitions in the classification of \cite{RW16}. These categories are closed under a certain coarsening operation for partitions, see \Cref{ssec::categories-of-partitions} and \Cref{lem::common_coarsening}. 

\begin{theorem}[\Cref{thm-grouptheo-semisimple}] \label{thm::main_thm_1}
Let $\cC$ be a any group-theoretical category of partitions. Then $\uRep(\cC,t)$ is semisimple if and only if $t\not\in\NN_0$.
\end{theorem}

In particular, this recovers and generalises known results for $\uRep(S_t)$ as well as for the interpolation categories for the hyperoctahederal groups, $\uRep(H_t)$. To prove this general result, we observe that for group-theoretical categories of partitions, certain lattices of subobjects are, in fact, sublattices of the corresponding lattices of $\uRep(S_t)$. This enables us to apply techniques developed by Knop \cite{Kn07} originally to study generalisations of $\uRep(S_t)$, which involve a concise analysis of the mentioned sublattices, and which we carry out for arbitrary categories of partitions.

\medskip

We go on to derive a general parametrisation scheme of the indecomposable objects in interpolating partition categories. Since we are working in the context of Karoubian categories, the study of indecomposables amounts to an analysis of primitive idempotents in endomorphism algebras, which in our case are the algebras spanned by partitions with a fixed number of upper and lower points. Using suitable filtrations on these endomorphism algebras, we show that their idempotents and, hence, the indecomposables in question are parametrised by the irreducible complex representations of certain finite-groups, which we associate to a distinguished set of so-called projective partitions (\Cref{def::projPart}), extending the work of Freslon and Weber \cite{FW16}. Hence, up to the representation theory of certain finite groups, all indecomposable objects can be found by determining the set of projective partitions in a given partition category. This yields a general description of the indecomposable objects for all categories of partitions. 

Given a category of partitions $\cC$, we construct a set of projective partitions $\cP$ with an equivalence relation $\sim$ (\Cref{def::P}, \Cref{def::equivprojpart}), certain finite groups $S(p)$ associated to each projective partition $p$ (\Cref{def::GroupsSp}), and an indecomposable object $\L_p(V)$ in $\RepCt$ for each irreducible $S(p)$-module $V$ (\Cref{def::LpV}).

\begin{theorem}[\Cref{thm::indecompsable_obj_by_A_k}] \label{thm::main_thm_2}
Let $\cC$ be a category of partitions and let $t$ be a non-zero complex number. Then the mapping $V\mapsto \L_p(V)$ induces a bijection between the collection of all isomorphism classes of irreducible complex representations of all finite groups $S(p)$ for $[p]\in\mathcal{P}/{\sim}$ and the isomorphism classes of indecomposable objects in $\uRep(\cC,t)$.
\end{theorem}

In particular, this is an analogue and, in fact, a generalisation of the parametrisation of the indecomposables by Young diagrams of arbitrary size for $\uRep(S_t)$ as explained above. From the knowledge of all indecomposables in $\RepCt$ we derive a description of the associated graded ring of the Grothendieck ring, using a suitable filtration, for all $\cC$ (\Cref{Grothendieck-ring}).


Beyond that, we apply our general results to obtain a concrete parametrisation of the indecomposable objects in $\RepCt$ for all categories of partitions $\cC$ which contain the partition $\cross$, and also for categories $\cC$ of non-crossing partitions. Moreover, we show that our results generalise the known description of indecomposables by Jones--Wenzl idempotents for the Temperley--Lieb categories $\uRep(O^+_t)$ (\Cref{prop::Indec_Otp}), which we relate to the interpolation categories for non-crossing partitions $\uRep(S^+_t)$ by constructing a suitable monoidal equivalence (\Cref{lem::equivalence_St+_Ot+}, \Cref{lem::indecomposables_St+}).


It will be interesting to convert the general result of \Cref{thm::main_thm_2} to concrete parametrisations for more families of partition categories.
Beyond that, it seems intriguing to study semisimplicity and indecomposable objects in interpolation categories of unitary easy quantum groups (see \cite{TW17}), corresponding to a calculus of two-colored partitions, or of linear categories of partitions (see \cite{GW19}), whose generators are not necessarily partitions, but more generally, linear combinations thereof. Eventually, such an analysis can be undertaken for the generalisations of partition categories described in \cite{MR19}, whose morphisms involve finite graphs.

\medskip

\textbf{Structure of this paper.} In Section 2, we recall the definition and classification of categories of partitions and introduce the interpolating categories $\uRep (\cC,t)$. In Section 3, we provide some general results on the semisimplicity of these categories and recall explicit computations for several known special cases. Moreover, we determine all parameters $t$ for which $\uRep(\cC,t)$ is semisimple in the case that $\cC$ is group-theoretical. We start Section 4 with some general results on indecomposable objects in $\uRep (\cC,t)$ before deriving an explicit description of the indecomposables using projective partitions, as well as results on its Grothendieck ring and semisimplifications coming from interpolating partition categories. In Section 5 we apply our general scheme to various special cases, including to $\uRep(H_t)$ and $\uRep(H_t^+)$, and to the well-studied example of Temperley--Lieb categories.

\medskip

\newcommand\ouracknowledgments{Both authors thank Gerhard Hi\ss{} for his encouragement to link the theory of Deligne's interpolating categories to easy quantum groups. We are also very grateful to Moritz Weber for many useful discussions and his constant support, to Pavel Etingof and Thorsten Heidersdorf for invaluable feedback on an earlier version of this paper, in particular, for pointing out to us inaccuracies in that version, and to Amaury Freslon for feedback on an earlier version of this paper and for the fruitful idea to consider projective partitions. We thank the referees for their extemely useful feedback. 
The research of both authors was funded by the Deutsche Forschungsgemeinschaft (DFG, German Research Foundation) -- Project-ID 286237555 -- TRR 195, and this article contributes to the project ``I.13. -- Computational classification of orthogonal quantum groups''. The second author is also supported by the Scholarship for Doctoral Students of the RWTH Aachen University. This article is part of the second author's PhD thesis supervised by Gerhard Hi\ss{} and Moritz Weber.} 
\textbf{Acknowledgements.} \ouracknowledgments 

\section{Interpolating partition categories}
In this section, we introduce interpolating partition categories. To this end, we start by recalling the theory of categories of partitions, including their classification. At the end of the section, we explain how interpolating partition categories interpolate the representation categories of the corresponding easy quantum groups.

\subsection{Categories of partitions}  \label{ssec::categories-of-partitions}
For the following definitions and examples, we refer to the initial article \cite{BS09}. For any $k,l\in \NN_0$ we denote by $P(k,l)$ the set of partitions of $\{ 1,\ldots ,k,1',\ldots ,l' \}$ into disjoint, non-empty subsets. These subsets are called the \emph{blocks} of a partition $p\in P(k,l)$ and we denote their number by $\#p$. We can picture every partition $p\in P(k,l)$ as a diagram with $k$ upper and $l$ lower points arranged on two parallel straight lines, where all points in the same block of $p$ are connected by strings which are entirely contained in the region between the two lines. For instance, the partition $\{ \{1\},\{1'\},\{2,4,2',4'\},\{3,3'\},\{5,6\},\{5',7'\},\{6'\} \} \in P(6,7)$ can be pictured as follows:
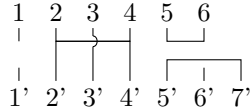
\begin{align*}
\BigPartition{
\Pblock 0 to 0.75:4,8
\Pblock 0 to 0.4:10,14
\Psingletons 0 to 0.25:2,12
\Psingletons 0 to 0.65:6
\Pblock 1 to 0.75:4,8
\Pblock 1 to 0.75:10,12
\Psingletons 1 to 0.75:2
\Psingletons 1 to 0.85:6
\Pcurveto (6,0.85)(6.3,0.85)(6.3,0.65)(6,0.65)
\Ptext(2,1.3){1}
\Ptext(4,1.3){2}
\Ptext(6,1.3){3}
\Ptext(8,1.3){4}
\Ptext(10,1.3){5}
\Ptext(12,1.3){6}
\Ptext(2,-0.3){1'}
\Ptext(4,-0.3){2'}
\Ptext(6,-0.3){3'}
\Ptext(8,-0.3){4'}
\Ptext(10,-0.3){5'}
\Ptext(12,-0.3){6'}
\Ptext(14,-0.3){7'}}
\end{align*}


Note that only the connected components of a diagram representing a partition are unique, not the diagram itself. However, all arguments and constructions involving partitions and diagrams representing them will be independent of the choice of a diagram. In the following we will repeatedly consider the following special partitions: 
\begin{alignat*}{4}
    &\singleton &&=\{\{1'\}\} \in P(0,1),                 && \cross &&=\{\{1,2'\},\{2,1'\}\} \in P(2,2), \\
    &\Paa &&=\{\{1,1'\}\} \in P(1,1),                     &&\legpart &&=\{\{1,2'\},\{2\},\{1'\}\}\in P(2,2),\\
    &\bbar &&=\{\{1\},\{1'\}\}\in P(1,1),                 && \vierpart &&=\{\{1,2,1',2'\}\}\in P(2,2),\\
    &\Uaa &&=\{\{1,2\}\}\in P(2,0),                       &&\twoblocks &&=\{\{1,2\},\{1',2'\}\}\in P(2,2),\\
    &\Laa &&=\{\{1',2'\}\}\in P(0,2), \qquad \qquad       &&\twoblockspart &&=\{\{1,3'\},\{2,3\},\{1',2'\}\}\in P(3,3).
\end{alignat*}

    A \emph{category of partitions} $\cC$ (\cite[Def.~2.2]{BS09}) is a collection of subsets $\cC(k,l) \subseteq P(k,l),k,l\in \NN_0$, containing the partitions $\Uaa \in P(2,0)$ and $\Paa \in P(1,1)$, which is closed under the following operations:
\begin{enumerate}[label=$\bullet$]
\item The \emph{tensor product} $p\otimes q \in P(k+k',l+l')$ is the horizontal concatenation of two partitions $p\in P(k,l)$ and $q\in P(k',l')$.
\item The \emph{involution} $p^* \in P(l,k)$ is obtained by turning a partition $p\in P(k,l)$ upside-down.
\item Let $p\in P(k,l)$ and $q\in P(l,m)$. Then we can consider the vertical concatenation of the partitions $p$ and $q$. We may obtain connected components, called \emph{loops}, which are neither connected to upper nor to lower points. We denote their number by $\ell(q,p)$. The \emph{composition} $q\cdot p \in P(k,m)$ of $p$ and $q$ is the vertical concatenation, where we remove all loops.
\end{enumerate}

\begin{remark} 
Note that a category of partitions is despite its name not a category, but rather a combinatorial datum which eventually will be associated to a category as customary in the literature on compact matrix quantum groups. We will explain the correspondence between combinatorial datum and actual category in \Cref{ssec::interpol_part_cat} (see \Cref{rem::identify-category}).
\end{remark}

\begin{example} $\Laa \ot \Paa \ot \Uaa = \twoblockspart$, $(\Laa)^*=\Uaa$, $\Paa\cdot\Paa = \Paa$, $\Laa\cdot\Uaa=\twoblocks$, $\vierpart\cdot \vierpart=\vierpart$, $\twoblocks\cdot \twoblocks=\twoblocks$.
\end{example}

For any subset $E\subseteq P=\bigsqcup_{k,l} P(k,l)$ we denote by $\langle E \rangle$ the category of partitions generated by $E$, which is obtained by taking the closure of $E \cup \{ \Laa, \Paa \}$ under tensor products, involution and composition.

\begin{example}
We will study the following examples throughout the paper, see for instance \cite[Thm.~2.6,~Thm.~3.14]{BS09} and \cite[Sec.~2]{We13}.
\begin{itemize}[label=$\star$]
    \item The category of all partitions $P$ is obviously a category of partitions and we have $P=\langle \cross, \singleton, \vierpart \rangle$.
    \item The category of partitions $P_\even:=\langle \cross, \vierpart \rangle$ consists of the partitions which have only blocks of even size.
    \item The category of partitions $P_2:=\langle \cross \rangle$ consists of those partitions which have only blocks of size two.
    \item The category of partitions $NC:=\langle \singleton, \vierpart \rangle$ consists of all non-crossing partitions, i.e.~partitions with at least one representing diagram where no strings cross each other. Note that this can also be characterised in terms of the set partitions only, without considering representing diagrams. 
    \item The category of partitions $NC_\even:=\langle \vierpart \rangle$ consists of the non-crossing partitions which have only blocks of even size.
    \item The category of partitions $NC_2$ consists of those non-crossing partitions which have only blocks of size two; it is the minimal category of partitions in the sense that it is generated by $\emptyset\subset P$.
\end{itemize}
\end{example}

In 2016, Raum and Weber \cite[Thm.~4.2]{RW16} classified all categories of partitions and we briefly summarise their results. All categories of partitions fall into one of the following cases:
\begin{itemize}
    \item The categories of partitions $\cC$ with $\cross \in \cC$ are exactly 
          \[ P, P_\even, P_2, \langle \cross, \singleton\ot \singleton, \vierpart \rangle, \langle \cross, \singleton \rangle, \langle \cross, \singleton\ot \singleton \rangle,\]
          see \cite[Prop.~2.3, Thm.~2.8]{BS09}.
    \item The categories of partitions $\cC$ which contain only non-crossing partitions are exactly 
          \[ NC, NC_\even, NC_2, \langle \singleton\ot \singleton, \vierpart \rangle, \langle \singleton\ot \singleton \rangle, \langle \legpart \rangle, \langle \singleton \rangle,\]
          see \cite[Thm.~3.13,Thm.~3.16]{BS09} and \cite[Thm.~2.9]{We13}. Note that $\langle \cross, \singleton\ot \singleton \rangle = \langle \cross, \legpart \rangle$, which explains why the number of non-crossing partition categories is one more than the number of partition categories containing $\cross$.
    \item The categories of partitions $\cC$ with $\cross \notin \cC$ and $\starpart \in \cC$ are exactly 
          \[ \langle \starpart \rangle, \langle \starpart, \singleton \ot \singleton \rangle, \langle \starpart, \vierpart \rangle,  \langle \starpart, \vierpart,h_s \rangle,s\in \NN,\]
          where $\starpart$ denotes the partition $\{\{1,3'\},\{2,2'\},\{3,1'\}\}\in P(3,3)$ and $h_s$ denotes the partition $\{\{1,3,5,\ldots,2s-1\},\{2,4,6,\ldots,2s\}\}\in P(2s,0)$, see \cite[Thm.~3.14]{We13}. This are the so called \emph{half-liberated categories}.
    \item The categories of partitions containing 
          \[ \primarypart = \{\{1,2,2',3'\},\{3,1'\}\} \in P(3,3) \]
          are called \emph{group-theoretical}. The notion ``group-theoretical'' was introduced in \cite{RW15} and refers to the fact that the partitions of a group-theoretical category of partitions can be set in correspondence to the words of a normal subgroup of a free product of copies of $\ZZ_2$ which is invariant under a certain semigroup action. Raum and Weber also showed that there are uncountably many such group-theoretical categories of partitions, see \cite[Thm.~5.6]{RW15}.
    \item The categories of partitions $\cC$ with $\vierpart \in \cC$, $\singleton \ot \singleton \notin \cC$ and $\primarypart \notin \cC$ are exactly those generated by the element
          \begin{center} \begin{tikzpicture}
            \coordinate [label=left:{$\pi_k=$}](O) at (0,0.5);
            \coordinate [label=above:{$\ldots$}](A2) at (0.75,0);
            \coordinate [label=above:{$\ldots$}](A7) at (2.55,0);
            \coordinate [label=above:{$\ldots$}](A7) at (4.35,0);
            \coordinate [label=above:{$\ldots$}](A7) at (6.15,0);
            \coordinate [label=below:{$1'$}](A1) at (0,0);
            \coordinate (A2) at (0.3,0);
            \coordinate (A3) at (1.2,0);
            \coordinate [label=below:{$k'$}](A4) at (1.5,0);
            \coordinate (A5) at (1.8,0);
            \coordinate (A6) at (2.1,0);
            \coordinate (A7) at (3,0);
            \coordinate [label=below:{$2k'$}](A8) at (3.3,0);
            \coordinate (A9) at (3.6,0);
            \coordinate (A10) at (3.9,0);
            \coordinate (A11) at (4.8,0);
            \coordinate [label=below:{$3k'$}](A12) at (5.1,0);
            \coordinate (A13) at (5.4,0);
            \coordinate (A14) at (5.7,0);
            \coordinate (A15) at (6.6,0);
            \coordinate [label=below:{$4k'$}](A16) at (6.9,0);
            
            \draw (A1) -- (0,1.2) -- (6.9,1.2) -- (A16);
            \draw (A8) -- (3.3,1.2);
            \draw (A9) -- (3.6,1.2);
            \draw (A2) -- (0.3,0.9) -- (3,0.9) -- (A7);
            \draw (A10) -- (3.9,0.9) -- (6.6,0.9) -- (A15);
            \draw (3,0.9) to [bend left] (3.9,0.9);
            \draw (A3) -- (1.2,0.6) -- (2.1,0.6) -- (A6);
            \draw (A11) -- (4.8,0.6) -- (5.7,0.6) -- (A14);
            \draw (2.1,0.6) -- (2.7,0.6) to [bend left] (4.2,0.6) -- (4.8,0.6);
            \draw (A4) -- (1.5,0.3) -- (1.8,0.3) -- (A5);
            \draw (A12) -- (5.1,0.3) -- (5.4,0.3) -- (A13);
            \draw (1.8,0.3) -- (1.9,0.3) to [bend left] (2.3,0.3) -- (2.7,0.3) to [bend left] (4.2,0.3) -- (4.6,0.3) to [bend left] (5,0.3) -- (5.1,0.3);
          \end{tikzpicture}\end{center}
          for some $k\in \NN$ and $\langle \pi_k \mid k\in \NN \rangle$, see \cite[Thm.~3.9, Thm.~4.1]{RW16}. 
\end{itemize}
These cases are pairwise distinct except that $\langle \pi_1 \rangle = \langle \vierpart \rangle = NC_\even$ and the categories 
\[P,P_\even,\langle \cross, \singleton\ot \singleton, \vierpart \rangle,\langle \starpart, \vierpart \rangle,  \langle \starpart, \vierpart,h_s \rangle,s\in \NN,\] 
are also group-theoretical.

\begin{remark} \label{rem::uparrow_in_C}
Note that the only categories of partitions $\cC$ with $\singleton \in \cC$ are 
\[ P, NC, \langle \cross, \singleton \rangle, \langle \singleton \rangle. \]
\end{remark}

\begin{proof}
It follows from the classification that any category of partitions $\cC$ which is not one of these four is generated by partitions whose sum of upper and lower points is even. It follows that the sum of upper and lower points is even for any partition in $\cC$ and hence $\singleton \notin \cC$.  
\end{proof}

\subsection{Interpolating partition categories} \label{ssec::interpol_part_cat}
We refer for instance to \cite{EGNO15} and \cite{NT13} for the terminology in this subsection. The following natural definition may be deduced from Banica--Speicher's definition of easy quantum groups in \cite{BS09}. It may also be found in \cite[Def.~6.1.2]{Fr17}. Recall that $\ell(q,p)$ is the number of loops in the diagram obtained by stacking $p$ and $q$ for all $p\in P(k,l), q\in P(l,m)$. Let us denote the non-negative integer numbers by $\NN_0$.
\begin{definition}[Interpolating partition categories]
For any category of partitions $\cC$ and $t\in \CC$ the category $\uRep_0(\cC,t)$ has:
\begin{alignat*}{2}
&\text{Objects:} &&[k], k\in \NN_0 ,\\
&\text{Morphisms:} &&\Homm([k],[l])= \CC \cC(k,l),\\
&\text{Composition:} \quad &&\Homm([l],[m]) \times \Homm([k],[l]) \to \Homm([k],[m]), \\
& &&(q,p)\mapsto qp := t^{\ell(q,p)}~ q\cdot p \text{ for all } p\in \cC(k,l), q\in \cC(l,m)
\end{alignat*}
The \emph{interpolating partition category} $\uRep(\cC,t)$ is the Karoubi envelope or (pseudo-abelian completion) of $\uRep_0(\cC,t)$, that is, the idempotent completion of the additive completion.
\end{definition}

\begin{example}\label{ex::uRep}
By definition, $\uRep(P_2,t)=\uRep(O_t)$, the category interpolating the representation categories of the orthogonal groups $\Rep(O_n)$ introduced by Deligne in 1990 \cite{De90} and $\uRep(P,t)= \uRep(S_t)$, the category interpolating the representation categories of the symmetric groups $\Rep(S_n)$ introduced by Deligne in 2007 \cite{De07}. 
\end{example}

The tensor product of partitions turns $\uRep(\cC,t)$ into a (strict) monoidal category with unit object $\mathbf{1}=[0]$. Moreover, we can define duals in $\uRep(\cC,t)$ as follows. Any object is self-dual, i.e.~for any $k\in \NN_0$ the dual object of $[k]$ is given by $[k]^{\vee} := [k]$, and the (co)evaluation maps are\\
\ \\
\begin{minipage}[t]{0.45\textwidth}\begin{center}\begin{tikzpicture}
\coordinate [label=right:{ev$_{k}:[k]^{\vee} \ot [k]\to \mathbf{1}$ given by}](O) at (-1.2,1.5);
\coordinate [label=right:{ev$_k=$}](O) at (-1.2,0.6);
\coordinate [label=right:{$\in P(2k,0)$,}](O) at (2.7,0.6);
\coordinate [label=below:{$\ldots$}](O) at (0.45,1);
\coordinate [label=below:{$\ldots$}](O) at (2.25,1);
\coordinate (A1) at (0,1);
\coordinate (A2) at (0.9,1);
\coordinate (A3) at (1.2,1);
\coordinate (A4) at (1.5,1);
\coordinate (A5) at (1.8,1);
\coordinate (A6) at (2.7,1);

\fill (A1) circle (2pt);
\fill (A2) circle (2pt);
\fill (A3) circle (2pt);
\fill (A4) circle (2pt);
\fill (A5) circle (2pt);
\fill (A6) circle (2pt);

\draw (A1) -- (0,0.25) -- (2.7,0.25) -- (A6);
\draw (A2) -- (0.9,0.5) -- (1.8,0.5) -- (A5);
\draw (A3) -- (1.2,0.75) -- (1.5,0.75) -- (A4);
\end{tikzpicture}\end{center}\end{minipage}
\begin{minipage}[t]{0.45\textwidth}\begin{center}\begin{tikzpicture}
\coordinate [label=right:{coev$_{k}:\mathbf{1}\to [k]^{\vee} \ot [k]$ given by}](O) at (-1.6,0.7);
\coordinate [label=right:{coev$_k=$}](O) at (-1.6,-0.2);
\coordinate [label=right:{$\in P(0,2k)$.}](O) at (2.7,-0.2);
\coordinate [label=above:{$\ldots$}](A2) at (0.45,-0.5);
\coordinate [label=above:{$\ldots$}](A7) at (2.25,-0.5);
\coordinate (A1) at (0,-0.5);
\coordinate (A2) at (0.9,-0.5);
\coordinate (A3) at (1.2,-0.5);
\coordinate (A4) at (1.5,-0.5);
\coordinate (A5) at (1.8,-0.5);
\coordinate (A6) at (2.7,-0.5);

\fill (A1) circle (2pt);
\fill (A2) circle (2pt);
\fill (A3) circle (2pt);
\fill (A4) circle (2pt);
\fill (A5) circle (2pt);
\fill (A6) circle (2pt);

\draw (A1) -- (0,0.25) -- (2.7,0.25) -- (A6);
\draw (A2) -- (0.9,0) -- (1.8,0) -- (A5);
\draw (A3) -- (1.2,-0.25) -- (1.5,-0.25) -- (A4);
\end{tikzpicture}\end{center}\end{minipage}

The categorical left and right trace, induced by the dual structure, coincide and are given by
\begin{equation} \label{eq::trace}
\scalebox{.8}{ \begin{tikzpicture}
\coordinate [label=left:{\scalebox{1.25}{$\tr(p)=$ ev$_{k} \circ (p\ot \id_{[k]})~ \circ$ coev$_{k}=$}}](O) at (0,0);
\coordinate [label=left:{\scalebox{1.25}{$p$}}](O) at (1,0);
\coordinate [label=right:{\scalebox{1.25}{$\in \End([0])\cong \CC$}}](O) at (3.5,0);
\coordinate (A1) at (0,0.5);
\coordinate (A2) at (1.5,0.5);
\coordinate (A3) at (0,-0.5);
\coordinate (A4) at (1.5,-0.5);

\fill (A1) circle (2.5pt);
\fill (A2) circle (2.5pt);
\fill (A3) circle (2.5pt);
\fill (A4) circle (2.5pt);

\draw[dashed] (A1) -- (A2) -- (A4) -- (A3) -- (A1);
\draw (2,0.5) -- (2,-0.5);
\draw[gray] (2.5,0.5) -- (2.5,-0.5);
\draw[gray] (3,0.5) -- (3,-0.5);
\draw (3.5,0.5) -- (3.5,-0.5);
\draw (A1) to [bend left=90] (3.5,0.5);
\draw (A2) to [bend left=90] (2,0.5);
\draw (A3) to [bend right=90] (3.5,-0.5);
\draw (A4) to [bend right=90] (2,-0.5);
\draw[gray] (0.5,0.5) to [bend left=90] (3,0.5);
\draw[gray] (1,0.5) to [bend left=90] (2.5,0.5);
\draw[gray] (0.5,-0.5) to [bend right=90] (3,-0.5);
\draw[gray] (1,-0.5) to [bend right=90] (2.5,-0.5);
\end{tikzpicture} }
\end{equation}
for any $p\in \cC(k,k)$. Hence $\uRep(\mathcal{C},t)$ is a pivotal category with coinciding left and right traces. Note that we defined the evaluation and coevaluation maps slightly differently than Deligne, insofar as the $i$-th point is paired with the $(2k+1-i)$-th point, not with the $k+i$-th point in the above diagrams. 

Furthermore, the category $\uRep(\mathcal{C},t)$ is a so-called \emph{$*$-category}. This is a $\CC$-linear monoidal category with a contravariant involutive antilinear monoidal endofunctor $*$ which is the identity on objects, see for instance \cite[Sec.~2.1]{Mue}. The endofunctor $*$ is called the \emph{$*$-operation} and for $\uRep(\mathcal{C},t)$ it is given by the involution (i.e.~horizontal reflection) of partitions. 

\begin{remark} \label{rem::identify-category}
Note that we can naturally identify any category of partitions $\cC$ with the monoidal $^*$-category $\uRep_0(\cC,1)$, which explains the naming. We will also view $\cC$ as a subset of morphisms in $\RepCt$ for any $t$.
\end{remark}


\subsection{Interpolating partition categories and easy quantum groups}\label{subsec::easyQG}
Categories of partitions have initially been introduced by Banica and Speicher to define easy quantum groups. In this subsection, we recall their definition and explain how interpolating partition categories interpolate the representation categories of the corresponding easy quantum groups. For the rest of this article, however, we will only work with the interpolating partition categories themselves and no knowledge of easy quantum groups is required.

Let us start by briefly recalling the theory of compact matrix quantum groups. A \emph{compact matrix quantum group} is a triple $G=(A,u,n)$ of a C*-algebra $A$, an integer $n\in \NN_0$ and a matrix $u\in A^{n\times n}$ such that the elements $\{ u_{ij} \mid 1\leq i,j \leq n \}$ generate $A$, the matrix $u=(u_{ij})$ is unitary and its transpose is invertible and the map $\Delta :A \to A \ot A,~u_{ij}\mapsto \sum_{k=1}^n u_{ik} \ot u_{kj}$ is a *-homomorphism, see \cite{Wo87}. 
A \emph{representation of $G$ of dimension $m$} (for some $m\geq0$) is a matrix $v\in A^{m\times m}$  with $\Delta(v_{ij})=\sum_{k=1}^m v_{ik}\otimes v_{kj}$ (technically, this corresponds to a corepresentation of a Hopf *-algebra). A morphism between two representations $v\in A^{m\times m}$ and $v'\in A^{m'\times m'}$ is a linear map $T:\CC^m\to\CC^{m'}$ with $Tv = v'T$. In particular, the matrix $u\in A^{n\times n}$ is a representation of $G$, the so-called \emph{fundamental representation}. 

In 1988, Woronowicz proved a Tannaka--Krein type result \cite{Wo88} 
for compact matrix quantum groups showing that any compact matrix quantum group $G$ is uniquely determined by its representation category $\Rep(G)$, i.e.~the category of finite-dimensional, unitary rep\-resentation, (for more details see for instance \cite[Sec.~4]{We17}). In 2009, Banica and Speicher \cite[Def.~1.7]{BS09} defined for any category of partitions $\cC$ and $n\in \NN_0$ a fiber functor into the category of finite-dimensional Hilbert spaces
\begin{align*}
    &T: \uRep(\cC,n) \to \text{Hilb}_f \text{ with } \\
    &T([k])=(\CC^n)^{\ot k} \text{ for any } k\in \NN_0 \text{ and } \\
    &T_p:=T(p)\in \Homm((\CC^n)^{\ot k},(\CC^n)^{\ot l}) \text{ for any } p\in \cC(k,l).
\end{align*} 
A compact matrix quantum group $G=(A,u,n)$ is called \emph{(orthogonal) easy quantum group} if there exists a category of partitions $\cC$ such that $\Homm_{\Rep(G)} (u^{\ot k}, u^{\ot l})=\text{span}_{\CC} \{ T_p \mid p\in \cC(k,l)\}$. Tannaka--Krein duality implies that, for any category of partitions $\cC$ and $n\in \NN_0$, there exists an easy quantum group $(A,u,n)$ with $\Homm_{\Rep(G_n(\cC))} (u^{\ot k}, u^{\ot l})=\text{span}_{\CC} \{ T_p \mid p\in \cC(k,l)\}$. In general, this quantum group is only unique up to so-called similarity of compact matrix quantum groups (for more details see for instance \cite{RW15}). We denote any of these similar easy quantum groups by $G_n(\cC)$, and we note that by construction the category $\Rep(G_n(\cC))$ is unique.

\begin{example} \label{ex::std_ex}
The easy quantum group $G_n(P)$ is the triple $(C(S_n),u,n)$ where $C(S_n)$ is the set of complex-valued continuous functions over the symmetric group $S_n$ (regarded as a matrix group) and $u$ is the matrix of coordinate functions. Similarly, $G_n(P_\even)$ corresponds to the hyperoctahedral group $H_n=S_2 \wr S_n$ and  $G_n(P_2)$ corresponds to the orthogonal group $O_n$. This fits together with \Cref{ex::uRep} and based on that notation we denote $\uRep(H_t):=\uRep(P_\even,t)$.

The easy quantum groups $S_n^+=G_n(NC)$, $H_n^+=G_n(NC_\even)$ and $O_n^+=G_n(NC_2)$ are called free symmetric quantum groups, free hyperoctahedral quantum groups and free orthogonal quantum groups, respectively, and we denote $\uRep(S_t^+):=\uRep(NC,t)$, $\uRep(H_t^+):=\uRep(NC_\even,t)$ and $\uRep(O_t^+):=\uRep(NC_2,t)$.
\end{example}

For any category of partitions $\cC$ the canonical functor $\uRep(\cC,n) \to \Rep(G_n(\cC))$ is surjective on objects and morphisms (for $\cC=P$ compare with \cite[Prop.~3.19]{CO11}). In the following section we will discuss that $\Rep(G_n(\cC))$ is even equivalent to the unique semisimple quotient of $\uRep(\cC,n)$.

\begin{lemma} \label{lem::fiber-functor}
Let $\cC$ be a category of partitions, $n\in \NN_0$ and consider the easy quantum group $G_n(\cC)=(A,u,n)$. Then the functor $$\mathcal{G}:\uRep(\cC,n) \to \Rep(G_n(\cC)),~ [k]\mapsto u^{\ot k},~ p\mapsto \mathcal{F}(p)$$ is full and essentially surjective.
\end{lemma}

\begin{proof} The definition of easy quantum groups implies that $\mathcal{G}$ is a full functor. Since any irreducible representation of a compact matrix quantum group $G=(A,u,n)$ is contained in some tensor power of the fundamental representation $u$ (see for instance \cite{Timmermann}), $\mathcal{G}$ is also essentially surjective.
\end{proof}

\section{Semisimplicity for interpolating partition categories}
In this section we analyse the categories $\RepCt$ with respect to semisimplicity. We will recall the general theory of semisimplicity, give a criterion for recognising semisimplification functors for Karoubian tensor categories, and consider the categories from \Cref{ex::std_ex} on a case-by-case basis, before following a generic approach due to Knop to analyse $\RepCt$ for all group-theoretical categories of partitions $\cC$. To prove semisimplicity we use a reduction argument which shows that it suffices to check whether certain Gram determinants vanish.

\subsection{Semisimplicity for Karoubian tensor categories} \label{semisimplicity-karoubian} 

By construction, the category $\uRep(\cC,t)$ is Karoubian (i.e., pseudo-abelian), but in general, it is not abelian. However, we can construct a unique semisimple (and hence, abelian) quotient category from it, the \emph{semisimplification} $\widehat{\uRep(\cC,t)}$. Let us recall some definitions and general results on this idea, for more details see \cite{EO18} or \cite{AK}. 
We call a pivotal category over a field $k$  \emph{spherical} if its left and right traces coincide (but note that spherical shall not imply a category is abelian in our context). The \emph{dimension} of an object in such a category is defined as the trace of its identity morphism.

\begin{definition} Let $\mathcal{R}$ be a spherical category over a field $k$. A morphism $f:X \to Y$ in $\mathcal{R}$ is called \emph{negligible} if $\tr(f\circ g)=0$ for all morphisms $g:Y\to X$ in $\mathcal{R}$. We denote by $\mathcal{N}$ the set of all negligible morphisms in $\mathcal{R}$.
\end{definition}

\begin{remark}
 The set of all negligible morphisms $\mathcal{N}$ is a tensor ideal and the quotient category $\mathcal{R}/\mathcal{N}$ is again a spherical category with $\tr (f + \mathcal{N}) = \tr (f)$ for any endomorphism $f$ in $\mathcal{R}$.
\end{remark}

It is well-known that there are no non-zero negligible morphisms in a semisimple category (see for instance \cite[Cor.~7.1.7]{AK}).

\begin{lemma}[{\cite[Thm.~2.6]{EO18}}] \label{rem::negl_morphisms}
Let $k$ be an algebraically closed field. Let $\mathcal{R}$ be a spherical Karoubian tensor category over $k$ such that all morphism spaces are finite-dimensional and the trace of any nilpotent endomorphism is zero.
Then the quotient category 
\[ \widehat{\mathcal{R}} := \QR{\mathcal{R}}{\mathcal{N}} \]
is a semisimple category, the \emph{semisimplification of $\mathcal{R}$}, whose simple objects correspond to the indecomposable objects of $\mathcal{R}$ of non-zero dimension.
\end{lemma}

To use this result for interpolation categories $\RepCt$, we observe: 

\begin{lemma} \label{lem-trace-nilpotent}
For any category of partitions $\cC$ and $t\in\CC$, the trace of any nilpotent endomorphism in $\uRep(\cC,t)$ is zero.
\end{lemma}

\begin{proof} Let $f$ be a nilpotent endomorphism in $\uRep(\cC,t)$. Then $f$ is also a nilpotent endomorphism in $\uRep(P,t)$. By \cite[Thm.~3.24, Cor.~5.23]{CO11}, $\widehat{\uRep(P,t)}=\uRep(P,t)/\mathcal{N}$ is a semisimple category. Since the trace of any nilpotent endomorphism in a semisimple category is zero, we have $\tr_{\uRep(\cC,t)}(f)=\tr_{\uRep(P,t)}(f)=\tr_{\widehat{\uRep(P,t)}} (f + \mathcal{N})=0$.
\end{proof}

Combining the previous two lemmas, we obtain:
\begin{lemma} \label{cor-criterion}
Let $\cC$ be a category of partitions and $t\in\CC$. The category $\uRep(\cC,t)$ is semisimple if and only if all negligible morphisms are trivial.
\end{lemma}

\begin{proof} \Cref{rem::negl_morphisms} and \Cref{lem-trace-nilpotent} imply the ``if'' part. The ``only if''-part follows from \cite[Cor.~7.1.7]{AK}.
\end{proof}

This abstract argument can be made practical by realising that the existence of negligible endomorphisms is detected by the determinants of certain Gram matrices.

\begin{definition}[{\cite[Def.~4.2]{BC07}}] For any category of partitions $\cC$, we introduce the short-hand notation $\cC(k)=\cC(0,k)$, denoting the partitions in $\cC$ with no upper points. The \emph{Gram matrices} are given by
$$
 G^{(k)} := (t^{\ell(p^*,q)})_{p,q\in\cC(k)} 
 \quad\text{for all }k\in \NN_0.
$$
\end{definition}

Notice that the entries of the Gram matrix are just the traces of the compositions $p^* q$. 

\begin{example} The following table features the entries of the Gram matrix $G^{(2)}$ for $\uRep(S_t)$:
\[
\begin{tabular}{c|cc}
 & $\Laa$ & $\singleton\singleton$ \\
 \hline
 $\Laa$  & $t$ & $t$  \\
 $\singleton\singleton$ & $t$ & $t^2$ 
\end{tabular}
\]
Its determinant is $t^2(t-1)$.

Note that the Gram matrices explained here differ from those computed in \cite[Ex.~3.14]{CO11}, which use the ``usual'' trace form in the finite-dimensional endomorphism algebras.
\end{example}

\begin{proposition} \label{lem::semisimple_endo} \label{lem::semisimple_determinant_general}
 Let $t\in \CC$ and let $\cC$ be a category of partitions. Then $\RepCt$ is semisimple if and only if it satisfies $\det(G^{(k)})\neq 0$ for all $k\in\NN$.
\end{proposition}

\begin{proof}
By \Cref{cor-criterion}, $\RepCt$ is semisimple if and only if it does not contain any non-trivial negligible morphisms. 
Now $\RepCt$ is constructed as a Karoubi envelope, that is, an idempotent completion of an additive completion, but we claim that negligibility can be traced back to the original category, $\uRep_0(\cC,t)$ in this case. First, as any negligible morphism from or to a direct summand extends trivially to a negligible morphism from or to the full object, it suffices to consider the additive completion. We can think of its morphisms as matrices whose entries are morphisms in the original category. As each entry of such a matrix can be recovered by compositions with suitable row and column matrices consisting of zero or identity morphisms, such a matrix is negligible, if and only if all of its entries are negligible. Hence, $\RepCt$ is semisimple if and only if there are no non-trivial negligible morphism $f\in \Homm_{\uRep_0(\cC,t)} ([k],[l])$ for all $k,l\in \NN_0$.

We observe that, for two morphisms $f\in\Homm_{\uRep_0(\cC,t)} ([k],[l])$ and $g\in\Homm_{\uRep_0(\cC,t)} ([l],[k)$, the trace of $f\circ g$ is given by the composition $f'\circ g'$, where $f'\in\Homm_{\uRep_0(\cC,t)} ([k+l],[0])$, $g'\in\Homm_{\uRep_0(\cC,t)} ([0],[k+l])$ are compositions of $f$ and $g$ with the categorical evaluation and coevaluation map, respectively. Hence, $\uRep_0(\cC,t)$ has non-zero negligible morphisms if and only if there is a non-zero negligible morphism $f\in \Homm ([0],[k])$ for some $k\in \NN_0$. In other words, $\RepCt$ is semisimple if and only if the form
$$ \Homm([0],[k]) \times \Homm([0],[k]) \to \CC, ~(p,q)\mapsto t^{\ell(q^*,p)} $$
is non-degenerate. The Gram matrix of this form is exactly $G^{(k)}$, and hence, the form is non-degenerate if and only if $G^{(k)}$ has a trivial kernel. Thus the claim follows (note that $\det G^{(0)}$ is $1$).
\end{proof}

\begin{corollary} For any category of partitions $\cC$ and any transcendental $t\in\CC$, $\RepCt$ is semisimple.
\end{corollary}

\begin{proof} The determinant of the Gram matrix $\det(G^{(k)})$ depends on $t$ polynomially for any $k\in \NN_0$.
\end{proof}

Let us contrast this with the case $t=0$.

\newcommand\cF{\mathcal{F}}
\newcommand\cD{\mathcal{D}}
\newcommand\cN{\mathcal{N}}
\newcommand\cQ{\mathcal{Q}}
\newcommand\cG{\mathcal{G}}

\begin{lemma}\label{lem::semisimplification-t0} For any category of partitions $\cC$, $\widehat{\uRep(\cC,0)}$ is equivalent to the category of finite-dimensional complex vector spaces.
\end{lemma}

\begin{proof} The morphism space $\Homm([k],[l])$ in $\uRep(\cC,0)$ consists of negligible morphisms if $k>0$ or $l>0$, while the non-zero endomorphism $\id_0$ of the object $[0]$ is not negligible.
\end{proof}

Deligne showed that $\uRep(S_t)$ is semisimple if and only if $t\notin \NN_0$, see \cite[Thm.~2.18]{De07}. We will show in \Cref{ssec:group-theoretical} that this is also the case for all group-theoretical categories of partitions, including $\uRep(H_t)$.

Let us discuss semisimplicity for some examples of non-crossing partition categories. We will explain a way to study their semisimplifications in \Cref{prop::semisimplification-nc}.

\begin{remark} \label{rem::TL_semisimple} 
The category $\uRep(O_t^+)$ is exactly the (Karoubian envelope of) the Temperley--Lieb category $TL(q)$ with $t=q+q^{-1}$ (introduced in \cite[Def. 2.1]{GL98}). It is well-known to be semisimple if and only if $q$ is not a $2l$-th root of unity, i.e. $q\notin \{ e^{\frac{j \pi}{l}} \mid l\in \NN_{\geq 2}, j\in \{1,\ldots,l-1\}\}$, see for instance \cite[Cor.~3.2,~Thm.~3.3]{GW02}.
This implies that the category $\uRep(O_t^+)$ is semisimple if and only if 
 \[ t\notin \{2\cdot \cos \left(\frac{j \pi}{l}\right) \mid l\in \NN_{\geq 2}, j\in \{1,\ldots,l-1\}\}. \]
\end{remark}

\begin{proposition} \label{lem::St+_semisimple}
 The category $\uRep(S_t^+)$ is semisimple if and only if 
 \[ t\notin \{4\cdot \cos \left(\frac{j \pi}{l} \right)^2 \mid l\in \NN_{\geq 2}, j\in \{1,\ldots,l-1\}\}. \]
\end{proposition}

\begin{proof} By \cite{Tu93} or \cite[Prop.~5.37]{Ju19}, the determinants described in \Cref{lem::semisimple_determinant_general} are non-zero if and only if $t$ is of the asserted form. This implies the assertion with \Cref{lem::semisimple_determinant_general}.
\end{proof}

\begin{proposition} \label{lem::Ht+_semisimple}
 The category $\uRep(H_t^+)$ is semisimple if and only if 
 \[ t\notin \{4\cdot \cos \left(\frac{j \pi}{l}\right)^2 \mid l\in \NN_{\geq 2}, j\in \{1,\ldots,l-1\}\}. \]
\end{proposition}

\begin{proof}
By \cite[Thm.~3.9.16]{Ah16}, the determinants described in \Cref{lem::semisimple_determinant_general} are non-zero if and only if $t$ is of the asserted form. This implies the assertion with \Cref{lem::semisimple_determinant_general}.
\end{proof}


\subsection{Semisimplifications}

To compute the semisimplification of a given interpolating partition category practically, it suffices to find a suitable monoidal functor to a semisimple category.

\begin{lemma} \label{lem::maximality-N} For any category of partitions $\cC$ and $t\in\CC$, the ideal of negligible morphisms $\cN$ is the unique maximal tensor ideal in $\RepCt$.
\end{lemma}

\begin{proof} Assume $f$ is a morphism in $\RepCt$ which does not lie in $\cN$. So we can find $g$, a morphism in $\RepCt$ such that $f\circ g$ is an endomorphism with non-zero trace. By the definition of the trace of $f\circ g$, it is given by a non-zero endomorphism of the object $[0]$ in $\RepCt$ which is obtained from $f$ by tensor products and compositions (see \Cref{eq::trace}). Hence, any tensor ideal in $\RepCt$ containing $f$ also contains the identity morphism of the object $[0]$. But as tensoring with this identity morphism leaves any morphism in $\RepCt$ invariant, any tensor ideal which does not lie in $\cN$ contains all morphisms in $\RepCt$.
\end{proof}

We recall that a \emph{$*$-operation} on a $\CC$-linear monoidal category is a contravariant involutive antilinear monoidal endofunctor $*$ which is the identity on objects. 
A $*$-operation is called \emph{positive} if $f^*\circ f=0$ implies $f=0$ for any morphism $f$. The relevance of this notion in our context lies in the following observation (variations of which are well-known; see \cite[Thm.~2.1]{Mue}, for instance).

\begin{lemma} \label{lem::Cstar-cat} Assume a $\CC$-linear Karoubian monoidal category has finite-dimensional morphism spaces and admits a positive $*$-operation. Then it is semisimple (and hence, abelian).
\end{lemma}

\begin{proof} As such a category is Karoubian with finite-dimensional $\operatorname{Hom}$-spaces, it has a Krull--Schmidt property (every object is the direct sum of finitely many indecomposable objects in an essentially unique way) and we only need to show that morphisms between indecomposable objects are isomorphisms or zero morphisms.

The endomorphism algebras of objects in our category are finite-dimensional complex algebras with a positive involution given by the endofunctor $*$. Hence, these algebras are semisimple complex $C^*$-algebras. In particular, the endomorphism algebras of indecomposable objects, which contain no non-trivial idempotents, are one-dimensional.

Consider a non-zero morphism $f$ between indecomposable objects. Then $f^*\circ f$ and $f\circ f^*$ are non-zero endomorphisms of indecomposable objects, hence, scalars. By simplifying $f\circ f^*\circ f$ in two different ways we see that the scalars have to coincide. So $f$ is an isomorphism. 
\end{proof}

Let us call categories as in \Cref{lem::Cstar-cat} \emph{hom-finite $C^*$-categories}, they include the category of finite-dimensional Hilbert spaces. An important feature of hom-finite $C^*$-categories is the fact that any $*$-subcategory of a hom-finite $C^*$-category is automatically semisimple. 

Let us note, however, that the $*$-operation on $\uRep(S_t)=\uRep(\cC,t)$ given by horizontal reflections is not positive. Hence, even if $t$ is chosen such that $\uRep(S_t)$ is semisimple (i.e.~$t\not\in\NN_0$), $\uRep(S_t)$ may contain $*$-subcategories which are not semisimple, as illustrated for example by the non-crossing partition categories we discussed earlier (see \Cref{lem::St+_semisimple}, for example).

A functor between two $*$-categories is called a \emph{$*$-functor} if it preserves the $*$-operation. 
\begin{proposition} \label{lem-semisimplification} For any category of partitions $\cC$ and any complex number $t$, assume $\cF\colon\RepCt$ $\to\cD$ is a (non-zero) monoidal $*$-functor, where $\cD$ is a hom-finite $C^*$-category. Then $\cF$ induces an equivalence between the semisimplification of $\RepCt$ and the image of $\cF$.
\end{proposition}

\begin{proof} The image of $\cF$ is a hom-finite $C^*$-(sub)category of $\cD$ and, in particular, semisimple. As a semisimple category cannot have negligible morphisms, the kernel of $\cF$ contains all negligible morphisms. Thus, it is exactly the ideal $\cN$ by \Cref{lem::maximality-N}.

Note that in $\RepCt$, all nilpotent morphisms have trace $0$ by \Cref{lem-trace-nilpotent}. Hence, the image of $\cF$ is indeed equivalent to the semisimplification of $\RepCt$.
\end{proof}

As a consequence, we see that for any category of partitions $\cC$, the semisimple quotient categories $\widehat{\uRep(\cC,t)}$, $t\in \CC$, interpolate the representation categories of the corresponding easy quantum groups $\Rep(G_n(\cC))$, $n\in \NN_0$, in the following sense (for $\cC=P$ compare with \cite[Thm.~6.2]{De07}, for $\cC=P_2$ compare with \cite[Thm.~9.6]{De07}):

\begin{proposition} \label{prop::fiber-functor} For any category of partitions $\cC$ and $t\in\NN_0$, the functor $\cG$ (see \Cref{subsec::easyQG}) induces an equivalence between the semisimplification of $\RepCt$ and $\Rep(G_t(\cC))$.
\end{proposition}

\begin{proof} This follows directly from \Cref{lem-semisimplification} with $\cD$ the $C^*$-category of finite-dimensional Hilbert spaces. 
\end{proof}

\subsection{Semisimplicity in the group-theoretical case} \label{ssec:group-theoretical}
In this section we show our first main theorem, namely that any category $\RepCt$ associated to a group-theoretical category of partitions $\cC$ is semisimple if and only if $t\notin \NN_0$. In 2007, Knop \cite{Kn07} studied tensor envelopes of regular categories and Deligne's category $\uRep (S_t)$ is a special case in his setting. Using the semilattice structure of subobjects, he gives a criterion for semisimplicity for most of the tensor categories he is considering, including $\uRep(S_t)$. We will mimic his proof by studying it in the special case of $\uRep(S_t)$, and then generalising it to all categories $\RepCt$ associated to group-theoretical categories of partitions.

The key observation which allows us to use Knop's idea is the following. If we consider Knop's work in the special case of $\uRep(S_t)$, the semilattice of subobjects of $[k]$ corresponds to the meet-semilattice on partitions of $k$ points given by the refinement order. It is well-known that the (reversed) refinement order induces a lattice structure on partitions on $k$ points or non-crossing partitions on $k$ points, see for instance \cite[Prop.~9.17, Rem.~9.19]{NS06}. In the following, we will use that group-theoretical categories of partitions are closed under common coarsening of partitions, the meet with respect to the refinement order, and hence we also obtain a semilattice structure.

Let us start by briefly recalling some basics on partially ordered sets and semilattices, see \cite[Sec.~9]{NS06} and \cite[Sec.~7]{Kn07}.

\begin{definition}[{\cite[Def.~9.15]{NS06}}]
Let $(L,\leq)$ be a finite partially ordered set (poset). For two elements $u,v\in L$ we consider the set $\{ w\in L \mid w\leq u, w\leq v\}$. If the maximum of this set exists, it is called the \emph{meet of $u$ and $v$} and denoted by $u\wedge v$. If any two elements of $L$ have a meet, then $(L,\wedge)$ is called the \emph{meet-semilattice of $L$}.
\end{definition}

\begin{remark}[{\cite[Rem.~10.2]{NS06}}]
Let $L$ be a finite poset and let $L=\{u_1,\ldots,u_{|L|}\}$ be a listing. We consider the $|L|\times |L|$-matrix $M$ with
 \begin{align*}
  M_{ij}= 
  \begin{cases}
   1 & \text{if } u_i\leq u_j , \\
   0 & \text{otherwise} .
  \end{cases}
 \end{align*}
 Then $M$ is invertible in $\ZZ^{|L|\times |L|}$ and the function
 \[ \mu: L\times L \to \ZZ,~ (u_i,u_j)\mapsto (M^{-1})_{ij} \]
 is independent of the choice of the listing. 
\end{remark}

\begin{definition}[{\cite[Def.~10.5]{NS06}}]
Let $L$ be a finite poset. Then the above-noted function, $\mu_L: L\times L \to \ZZ$, is called the \emph{M\"obius function of $L$}.
\end{definition}

As usual, we write $u<v$ if $u\leq v$ and $u\neq v$ for $u,v\in L$. 

\begin{lemma} \label{lem::cover}
 Let $L$ be a finite poset and let $u,v\in L$.
 \begin{enumerate}[label=(\roman*)]
     \item Then $\mu_L (u,u)=1$.
     \item If $v$ covers $u$, i.e. $u<v$ and there is no element $w\in L$ with $u< w < v$, then $\mu_L (u,v)=-1$.
 \end{enumerate}
\end{lemma} 

\begin{proof}
 We can choose a listing of $L$ such that the matrix $M$, which defines the M\"obius function, is unitriangular, see \cite[Ex.~10.25]{NS06}. Hence $\mu_L (u,u)=(M^{-1})_{uu}=1$. If $v$ covers $u$, we can additionally assume that $u$ and $v$ appear one after the other in the listing of $L$. Then 
 \[ N:=
   \begin{pmatrix}
   M_{uu} & M_{vu} \\
   M_{uv} & M_{vv} 
   \end{pmatrix}
  =\begin{pmatrix}
   1& 0 \\
   1& 1
  \end{pmatrix} \]
  is a block on the diagonal of $M$. Since $M$ is unitriangular, we have
  \[
   \begin{pmatrix}
   (M^{-1})_{uu} & (M^{-1})_{vu} \\
   (M^{-1})_{uv} & (M^{-1})_{vv} \\
  \end{pmatrix}
  = N^{-1}=
  \begin{pmatrix}
   1& 0 \\
   -1& 1
  \end{pmatrix}.
 \]
 It follows that $\mu_L (u,v)=(M^{-1})_{uv}=-1$.
\end{proof}

The M\"obius function can be helpful for computing certain determinants derived from a meet-semilattice: 

\begin{lemma}[{\cite[Lem.~7.1]{Kn07}}] \label{lem::det_semilattice}
 Let $\phi:L\to \CC$ be a function on a finite poset $L$ which is a meet-semilattice. Then, with $\mu_L$ the M\"obius function of $L$,  
 \[ \det(\phi(u\wedge v))_{u,v\in L}) = \prod_{x\in L} \Big( \sum_{\substack{y\in L \\ y\leq x}} \mu_L(y,x) \cdot \phi(y) \Big) .\]
\end{lemma}

Now, we recall the definition of the refinement order on partitions and show that partitions of $k$ lower points in a  group-theoretical category of partitions have a meet-semilattice structure with respect to this partial order. Note that Nica and Speicher are considering the reversed refinement order in \cite{NS06}; however, to be consistent with the conventions in Knop's article \cite{Kn07}, our definition is dual to theirs. 

\begin{definition}[{\cite[Ch.~9]{NS06}}] 
Let $k,l\in \NN_0$ and partitions $p,q\in P(k,l)$ on $k+l$ points.
We write $p\leq q$ if and only if each block of $q$ is completely contained in one of the blocks of $p$. The induced partial order is called the \emph{refinement order}.
\end{definition}

Note that $p\leq q$, if $p$ can be obtained by coarsening the block structure of $q$ and we say that $p$ is \emph{coarser} than $q$. Moreover, the meet $p\wedge q$ of $p$ and $q$ exists in $P(k,l)$ and is the \emph{common coarsening}, i.e. the finest partition which is coarser than both $p$ and $q$.

\begin{lemma} \label{lem::common_coarsening}
Let $\cC$ be a category of partitions. Then $\cC$ is closed under common coarsenings if and only if $\cC$ is group-theoretical.
\end{lemma}

\begin{proof} If $\cC$ is closed under coarsening, then it contains $\primarypart$, since this partition is a coarsening of the partition $\twoblockspart$, which is contained in any category of partitions. See \cite[Lem.~2.3]{RW14} for the opposite implication.
\end{proof}

\begin{corollary}
Let $\cC$ be a group-theoretical category of partitions $\cC$ and $k\in \NN_0$. Then the poset $\cC(k)=\cC(0,k)$ has a meet-semilattice structure with respect to the refinement order.
\end{corollary}

Let us denote the M\"obius functions of all the posets $\cC(k)$ by $\mu_\cC$ (omitting the $k$). This allows us to give a condition for the semisimplicity of $\RepCt$, see \cite[Lem.~8.2]{Kn07}.

\begin{lemma} \label{lem::semisimple_determinant}
 Let $\cC$ be a group-theoretical category of partitions. Then $\RepCt$ is semisimple if and only 
 \[ \Omega_k := \prod_{p\in \cC(k)} \Big( \sum_{\substack{q\in \cC(k) \\ q\leq p}} \mu_\cC(q,p) \cdot t^{\#q} \Big) \neq 0 \quad \text{for all } k\in \NN.\]
\end{lemma}

\begin{proof} By \Cref{lem::semisimple_determinant_general}, $\RepCt$ is semisimple if and only if the matrices
$$
 G^{(k)} = (t^{\ell(u^*,v)})_{u,v\in\cC(k)}
$$
have non-zero determinants for all $k\in\NN$. We define the map $\phi:\cC(k)\to \CC,~ p\mapsto t^{\#p}$ and since $\#(u\wedge v)=\ell(u^*,v)$ for all $u,v\in \cC(k)$, \Cref{lem::det_semilattice} implies that 
 \begin{align*}
  \det(G^{(k)}) 
  &= \det(\phi(u\wedge v))_{u,v\in \cC(k)}) \\ 
  &= \prod_{p\in \cC(k)} \Big( \sum_{\substack{q\in \cC(k) \\ q\leq p}} \mu_\cC(q,p) \cdot \phi(q) \Big) \\
  &= \prod_{p\in \cC(k)} \Big( \sum_{\substack{q\in \cC(k) \\ q\leq p}} \mu_\cC(q,p) \cdot t^{\#q} \Big) \\
  &= \Omega_k
 \end{align*}
\end{proof}

To compute the above-noted determinant, we will further factorise it. For this purpose we recall a definition of Knop's in the special case of $\uRep(S_t)$. For any $k\in \NN$, we set $\underline{k}:=\{1,\ldots,k\}$. By definition, the set $P(k)=P(0,k)$ is the set of partitions of $\underline{k}$. We denote by $s_k\in P(k)$ the finest partition in $P(k)$, where each block is of size one. Moreover, we set $\underline{0}:=\emptyset$ and $s_0:=\emptyset \in P(0)$.

\begin{definition}[See {\cite[Sec.~8]{Kn07}}]
Let $k,l\in \NN_0$ with $k\leq l$ and let $e:\underline{k}\hookrightarrow \underline{l}$ be an injective map.
We define two maps 
\[ e_*:P(l) \to P(k) \text{ and } e^*:P(k) \to P(l)\] 
as follows. For any $p\in P(l)$, let $e_*(p)$ be the partition of $\underline{k}$ whose blocks are the preimages of the blocks of $p$ under $e$. For any $q\in P(k)$, we define $e^*(q)$ as the partition of $\underline{l}$ whose blocks are the images of the blocks of $p$ under $e$ and additionally blocks of size one containing the elements not in the image of $e$.

Moreover, we define a scalar
 \[ w_e := \sum_{\substack{q\in P(l) \\ e_*(q)=s_k}} \mu_P(q,s_l) \cdot t^{\#q-k} \in \CC.\]
\end{definition}
Before we go on, we consider this definition in two special cases.

\begin{remark} \label{rem::k=0}
We consider the case $k=0$ and $l\in \NN_0$. Then there is just one map $e:\underline{0} \to \underline{l}$, since $\underline{0}=\emptyset$. Moreover, the set $P(0)$ consists of only one partition $s_0=\emptyset$ and $\# s_0 = 0$. Thus it follows from the definition that
\begin{align*}
    & e_*:P(l) \to P(0),~ p\mapsto s_0,\\
    & e^*:P(0) \to P(l),~ s_0 \mapsto s_l,\\
    & w_e = \sum_{q\in P(l)} \mu_P(q,s_l) \cdot t^{\#q}.
\end{align*}
\end{remark}

\begin{lemma} \label{lem::w_e=1}
 Let $l\in \NN_0$ and let $e:\underline{l}\to \underline{l}$ be a bijection. Then $w_e=1$.
\end{lemma}

\begin{proof}
 It follows from the definition that both $e_*$ and $e^*$ preserve block sizes in this case, and hence the only partition $q\in P(l)$ with $e_*(q)=s_l$ is the partition $s_l$ itself. Hence, we have
 \[ w_e = \sum_{\substack{q\in P(l) \\ e_*(q)=s_l}} \mu_P(q,s_l) \cdot t^{\#q-l} = \mu_P(s_l,s_l) \cdot t^{l-l} = 1. \]
\end{proof}

\begin{lemma} \label{lem::determinant_in_we}
 Let $\cC$ be a group-theoretical category of partitions. Then
 \[ \Omega_k = \prod_{p\in \cC(k)} w_{\emptyset \hookrightarrow \underline{\#p}} \quad \text{for all } k\in \NN.\]
\end{lemma}

\begin{proof}
 Let $p\in \cC(k)=\cC(0,k)$. We can view $p$ as an element in $P(k)=P(0,k)$, i.e., a partition of $\underline{k}$. Since $\cC$ is a group-theoretical category of partitions, any coarsening of $p$ lies again in $\cC$. Thus there is a natural bijection of posets
 \[ f: \cC_{\leq p} :=\{q\in \cC(k)\mid q\leq p\} \to P(\#p) \]
 mapping a coarsening of $p$ to the partition indicating the fusion of the blocks of $p$ in the following way: We can sort the blocks of $p$ according to the smallest elements they contain and label them using the numbers $1,\dots,\#p$. Then the coarsenings of $p$ are exactly those partitions $q$ whose blocks are unions of blocks of $p$. Hence, they correspond exactly to partitions of the set of blocks of $p$, and hence, the set $\underline{\#p}$, and this correspondence is compatible with the respective refinement order. It now follows directly that
 \begin{itemize}
  \item $\mu_\cC(q,q')=\mu_P(f(q),f(q'))$ for all $q,q'\in \cC_{\leq p}$,
  \item $\#q=\#(f(q))$ for all $q\in \cC_{\leq p}$ and
  \item $f(p)=s_{\#p}$.
 \end{itemize}
Together with \Cref{rem::k=0} it follows that
\begin{align*}
 \Omega_k &= \prod_{p\in \cC(k)} \Big( \sum_{\substack{q\in \cC(k) \\ q\leq p}} \mu_\cC(q,p) \cdot t^{\#q} \Big) \\
 &= \prod_{p\in \cC(k)} \Big( \sum_{q\in P(\#p)} \mu_P(q,s_{\#p}) \cdot t^{\#q} \Big) \\
 &= \prod_{p\in \cC(k)} w_{\emptyset \hookrightarrow \underline{\#p}}
\end{align*}
\end{proof}

\Cref{lem::semisimple_determinant} and \Cref{lem::determinant_in_we} imply the following corollary.
\begin{corollary} \label{cor::fac_determinant}
Let $\cC$ be a group-theoretical category of partitions. Then $\RepCt$ is semisimple if and only if $w_{\emptyset \hookrightarrow \underline{\#p}} \neq 0$ for all $k\in \NN$ and $p\in \cC(k)$.
\end{corollary}

In the following, we factorise the elements $w_{\emptyset \hookrightarrow \underline{\#p}}$ with $p\in \cC(k)$. As they are independent of $\cC$, we can apply \cite[Lem.~8.4]{Kn07} in the special case of $\uRep(S_t)$ to shows that the mapping $e\mapsto w_e$ turns compositions into products, as explained in the following.

\begin{lemma}[See {\cite[Lem.~8.4]{Kn07}}]
 Let $k,l\in \NN_0$ with $k\leq l$ and let $e:\underline{k}\hookrightarrow \underline{l}$ be an injective map. Then the pair $(e_*,e^*)$ is a Galois connection between $P(l)$ and $P(k)$, i.e.~$e_*(p)\leq q$ if and only if $p\leq e^*(q)$ for all $p\in P(l)$ and $q\in P(k)$.
\end{lemma}

\begin{proof}
First, assume $e_*(p)\leq q$. We consider two distinct points $x,y\in \underline{l}$ which lie in the same block of $e^*(q)$. As all points in $\underline{l}\setminus e(\underline{k})$ are in singleton blocks of $e^*(q)$, we have $x,y\in e(\underline{k})$ and their unique preimages $e^{-1}(x)$ and $e^{-1}(y)$ lie in the same block of $q$. As $e_*(p)\leq q$, they also lie in the same block of $e_*(p)$. It follows that $x$ and $y$ lie in the same block of $p$ and thus $p\leq e^*(q)$.

Let $p\leq e^*(q)$. We consider two distinct points $x,y\in \underline{k}$ which lie in the same block of $q$. Thus $e(x)$ and $e(y)$ lie in the same block of $e^*(q)$ and as $p\leq e^*(q)$, they lie in the same block of $p$. It follows that $x$ and $y$ lie in the same block of $e_*(p)$ and thus $e_*(p)\leq q$.
\end{proof}

In the following, let us extend the coarsening operation $\wedge$ $\ZZ$-linearly to $\ZZ$-linear combinations of partitions.

\begin{lemma}[See {\cite[Lem.~8.4]{Kn07}}] \label{lem::we_multiplicative}
 Let $j,k,l\in \NN_0$ with $j\leq k\leq l$ and let $\underline{j} \overset{\bar{e}}{\hookrightarrow} \underline{k}\overset{e}{\hookrightarrow} \underline{l}$ be injective maps. Then we have
 \[ w_{e\bar{e}} = w_e w_{\bar{e}} .\]
\end{lemma}

\begin{proof}
 By \cite[Lem.~7.2]{Kn07} we have
 \[ \sum_{\substack{q\in P(l) \\ q\leq p}} \mu_P(q,p) q = \Big(\sum_{\substack{r\in P(k) \\ r\leq e_*(p)}} \mu_P (r,e_*(p)) e^*(r)\Big) \wedge \Big(\sum_{\substack{s\in P(l) \\ s\leq p \\ e_*(s)=e_*(p)}} \mu_P(s,p) s\Big) \]
 for all $p\in P(l)$. For $p=s_l$ we obtain
 \[ \sum_{q\in P(l)} \mu_P(q,s_l) q = \Big( \sum_{r\in P(k)} \mu_P (r,s_k) e^*(r)\Big) \wedge \Big( \sum_{\substack{s\in P(l) \\ e_*(s)=s_k}} \mu_P(s,s_l) s\Big) .\]
 We define a $\CC$-linear map as follows:
 \[ \varphi: \CC P(l) \to \CC, ~ q \mapsto 
  \begin{cases}
   t^{\#q-j} & (e\bar{e})_*(q)=s_j, \\
   0 & \text{otherwise},\\
  \end{cases}
  \quad\text{for all }q\in P(l).
 \]
 We apply $\varphi$ on both sides of the equation and obtain
 \[ \sum_{\substack{q\in P(l) \\ (e\bar{e})_*(q)=s_j}} \mu_P(q,s_l) t^{\#q-j} = \sum_{r\in P(k)} \sum_{\substack{s\in P(l) \\ e_*(s)=s_k}} \mu_P(r,s_k) \mu_P(s,s_l) \varphi(e^*(r)\wedge s) .\]
 As the left-hand side is $w_{e\bar{e}}$, to prove that $w_{e\bar{e}} = w_e w_{\bar{e}}$ it suffices to show
 \[ \varphi(e^*(r)\wedge s) = 
  \begin{cases}
   (t^{\#r-j}) (t^{\#s-k})& \bar{e}_*(r)=s_j, \\
   0 & \text{otherwise},\\
  \end{cases}  \]
 for all $r\in P(k),s\in P(l)$ with $e_*(s)=s_k$. Now $e_*(s)=s_k$ implies that
 \[ (e\bar{e})_*(e^*(r)\wedge s) = \bar{e}_*(r\wedge e_*(s)) = \bar{e}_*(r\wedge s_k) =  \bar{e}_*(r), \]
 so we have $(e\bar{e})_*(e^*(r)\wedge s)=s_j$ if and only if $\bar{e}_*(r)=s_j$. Since all blocks of $e^*(r)$ involving the elements of $l\setminus e(\underline{k})$ are singletons and since $e_*(s)=s_k$, the common coarsening $e^*(r)\wedge s$ has exactly $\#r$ blocks involving only the elements of $e(\underline{k})$ and another $\#s-k$ blocks involving only the elements of $\underline{l}\setminus e(\underline{k})$. It follows that $\#(e^*(r)\wedge s)=\#r+\#s-k$ and hence 
 \[ \varphi(e^*(r)\wedge s) = t^{\#r+\#s-k-j} = (t^{\#r-j}) (t^{\#s-k}).\]
\end{proof}

Let us illustrate the lemma above with an example.
\begin{example}
 Let $\cC$ be a group-theoretical category of partitions, $m\in \NN$ and $p\in \cC(m)$. We set $l=\# p$ and consider an arbitrary injective map $e:\underline{1} \hookrightarrow \underline{\# p}$. Then $\emptyset \hookrightarrow \underline{l}$ decomposes into 
 \[ \emptyset \hookrightarrow \{1\} \overset{e}{\hookrightarrow} \underline{l}.\] 
 We have 
  \begin{align*}
      w_{\emptyset \hookrightarrow \{1\}} &= \sum_{q\in P(1)} \mu_P(q,s_1) \cdot t^{\#q} = \mu_P(s_1,s_1) \cdot t^{1} = t \\
      \text{and }w_e &= \sum_{\substack{q\in P(l) \\ e_*(q)=s_1}} \mu_P(q,s_l) \cdot t^{\#q-1} = \sum_{q\in P(l)} \mu_P(q,s_l) \cdot t^{\#q -1}
  \end{align*}
 and hence 
 \[ w_{\emptyset \hookrightarrow \underline{l}} = \sum_{q\in P(l)} \mu_P(q,s_l) \cdot t^{\#q} = w_{\emptyset \hookrightarrow \{1\}} \cdot w_e.\]
\end{example}

Now, we are ready to prove our first main theorem, see \Cref{thm::main_thm_1}.

\begin{theorem}\label{thm-grouptheo-semisimple}
 Let $\cC$ be a group-theoretical category of partitions and $t\in \CC$. Then $\RepCt$ is semisimple if and only if $t\notin \NN_0$.
\end{theorem}

\begin{proof} By \Cref{cor::fac_determinant}, the category $\RepCt$ is semisimple if and only if $w_{\emptyset \hookrightarrow \underline{\#p}} \neq 0$ for all $m\in \NN$, $p\in \cC(m)$. Hence \Cref{lem::we_multiplicative} implies that $\RepCt$ is semisimple if and only if $w_e\neq 0$ for any map $e:\underline{k}\hookrightarrow \underline{l}$, $k,l\in\NN_0$, which does not have a factorisation $e=e_1e_2$ with $e_1,e_2$ injective and not bijective maps, i.e.~for all $e:\underline{k}\hookrightarrow \underline{k+1}$, $k\in\NN_0$.

Let us describe $w_e$ for a given injective map $e:\underline{k}\hookrightarrow \underline{k+1}$. Set $l=k+1$. We can assume that $e(i)=i$ for any $i\in \underline{k}$, since this can be achieved by post-composing with an isomorphism $e':\underline{l}\to \underline{l}$ and $w_{e'}=1$ by \Cref{lem::w_e=1}. Thus $\{q\in P(l)\mid e_*(q)=s_k\}$ contains the partition $s_l\in P(l)$ and $X:= \{q\in P(l)\mid e_*(q)=s_k\} \setminus \{s_l\}$ contains exactly those $k$ partitions which consist of a block $\{i,l\}$ and singleton blocks otherwise, for $1\leq i\leq k$. It follows that
 \[ w_e = \mu_P(s_l,s_l) t^{l-k} + \sum_{q\in X} \mu_P(q,s_l) t^{k-k} .\]
Since $s_l$ covers every partition $q\in X$, we can apply \Cref{lem::cover} and conclude that 
 \[ w_e = 1\cdot t^{1} + \sum_{q\in X} (-1)t^{0} = t-k .\]
This proves our assertion that $\RepCt$ is semisimple if and only if $t\not\in\NN_0$.
\end{proof}

Together with \Cref{cor-criterion}, \Cref{thm-grouptheo-semisimple} implies that there are negligible morphisms in $\RepCt$ as soon as $t\in\NN_0$. To better understand negligible morphisms, we discuss some examples.

\begin{definition} For any group-theoretical category of partitions $\cC$, any $k,l\in \NN_0$ and any partition $p\in\cC(k,l)$, we define recursively
$$
x_p := p - \sum_{q\lneq p} x_q
\quad \in \Homm_\RepCt([k],[l]) .
$$
\end{definition}

\begin{remark}\label{rem-negligible-mor} 
If $t\in\NN_0$, then by \cite[Rem.~3.22]{CO11}, $x_p$ is negligible in $\uRep(S_t)$ if $p$ is a partition with more than $t$ parts (and in fact, those span the ideals of negligible morphisms in $\uRep_0(S_t)$). This implies that such $x_p$ are negligible in $\RepCt$ for any group-theoretical $\cC$. Subtracting such negligible morphisms we can see that modulo the tensor ideal of negligible morphisms, any morphism in $\RepCt$ is equivalent to a morphism which consists of partitions with at most $t$ parts each.
\end{remark}

\begin{example} \label{expl-xid} If $t\in \NN_0$ and $\cC$ is group-theoretical, then $x_{\id_{t+1}}$ is a non-trivial negligible endomorphism in $\RepCt$.
\end{example}
\section{Indecomposable objects}

In this section, we take a look at indecomposable objects in $\RepCt$ for any category of partitions $\cC$. Notions like $\End$ and $\Homm$ are meant with respect to the category $\RepCt$. We prove a classification result for indecomposable objects in $\RepCt$, \Cref{thm::main_thm_2}, which is uniform in $\cC$: for each category of partitions $\cC$, a distinguished set of projective partitions $\cP$ will be considered which defines a system of finite groups, the union of whose irreducible complex representations will be shown to correspond to the indecomposable objects in $\RepCt$. We also consider and derive results on the Grothendieck ring and the semisimplification of $\RepCt$.

\subsection{From indecomposable objects to primitive idempotents}\label{ssec::nuk}
In the following, we provide a strategy which reduces the problem of classifying indecomposable objects in $\RepCt$ to a classification of primitive idempotents in certain quotient algebras. Recall the following definitions.
\begin{definition}
Let $R$ be a ring. Two elements $a,b\in R$ are said to be \emph{conjugate} if there exists an invertible element $c\in R$ such that $a=cbc^{-1}$.

An element $e\in R$ is called \emph{idempotent} if $e^2=e$. Two idempotents $e_1,e_2\in R$ are said to be \emph{orthogonal}, if $e_1e_2=e_2e_1=0$. An idempotent $e\in R$ is called \emph{primitive} if it is non-zero and can not be decomposed as a sum of two non-zero orthogonal idempotents.
\end{definition} 

For $\cC=P$, the following statements are discussed in \cite[Prop.~2.20]{CO11}. They follow in our more general situation from the fact that $\RepCt$ is a Karoubian category with finite-dimensional morphism spaces. 

For any object $A\in \RepCt$ and any idempotent $e\in \End(A)$ we denote the image of $e$ by $(A,e)$.

\begin{lemma} \label{lem::lem_Krull-Schmidt}
Let $\cC$ be a category of partitions and $t\in \CC$. 
\begin{enumerate}[label=(\roman*)]
    \item Let $k\in \NN_0$ and let $e\in \End([k])$ be an idempotent. Then $([k],e)$ is indecomposable in $\RepCt$ if and only if $e$ is primitive. 
    \item For any two idempotents $e,e'\in \End([k])$ the objects $([k],e)$ and $([k],e')$ are isomorphic if and only if $e$ and $e'$ are conjugate in $\End([k])$.
    \item For any indecomposable object $X$ of $\RepCt$ there exist a $k\in \NN_0$ and a primitive idempotent $e\in \End([k])$ such that $X\cong ([k],e)$.
    \item (Krull--Schmidt property) Every object in $\RepCt$ is isomorphic to a direct sum of indecomposable objects, and this decomposition is unique up to the order of the indecomposables. 
\end{enumerate}
\end{lemma}

\begin{proof}
This is known for general Karoubian categories with finite-dimensional morphism spaces; see for instance \cite[Prop.~2.7.1]{Comes_Wilson}.
\end{proof}

For any conjugacy class $c$ of idempotents in $\End([k])$, we denote by $([k],c)$ the corresponding isomorphism class of objects in $\RepCt$. However, we frequently identify a primitive idempotent with its conjugacy class and an object with its isomorphism class. 

The following well-known lemmas allow us to classify primitive idempotents inductively.
\begin{definition}
For any algebra $B$, we denote by $\Lambda(B)$ the set of conjugacy classes of primitive idempotents of $B$.
\end{definition}

\begin{lemma}[{\cite[Lem.~3.3]{CO11}}] \label{lem::idempotent_corr} 
Let $A$ be a finite-dimensional $\CC$-algebra, $\xi \in A$ an idempotent and $(\xi)=A\xi A$ the two-sided ideal of $A$ generated by $\xi$. 
Then there is a bijective correspondence
\[ \Lambda(A) \overset{bij.}{\longleftrightarrow} \Lambda(\xi A \xi) \sqcup \Lambda(A/(\xi)) ;\]
a primitive idempotent in $A$ corresponds to a primitive idempotent in the subalgebra $\xi A \xi$ as soon as it lies in $(\xi)$, otherwise, its image under the quotient map $A\to A/(\xi)$ is a primitive idempotent in $A/(\xi)$, and for each primitive idempotent in $A/(\xi)$, there is a unique lift (up to conjugation) in $A$.
\end{lemma}

\begin{lemma}[Rosenberg's lemma {\cite[Lem.~3.3a]{Gr62}}] \label{lem::Rosenberg}
Let $A$ be a finite-dimensional $\CC$-algebra and $e\in A$ a primitive idempotent. If $I$ and $J$ are two-sided ideals of $A$ such that $e\in I+J$, then $e\in I$ or $e\in J$. 
\end{lemma}

In the following, we will exhibit the objects $[k]\in\cC$ for $k\geq0$ as subobjects of $[k+1]$ or $[k+2]$ (up to isomorphisms), depending on $\cC$. More precisely, we will distinguish the cases $\singleton \in \cC$ and $\singleton \notin \cC$, as in the latter case 
we have the following useful feature. Recall that the composition $\cdot$ in $\cC$ does not depend on the interpolation parameter $t$, in contrast to the composition in the interpolation categories $\RepCt$.

\begin{lemma} \label{lem::hom_different_parity} If $\singleton \notin \cC$, then $\cC(k,l)=\emptyset$ whenever $k \not\equiv l \mod 2$.
\end{lemma}

\begin{proof} 
Assume that there exists a partition $p\in \cC(k,l)$ with $k \not\equiv l \mod 2$. By successive composition with $\Paa\ot \cdots \ot \Paa \ot \Laa$ and $\Paa\ot \cdots \ot \Paa \ot \Uaa$ we would obtain the partition $\singleton\in \cC(0,1)$ or $\upsingleton \in \cC(1,0)$ and hence $\singleton \in \cC$. 
\end{proof}

\begin{definition}
For $t\neq 0$ we define the idempotents
\begin{align*}
    \nu_0 := 0,
    \quad 
    \nu_1 := \begin{cases} \frac1t \bbar & \singleton \in \cC , \\
    0 & \text{else}
    \end{cases}  ,
    \quad
    \nu_k := \begin{cases} \frac1t ~ \id_{k-1} \ot \bbar & \singleton \in \cC , \\
    \frac1t ~ \id_{k-2} \ot \twoblocks & \text{else} 
    \end{cases} ,
    \quad\text{for all }k\geq 2
\end{align*}
in $\End_{\RepCt}([k])$, $k\in\NN_0$.
\end{definition}

\begin{lemma} \label{lem::chi_isomorphisms}
Set $d:=1$ if $\singleton\in\cC$, and $d:=2$ otherwise. Then for $t\neq0$ and $k\geq d$,
$$ ([k],\nu_k) \cong [k-d] \quad\text{in }\RepCt.$$
More generally, for any $0\leq l<k$ with $k\equiv l \mod d$, there exists a partition $\nu \in \cC(l,k)$ such that for any idempotent $e\in\End([l])$, $t^{(l-k)/d} \nu e \nu^*$ is an idempotent in $\End([k])$ and $([l],e)\cong ([k],t^{(l-k)/d} \nu e \nu^*)$.
\end{lemma}

\begin{proof}
For $l=k-d$, we set
\[
\scalebox{.8}{\begin{tikzpicture}
\coordinate [label=left:{\scalebox{1.25}{$\nu:=$}}](O) at (0,0.45);
\coordinate [label=right:{\scalebox{1.25}{$\in\cC(l,k),$}}](O) at (3,0.45);
\coordinate [label=right:{$\ldots$}](O) at (0.35,0.5);
\coordinate (A1) at (0,0);
\coordinate (A2) at (1.5,0);
\coordinate (A3) at (2,0);
\coordinate (A4) at (2.5,0);
\coordinate (B1) at (0,1);
\coordinate (B2) at (1.5,1);
\coordinate (B3) at (2,1);

\fill (A1) circle (2.5pt);
\fill (A2) circle (2.5pt);
\fill (A3) circle (2.5pt);
\fill (A4) circle (2.5pt);
\fill (B1) circle (2.5pt);
\fill (B2) circle (2.5pt);
\fill (B3) circle (2.5pt);

\draw (A1) -- (B1);
\draw (A2) -- (B2);
\draw (2.5,0.4) -- (A4);
\draw (A3) -- (B3);
\end{tikzpicture}} 
\]
if $\singleton\in\cC$, or otherwise
\[
\scalebox{.8}{\begin{tikzpicture}
\coordinate [label=left:{\scalebox{1.25}{$\nu:=$}}](O) at (0,0.45);
\coordinate [label=right:{\scalebox{1.25}{$\in\cC(l,k).$}}](O) at (3.5,0.45);
\coordinate [label=right:{$\ldots$}](O) at (0.35,0.5);
\coordinate (A1) at (0,0);
\coordinate (A2) at (1.5,0);
\coordinate (A3) at (2,0);
\coordinate (A4) at (2.5,0);
\coordinate (A5) at (3,0);
\coordinate (B1) at (0,1);
\coordinate (B2) at (1.5,1);
\coordinate (B3) at (2,1);

\fill (A1) circle (2.5pt);
\fill (A2) circle (2.5pt);
\fill (A3) circle (2.5pt);
\fill (A4) circle (2.5pt);
\fill (A5) circle (2.5pt);
\fill (B1) circle (2.5pt);
\fill (B2) circle (2.5pt);
\fill (B3) circle (2.5pt);

\draw (A1) -- (B1);
\draw (A2) -- (B2);
\draw (A3) -- (B3);
\draw (A5) -- (3,0.4) -- (2.5,0.4) -- (A4);
\end{tikzpicture}} 
\]
Then $\nu \nu^* = t~ \nu_k$ and $\nu^* \nu = t~\id_l$ and thus
\begin{align*}
\nu: [l] \to ([k],\nu_k),\qquad
\frac1t \nu^*: ([k],\nu_k) \to [l]
\end{align*}
define mutually inverse isomorphisms, which also restrict to subobjects. An iterative application yields the second claim.



\end{proof}

\begin{remark}
The previous lemma implies that every object is isomorphic to a subobject of $[k]$ for some sufficiently large $k$, if $\singleton \in \cC$ (see \cite[Pf.~of.~Lem.~3.6]{CO11} for the case $\cC=P$), while if $\singleton\not\in\cC$, then any object $X$ in $\RepCt$ is isomorphic to a subobject of $[k]\oplus[k+1]$ for some sufficiently large $k$. As there are no non-zero morphisms between $[k]$ and $[k+1]$, the endomorphism algebra of $X$ is a direct summand in $\End([k]\oplus[k+1])$. So in particular, in both cases $\End(X)$ is semisimple if $\End([k])$ is semisimple for all $k\in \NN_0$, which can be checked by verifying that $G^{(2k)}\neq0$ for all $k\in \NN_0$ (see the proof of \Cref{lem::semisimple_determinant_general}). 
\end{remark}

With \Cref{lem::chi_isomorphisms} we are now able to decide whether a given subobject of $[k]$ is isomorphic to a subobject of $[l]$ with $l\leq k$.
\begin{lemma} \label{lem::Rk}
Let $t\neq 0$, $k\in \NN_0$ and $e\in \End ([k])$ a primitive idempotent. Then $([k],e)$ is isomorphic to a subobject of $[l]$ for some $l<k$ if and only if $e\in (\nu_k)$.
\end{lemma}

\begin{proof} Let $e\in (\nu_k)$. Then \Cref{lem::idempotent_corr} implies that $e$ is conjugated to some primitive idempotent in $\nu_k \End ([k]) \nu_k$ and hence we can assume that $e \in \nu_k \End ([k]) \nu_k$. Then $([k],e)$ is isomorphic to a subobject of $([k],\nu_k)$ and \Cref{lem::chi_isomorphisms} tells us that $([k],e)$ is isomorphic to a subobject of $[k-d]$. 

Now, let $e\notin (\nu_k)$. Consider an object of the form $([l],f)$ with $l<k$ and we assume that it is isomorphic to $([k],e)$. Then \Cref{lem::chi_isomorphisms} together with \Cref{lem::hom_different_parity} implies that there exists an idempotent $f'\in \End([k])$ with $([l],f)\cong ([k],f')$ and $f'\in (\nu_k)$. But since $e\notin (\nu_k)$, the idempotents $f'$ and $e$ are not conjugate. Hence $([l],f)$ is not isomorphic to $([k],e)$, which is a contradiction. 
\end{proof}

We obtain our first general description of the indecomposable objects in interpolating partition categories.

\begin{definition} \label{def::Lambda} 
For $t\neq 0$ we set 
$$ \Lambda_k := \Lambda(\End([k]) / (\nu_k)),$$
so $\Lambda_k$ is the set of conjugacy classes of primitive idempotents in the quotient algebras defined by the idempotents $\nu_k,k\in \NN_0$. For any $e\in\Lambda_k$, we denote its unique (primitive idempotent) lift in $\Lambda(\End([k]))$ by $L_e$ (see \Cref{lem::idempotent_corr}).
\end{definition}

Note that $\Lambda_0=\{\id_0\}$ and
$
\Lambda_1 = \begin{cases}
    \{ \id_1 - \frac{1}{t}\bbar, \frac{1}{t}\bbar \} 
    & \text{if }\bbar\in\cC(1,1) ,
    \\ 
    \{ \id_1 \} & \text{else} .
\end{cases}
$

\begin{proposition} \label{thm::indecomp_obj} 
For any category of partitions $\cC$ and $t\in \CC \backslash \{0\}$ there is a bijection
\begin{align*}
    \phantom{\qquad \qquad} \phi: \bigsqcup_{k\in \NN_0} \Lambda_k \to \left\{ \begin{matrix} \text{isomorphism classes of non-zero} \\ \text{indecomposable objects in } \RepCt \end{matrix} \right\}, ~\Lambda_k\ni e \mapsto ([k], L_e).
    \end{align*}
\end{proposition}

\begin{proof} 
By \Cref{lem::lem_Krull-Schmidt} the isomorphism classes of non-zero indecomposable objects in $\RepCt$ are in bijection with the conjugacy classes of primitive idempotents in $\End ([k])$ for which $([k],e)$ is not a subobject of $[l]$ for any $l<k$. By \Cref{lem::Rk} these are exactly the conjugacy classes of idempotents in $\End([k]) \backslash (\nu_k)$. Now, \Cref{lem::idempotent_corr} implies that these coincide with $\Lambda_k$.   
\end{proof}

\begin{remark}
The case $t=0$ can be treated analogously by adjusting the definition of $\nu_k$ as follows
\begin{align*}
    \nu_0 := \nu_1 := 0,
    \quad
    \nu_2:= \begin{cases} \vierpart & \singleton, \vierpart \in \cC, \\ 0 & \text{else}, 
    \end{cases}
    \quad 
    \nu_k := \begin{cases} \id_{k-2} \ot \vierpart & \singleton, \vierpart \in \cC  ,\\
    \id_{k-2} \ot \twoblockspart & \text{else}
    \end{cases},
    \quad\text{for all }k\geq 3.
\end{align*}
Check for instance that every composition $([0],\id_0) \to ([l],e) \to ([0],\id_0)$ is a non-zero power of $t$, and hence zero, if $t=0$ and $l>0$. Thus, for $t=0$, we have to set $\nu_1=0$ and $\nu_2=0$, if $\singleton \notin \cC$. An analogous argument shows that the statement of \Cref{thm::indecomp_obj} is still true in the case $t=0$ with the given modifications for $\nu_k$.
\end{remark}

\subsection{Projective partitions} \label{ssec::projectives}
In the previous subsection we reduced the problem of classifying indecomposable objects in $\RepCt$ to a classification of primitive idempotents in (certain quotients of) the endomorphism algebras. We will now provide a strategy which reduces the problem further to a combinatorial problem of computing equivalence classes of certain distinguished partitions. 

For the rest of this article we will assume that $t\neq 0$. Recall that we denote by $q\cdot p$ the partition obtained by the composition of $p$ and $q$ for two compatible partitions $p,q$, while we denote by $qp= t^{\ell(q,p)} q\cdot p$ the multiplication in $\RepCt$, where $\ell(p,q)$ is the number of connected components concentrated in the ``middle row'' of the vertical concatenation of $p$ and $q$. By assuming $t\neq 0$ we have $q\cdot p= t^{-\ell(q,p)} qp$. Note also that $p\cdot q$ is the composition in $\uRep(\cC,1)$.

We fix some $k\in \NN_0$ and denote $E:=\End_\cC([k])=\CC\cC(k,k)$. 
We will use methods of \cite{FW16} and we start by recalling some definitions:

\begin{definition} \label{def::through-blocks}
A block (= connected component) of a partition $p\in P(k,l)$ is called a \emph{through-block} if it contains upper points as well as lower points. We denote the number of through-blocks by $t(p)$. Moreover, we denote by
$$
E_T := E_T(k) := \bigoplus_{q\in\cC(k,k): t(q)< T} \CC q
\quad\subset E
$$
the subspace by all partitions with less than $T$ through-blocks, for any $T\in \NN_0$.
\end{definition}

We will often omit the symbol $k$ and just write $E_T$ for $E_T(k)$.

\begin{lemma} \label{lem::ET} The subspaces $E_0\subset\dots\subset E_k$ form an (exhaustive ascending) filtration of the algebra $E$ by two-sided ideals.
\end{lemma}

\begin{proof} We can view the endomorphism algebra $E$ as a subalgebra of the endomorphism algebra $\End_P([k])$ of $[k]$ in $\uRep(P,t)$. Then for any $T\geq1$, $E_T$ is the subspace of endomorphisms in $E$ which factor through the object $[T-1]$ in $\uRep(P,t)$. In particular, this makes it clear that $E_T$ is an ideal for $T\geq 1$, and the same is true for $E_0=\{0\}$. 

The containments $E_0\subset\dots\subset E_k$ follow directly from the definition. Clearly, $E_k=E$, as no partition in with $k$ upper and lower points can have $t(p)>k$.
\end{proof}

We define projective partitions in the spirit of \cite[Def.~2.7]{FW16}.

\begin{definition} \label{def::projPart}
A partition $p\in P(k,k)$ is called \emph{projective}, if there exists a partition $p_0\in P(k,t(p))$ such that $p=p_0^*p_0$. For any category of partitions $\cC$, we denote by $\Projk$ the set of all projective partitions in $\cC(k,k)$. 
\end{definition}

\begin{remark}
Note that for a projective partition $p=p_0^*p_0\in \cC$, $p_0$ is a partition in $P(k,t(p))$, but not necessarily in $\cC(k,t(p))$. Moreover, by the structure of $p_0$, there can not be any loops in the composition $p_0^* \cdot p_0$ and hence $p_0^*p_0$ is indeed a partition, not a scalar multiple of one.

By \cite[Lem.~2.11]{FW16}, a partition $p\in \cC(k,k)$ is projective if and only if $p=p^*$ and $p=p\cdot p$. Thus, $t^{-\ell(p,p)} p$ is an idempotent in $\End_{\RepCt}([k])=\CC \cC(k,k)$. Also, for any $q\in\cC(k,k)$ the two partitions $q\cdot q^*$ and $q^*\cdot q$ are projective partitions.
\end{remark}

\begin{example}
The partitions $\vierpart \in P(2,2)$ and $\twoblocks \in P(2,2)$ are projective, but $\twoblockspart \in  P(3,3)$ is not.
\end{example}

The following lemma shows that we can use projective partitions to decompose the ring $E$ as a sum of ideals, which will then help us to compute primitive idempotents in $E$.
\begin{lemma} \label{lem-ideals-p} For any $T\in \NN_0$
$$ 
  E_T = \sum_{p\in \Projk, t(p)< T} (p)
$$
and, in particular, 
$$
  E = \sum_{p\in \Projk} (p).
$$
\end{lemma}

\begin{proof}
Consider $q\in \cC(k,k)$ with $t(q)< T$. We set $p:=q\cdot q^* \in \cC(k,k)$. By \cite[Lem.~2.11]{FW16} the partition $p$ is projective, $q=p\cdot q$, and $t(p)\leq t(q)< T$. It follows that $q=p\cdot q = t^{-\ell(p,q)} pq \in (p)$.

This proves the inclusions of the left-hand sides in the right-hand sides. The opposites inclusions follow from the fact that the number of through-blocks of a product is limited by the number of through-blocks of each factor. 
\end{proof}

In \cite[Def.~4.1]{FW16}, Freslon and Weber associated to every projective partition a representation of the corresponding easy quantum group using the functor $\mathcal{F}$ described in \Cref{subsec::easyQG}. They observe that this representation is far from being irreducible, and go on to determine its irreducible components. Similarly, the ideals $(p)$ contain a lot of primitive idempotents with a complicated structure. Thus, using \Cref{lem::idempotent_corr}, we will break these sets up into smaller sets of primitive idempotents, which we understand.
\begin{definition}
For any $p\in\cC(k,k)$ we denote by 
$$
I_p := p E p \cap E_{t(p)} = p E_{t(p)} p
$$ 
the ideal in $p E p$ which is spanned by all partitions with less than $t(p)$ through-blocks.
\end{definition}

\begin{proposition} \label{lem::surjection}
For any primitive idempotent $e\in\Lambda(pEp/I_p)$, there is a unique primitive idempotent lift $\L_p(e)\in\Lambda(pEp)\subset\Lambda(E)$, and the mapping
\[ \L: \bigsqcup_{p\in \Projk} \Lambda(p E p / I_p ) \to \Lambda(E),~ e\mapsto \L_p(e),
\]
is surjective.
\end{proposition}

\begin{proof} 
Recall from \Cref{lem::idempotent_corr} that we can uniquely lift conjugacy classes of (primitive) idempotents modulo any ideal which is generated by an idempotent. Since by \Cref{lem-ideals-p}, $E_{t(p)}$ is a sum of principal ideals in $E$ generated by idempotents, $I_p=pE_{t(p)}p$ is a sum of principal ideals in $pEp$ generated by idempotents and we can lift any conjugacy  class of (primitive) idempotents from $pEp/I_p$ to a unique conjugacy class of primitive idempotents $\L_p(e)$ in $pEp$.

Let $f\in E$ be a primitive idempotent. Let $T\in \NN$ be minimal such that $f\in E_{T+1}$, i.e.~any summand of $f$ has at most $T$ through-blocks and $f$ has a summand with $T$ through-blocks. By \Cref{lem-ideals-p}, $f$ lies in the sum of ideals generated by projective partitions in $E_{T+1}$ and, by Rosenberg's Lemma \ref{lem::Rosenberg}, $f$ lies in one of these ideals, i.e.~$f\in (p)$ for some $p\in \Projk$ with $t(p)\leq T$. Since $E_{t(p)+1}$ is an ideal, $(p)\subseteq E_{t(p)+1}$, and hence, by minimality of $T$, $t(p)=T$. In particular, it follows that $f\notin I_p$. If we apply \Cref{lem::idempotent_corr} inductively for all projective partitions in $I_p$, it follows, together with \Cref{lem-ideals-p}, that there exists a primitive idempotent $e\in p E p / I_p$ such that its lift $\L_p(e)$ is conjugate to $f$.

Note, in particular, that idempotents made up of partitions with at most $T$ through-blocks can be obtained as lifts of idempotents in $(p)$ for a projective partition $p$ with the same number of through-blocks $T$, for any $T\geq0$.
\end{proof}

Thus, in order to understand indecomposables in $\RepCt$, we have to describe the primitive idempotents in the quotients $pEp/I_p$. It turns out that this can be achieved using combinatorial ideals explained in \cite[Sec.~4.2]{FW16}. In particular, we will need a certain subgroup $S(p)$ of a symmetric group which we associate to any projective partition $p$.

\begin{definition}[{cf.~\cite[Def.~4.7]{FW16}}] \label{def::GroupsSp}
Let $p\in\Projk$ be a projective partition with $T:=t(p)$ through-blocks and with a decomposition $p=p_0^*p_0$ with $p_0\in P(k,T)$ (we call such a decomposition \emph{through-block factorisation} of $p$). For any $\sigma\in S_T$ we define $p_\sigma := p_0^* \sigma p_0$ in $P(k,k)$ and $S(p) := \{ \sigma\in S_T \mid p_\sigma \in\cC(k,k) \}$.
\end{definition}
In contrast to the definition in \cite{FW16}, our definition of a through-block decomposition does not result in a canonical choice of $p_0$ for a given projective partition $p$, as this will not be necessary for our purposes.

Note that $p=p\cdot p=p_0^*(p_0\cdot p_0^*)p_0$ implies that $p_0 \cdot p_0^*\in P(T,T)$ is a partition with at least $T$ through-blocks, hence, it is a permutation. Due to its symmetric factorisation, we even get $p_0\cdot p_0^*=\id$.
This implies that $p_\sigma \cdot p_\tau= p_{\sigma\tau}$ for $\sigma,\tau\in S_T$. As also $p_\id=p$, $S(p)$ is a subgroup of $S_T$. In fact, the subgroup is the same up to conjugation in $S_T$ for all choices of $p_0$.

\begin{example}\label{ex::FW}
If $\cC=P$ is the category of all partition, we have $S(p)=S_{t(p)}$ for all $p\in \ProjC$. It is easy to check that the same holds for $\cC=P_2$, the category of partitions with only blocks of size two, and $\cC=\langle \cross, \singleton \rangle$, the category of partitions with blocks of size one or two.

If $\cC \subseteq NC$ is a category of partitions in which all partitions are non-crossing, then $S(p)=\{ \id \}$ for all $p\in \Projk$. 
\end{example}

We will compute the groups $S(p)$ for some examples below (see \Cref{lem::Sp-concrete}).

The next lemma is an abstraction of Proposition 4.15 in \cite{FW16}. 
\begin{lemma}\label{lem::FW}
 Let $p\in\Projk$ be a projective partition. Then the map $\CC S(p)\to p E p, ~\sigma\mapsto t^{-\ell(p,p)} p_\sigma$, induces an algebra isomorphism between $\CC S(p)$ and $p E p / I_p$.
\end{lemma}

\begin{proof} 
Due to the observed multiplicativity, the map is an algebra homomorphism.

Now $p E p/I_p$ is spanned by $p\cdot q\cdot p+I_p$, where $p\cdot q\cdot p$ is a partition with $T:=t(p)$ through-blocks. As $p\cdot q \cdot p =  p_0^* (p_0\cdot q\cdot p_0^*) p_0$, this means $p_0\cdot q\cdot p_0^*\in P(T,T)$ has at least $T$ through-blocks. Hence it is a permutation, and $p\cdot q\cdot p = p_{p_0\cdot q\cdot p_0^*}$ lies in the image of our map.

We claim that $p_\sigma \neq p$ for any $\id\neq\sigma\in S_T$. Indeed, assume $p_\sigma=p$, then
$$ \sigma = (p_0 \cdot p_0^*) \sigma (p_0 \cdot p_0^*) = p_0 \cdot (p_0^* \sigma p_0) \cdot p_0^* = p_0 \cdot p \cdot p_0^* = p_0 \cdot p_0^* \cdot p_0 \cdot p_0^* = \id_T,
$$
as $p_0\cdot p_0^*=\id_T$. This implies that the $p_\sigma$ form a set of distinct partitions with exactly $T$ through-blocks. Hence, they are linearly independent even modulo $I_p$, and our map is bijective.
\end{proof}

In particular, the group algebra of the group $S(p)$ encodes the relevant information on primitive idempotents in the quotient $pEp/I_p$ for any fixed projective $p$. 

To investigate how primitive idempotents stemming from different projective idempotents $p$ and $q$ interact in $E$, let us make the following definition:


\begin{definition}
Let $p\in \Projk$ be a projective partition. We denote by 
$$\Lambda_k^{(p)}=\{ \L_p(e) \mid e \in \Lambda(pEp / I_p)\}$$ 
the set of conjugacy classes of (primitive idempotent) lifts of all idempotents in $\Lambda(pEp / I_p)$ into $E$. 
\end{definition}

Now, we want to study under which conditions $\Lambda_k^{(p)}\cap \Lambda_k^{(q)}\neq \emptyset$ for projective partitions $p,q\in \Projk$. It turns out that this is exactly the case if $p$ and $q$ are equivalent in the sense of \cite[Def.~4.17]{FW16} and then we have $\Lambda_k^{(p)}=\Lambda_k^{(q)}$.

\begin{definition} \label{def::equivprojpart}
Two projective partitions $p,q\in \Projk$ are \emph{equivalent in $\cC$}, denoted by $p\sim q$, if there exists a partition $r\in \cC(k,k)$ such that $r\cdot r^*=p$ and $r^*\cdot r=q$. We denote the set of equivalence classes by $\Projk / \sim$.
\end{definition}

Note that $p$ and $q$ being equivalent implies $t(p)=t(q)$ by \cite[Lem.~4.19]{FW16}.

\begin{lemma} \label{lem::part_equiv}
 Two projective partitions $p,q\in \Projk$ are equivalent if and only if the ideals $(p),(q)\unlhd E$ coincide.
\end{lemma}

\begin{proof}
If $p$ and $q$ are equivalent, then $p=p\cdot p=r\cdot r^*\cdot r\cdot r^*=r\cdot q\cdot r^* = t^{-2\ell(r,q)} rqr^* \in (q)$. Similarly, we have $q\in (p)$ and hence $(p)=(q)$.

Now, let $(p)=(q)$. Then we have $t(p)=t(q)$, which is largest number of through-blocks of any partition contained in the ideal, and there exist elements $a,b\in \CC \cC(k,k)$ with $p=aqb$. Since $p$ and $q$ are both partitions, we can assume that $a,b$ are partitions as well and $p=a\cdot q\cdot b$. Moreover, as $p$ and $q$ are symmetric partitions, we have $b=a^*$. Then $p=a\cdot q\cdot a^*=a\cdot q\cdot q\cdot a^*=(a\cdot q)\cdot (a\cdot q)^*$.

Let $T:=t(p)=t(q)$, and write $q=q_0^*q_0$ for some $q_0\in P(k,T)$. As $p=p^*=p\cdot p$, we have
$$
p = (a\cdot q\cdot a^*)\cdot (a\cdot q\cdot a^*)^* = (a \cdot q_0^* q_0 \cdot a^*) \cdot (a\cdot q_0^* q_0\cdot a^*) = (a \cdot q_0^*)(q_0\cdot a^* \cdot a\cdot q_0^*)(q_0\cdot a^*)
. 
$$
Here, $q_0 \cdot a^* \cdot a \cdot q_0^*$ is a partition in $P(T,T)$ with at least $T$ through-blocks, so all blocks contain exactly one upper and one lower point. Moreover, it has a symmetric factorisation as $(q_0 \cdot a^*)\cdot (q_0 \cdot a^*)^*$, so it must be the identity partition. This means $(a\cdot q)^*\cdot (a\cdot q) = q_0^* (q_0 \cdot a^* \cdot a \cdot q_0^*) \cdot q_0 = q$, showing that $p$ and $q$ are equivalent, as desired.
\end{proof}

\begin{lemma} \label{lem::equiv_idemp_equiv_proj}
 Let $p,q\in \Projk$ be two projective partitions. 
 \begin{enumerate}[label=(\roman*)]
    \item If $p$ and $q$ are equivalent, then $\Lambda_k^{(p)}=\Lambda_k^{(q)}$.
    \item If $p$ and $q$ are not equivalent, then $\Lambda_k^{(p)}\cap \Lambda_k^{(q)}=\emptyset$.
 \end{enumerate}
\end{lemma}

\begin{proof} (i) By \Cref{lem::idempotent_corr} the set $\Lambda_k^{(p)}$ contains the conjugacy classes of primitive idempotents in $(p)$ but not in $I_p$. If $p$ and $q$ are equivalent, then $(p)=(q)$ and $t(p)=t(q)$.
\smallskip

(ii)  Let $e$ be a primitive idempotent in $(p)\cap (q)$, but not in $I_p$ or $I_q$. Then we can assume that $e\in pEp$ and write
\[ e = \sum_{r\in p\cC(k,k)p\cap(q)} a_r r  \]
with $a_r\in \CC$ for all $r\in \cC(k,k)$. Here we use that $(q)$ is spanned by the partitions it contains. Since $e\notin I_p$, there exists a partition $r$ with $a_r\neq0$ and $t(p)$ through-blocks.

By \Cref{lem::FW}, $r$ lies in the span of partitions of the form $p_\sigma$ modulo $I_p$, but as both $r$ and $p_\sigma$ are partitions with $t(p)$ through-blocks, and as sets of distinct partitions are linearly independent, $r=p_{\sigma}$ for a permutation $\sigma \in S(t(p))$. This yields $p=p_{\id_{t(p)}}=p_{\sigma} \cdot p_{\sigma^{-1}} = r \cdot p_{\sigma^{-1}} = t^{-\ell(r,p_{\sigma^{-1}})}~ rp_{\sigma^{-1}} \in (q)$. Similarly, one can check that $q\in (p)$ and hence $(p)=(q)$. By \Cref{lem::part_equiv} this implies that $p$ and $q$ are equivalent. 
\end{proof}

The previous lemma together with \Cref{lem::surjection} gives the following description of the primitive idempotents in $E$.
\begin{lemma} \label{lem::proj_part_final_corr}
The following mapping is a bijection
\[ \L: \bigsqcup_{[p]\in \Projk/ \sim} \Lambda(p E p / I_p ) \to \Lambda(E),
~ e\mapsto \L_p(e)
 .
\]
\end{lemma}

\subsection{Parametrising indecomposable objects}
The previous subsection resulted in a description of all primitive idempotents in the endomorphism algebra $\End ([k])$ up to conjugation, and hence a description of the indecomposable objects of the form $([k],e)$ up to isomorphism, for a fixed $k\in \NN_0$. Now, in order to describe all indecomposable objects in $\RepCt$ up to isomorphism, we apply the results of \Cref{ssec::nuk} to determine those primitive idempotents $e\in \End ([k])$ which do not yield subobjects of $[l]$ for some $l<k$.

\begin{definition} \label{def::P} Let us define the subset
$$
\cP := \{ p\in\Projk: k\in \NN_0, p\not\in b^* \cdot \Projl \cdot b \text{ for all } 0\leq l<k, b\in \cC(k, l) \} 
$$
of $\ProjC$. 
\end{definition}

Note that the equivalence relation $\sim$ induces one on $\cP$, because if $r\cdot r^*=b^*\cdot a\cdot b$ for some $r\in\cC(k,k)$, $b\in\cC(k,l)$, $a\in\Projl$, then 
$$
r^*\cdot r = r^*\cdot (r\cdot r^*)\cdot r = r^* \cdot (b^* \cdot a\cdot b)\cdot r = (b\cdot r)^* \cdot a \cdot (b\cdot r) 
 .
$$

\begin{lemma} \label{lem::Projk_Prokl}
Let $k\in \NN_0$ and let $p\in \Projk$ be a projective partition. Then the following are equivalent:
\begin{enumerate}
    \item $p\in \Projk \backslash \cP$,
    \item there is an $l<k$ such that $([k],\L_p(e))$ is isomorphic to a direct summand of $[l]$ for all $e\in \Lambda_k^{(p)}$,
    \item there are $e\in \Lambda_k^{(p)}$ and $l<k$ such that $([k],\L_p(e))$ is isomorphic to a direct summand of $[l]$.
\end{enumerate}
\end{lemma}

\begin{proof}
$(1) \Rightarrow (2)$. Let $p\in \Projk \backslash \cP$. Then there exist $l<k$, $b\in\cC(k, l)$, and $q\in\Projl$ such that $p=b^*\cdot q\cdot b$. Replacing $q$ by $b\cdot b^*\cdot q\cdot b\cdot b^*$, we can assume $b\cdot b^*\cdot q=q=q\cdot b\cdot b^*$. Then we have $\ell(p,p)=\ell(q,q)=\ell(b,b^*)$ and we set $\alpha =t^{-\ell(p,p)}$, so $\alpha p$ and $\alpha q$ are idempotent endomorphisms of the objects $[k]$ and $[l]$, respectively. As in the proof of \Cref{lem::chi_isomorphisms}, we see that $q\cdot b:([k],\alpha p) \to ([l],\alpha q)$ and $b^*\cdot q:([l],\alpha q) \to ([k],\alpha p)$ yield isomorphisms between the objects $([k],\alpha p)$ and $([l],\alpha q)$ in $\RepCt$, and hence for arbitrary subobjects.

$(2) \Rightarrow (3)$, clearly.

$(3) \Rightarrow (1)$. We assume that $([k],\L_p(e))$ is isomorphic to a direct summand of $[l]$ for some $l<k$ and some $e\in \Lambda_k^{(p)}$. Then by \Cref{lem::surjection} there exists a projective partition $q\in \Projl$ such that $([k],\L_p(e))\cong ([l],\L_f)$ for some $f\in \Lambda_l^{(q)}$ and by \Cref{lem::chi_isomorphisms} we have $([l],\L_f) \cong ([k],t^{-m} \nu \L_f \nu^*)$ for some partition $\nu \in \cC (l,k)$ and $m\in\NN_0$. But $t^{-m} \nu f \nu^* \in \Lambda_k^{(\nu q\nu^*)}$, hence $\Lambda_k^{(p)} \cap \Lambda_k^{(\nu q\nu^*)} \neq \emptyset$. Thus $p$ and $\nu q\nu^*$ are equivalent by \Cref{lem::equiv_idemp_equiv_proj}, so as $\nu^* q\nu \notin \cP$, it follows that $p\notin \cP$.
\end{proof}

\begin{remark}
Recall that we used distinguished idempotents $\nu_k\in \End ([k])$, $k\in \NN_0$, to establish a correspondence between indecomposables in $\RepCt$ and primitive idempotents in $\Lambda_k = \Lambda(E/ (\nu_k))$, see \Cref{thm::indecomp_obj}. The previous lemma together with \Cref{lem::Rk} implies
$$\cP = \{ p\in\Projk : k\in \NN_0, p\not\in(\nu_k) \}.$$
\end{remark}

We are ready to prove \Cref{thm::main_thm_2}, which reduces the computation of indecomposable objects in $\RepCt$ to the computation of equivalence classes of projective partitions. Let us denote the isomorphism classes of irreducible complex representations of a group $G$ by $\Irr(G)$.

\begin{definition} \label{def::LpV} For $k\in \NN_0$, $p\in\ProjC(k)$, and $V\in\Irr(S(p))$, let $e'$ be an idempotent in $\CC S(p)$ whose generated left ideal is in the isomorphism class $V$ of $S(p)$-modules. Let $e\in p\End([k])p/I_p$ be the image of $e'$ under the isomorphism $\CC S(p)\cong p\End([k])p/I_p$ from \Cref{lem::FW} and let $\L_p(V)$ be the isomorphism class of the object $([k],\L_p(e))$ in $\RepCt$.
\end{definition}

We note that in the definition, the idempotents $e'$ and $e$ are defined only up to conjugation in the respective algebras, so $\L_p(V)$ is defined only as an isomorphism class of objects in $\RepCt$.

\begin{theorem} \label{thm::indecompsable_obj_by_A_k}
 Let $\cC$ be a category of partitions and $t\in \CC\backslash \{0\}$. Then the mapping $\Irr(S(p))\ni V\mapsto\L_p(V)$ induces a bijection
 \begin{align*}
    \bigsqcup_{[p]\in \cP/ \sim} \Irr(S(p)) \longleftrightarrow \left\{ \begin{matrix} \text{isomorphism classes of non-zero} \\ \text{indecomposable objects in }  \RepCt \end{matrix} \right\}.
    \end{align*}
\end{theorem}

\begin{proof}
By \Cref{lem::proj_part_final_corr} we have a bijection 
\[ \L: \bigsqcup_{[p]\in \Projk/ \sim} \Lambda(p E p / I_p ) \to \Lambda(E),
~ e\mapsto \L_p(e), \]
where $E:=\End([k])$, as before. The isomorphisms classes of non-zero indecomposable objects in $\RepCt$ are in bijection with the conjugacy classes of primitive idempotents in $E$ for which $([k],e)$ is not a subobject of $[l]$ for any $l<k$ by \Cref{lem::lem_Krull-Schmidt}. Thus by \Cref{lem::Projk_Prokl} we have a bijection
 \begin{align*}
    \bigsqcup_{[p]\in \cP/ \sim} \Lambda(p E p / I_p ) \longleftrightarrow \left\{ \begin{matrix} \text{isomorphism classes of non-zero} \\ \text{indecomposable objects in }  \RepCt \end{matrix} \right\}.
    \end{align*}
By \Cref{lem::FW} the algebra $p E p / I_p $ is isomorphic to the group algebra $\CC S(p)$ for any $p\in \Projk$. Finally, the primitive idempotents of a complex group algebra up to conjugation correspond to the irreducible complex representations of the group, where any primitive idempotent generates an irreducible subrepresentation inside the (semisimple) regular representation.
\end{proof}

\begin{example} \label{ex::P}
In $\uRep(P,t)=\uRep(S_t)$, we have the decomposition $p = p_0^* \id_{t(p)} p_0$ for any projective partition $p=p_0^*p_0\in\operatorname{Proj}_P(k)$. Thus $p\in \cP$ if and only if $p=\id_k$. Now $S(\id_k)=S_k$, the full symmetric group, and the indecomposables are parametrised by Young diagrams of arbitrary size. This reproduces the known results from \cite{De07, CO11} (see also Halverson and Ram's survey on partition algebras \cite{HR04}) in this case.
\end{example}

More examples will be considered in \Cref{sec-examples}.

\subsection{Grothendieck rings}

\Cref{thm::indecompsable_obj_by_A_k} yields a description of the (additive) Grothendieck group of the additive category $\RepCt$. Since the latter category also has a monoidal structure, we want to extend this to a description of the Grothendieck ring.

\newcommand{\Ind}{\operatorname{Ind}}
\newcommand{\Res}{\operatorname{Res}}
\newcommand{\Hom}{\operatorname{Hom}}

Let $\ProjC:=\bigsqcup_{k\geq0}\Projk$ be the set of projective partitions in $\cC$. We observe that $\ProjC$ is a semigroup with the operation $\otimes$ and the identity element being the empty partition $p_0\in\cC(0,0)$. We also observe that for any $p,q\in\ProjC$, we have an embedding $S(p)\times S(q)\to S(p\otimes q)$. For each $p\in\ProjC$, let us denote the Grothendieck group of $\Rep(S(p))$ by $K(S(p))$, that is, $K(S(p))$ is the abelian group whose elements are isomorphism classes $[V]$ of (complex) virtual $S(p)$ representations with the operation $[V]+[W]=[V\oplus W]$ for any two $S(p)$ representations $V,W$.

\newcommand\simdot{\mathrel{\dot\sim}}
\begin{remark} 
Recall that the equivalence relation $\sim$ induces an equivalence relations on $\ProjC$ such that two projective partitions can be equivalent only if they are both elements in $\Projk$ for some $k\in \NN_0$. We can drop this restriction and define an equivalence relation on $\ProjC$ by setting $p\simdot q$ if there is a partition $r\in\cC$ such that $r\cdot r^*=p$ and $r^*\cdot r=q$ (see \cite[Def.~4.17]{FW16}). Then the equivalence classes $\cP/{\sim}$ are in bijection with the equivalence classes $\ProjC/{\simdot}$. This follows from \Cref{lem::equivalent-equivalences} as in \Cref{prop::equivalence-half}. Note that the semigroup operation $\otimes$ induces one on the equivalence classes $\ProjC/{\simdot}$ and hence on the equivalence classes $\cP/{\sim}$.
\end{remark}

\begin{definition} We define the ring
$$R
:=\bigoplus_{[p]\in \cP/{\sim}} K(S(p))
$$
with the multiplication
$$
[V]\cdot[W] := \Ind_{S(p)\times S(q)}^{S(p\otimes q)} (V\boxtimes W)
$$
for all $V\in\Rep(S(p))$ and $W\in\Rep(S(q))$, with the identity element corresponding to the one-dimensional representation of the trivial group $S(p_0)$.
\end{definition}

Recall that we can associate to any partition $p$ the number of through-blocks $t(p)$, which we have used to define filtrations $(E_T(k))_{T\geq0}$ on the endomorphism spaces $\End([k])$ in $\RepCt$ (\Cref{def::through-blocks}). These filtrations in turn can be used to define a filtration on the Grothendieck ring $K(\cC,t)$, as we will see in the following.

\begin{definition} \label{def::filtration} For any indecomposable object $X$ in $\RepCt$, let us define
$$
t(X) := \min \{T\geq0: \exists k\geq0, e\in E_T(k): e^2=e, X\cong([k],e) \} - 1.
$$
\end{definition}

By definition, $t(X)$ is constant across isomorphism classes of indecomposable objects. 

Let us denote the Grothendieck ring of $\RepCt$ by $K(\cC,t)$. As it has a $\ZZ$-basis given by the isomorphism classes of indecomposable objects in $\Rep(\cC,t)$, the numbers $t(X)$ for indecomposable objects $X\in\Rep(\cC,t)$ uniquely induce a filtration of $K(\cC,t)$ as a free $\ZZ$-module. 

\begin{lemma} The degrees defined in
\Cref{def::filtration} induce a filtration of $K(\cC,t)$ as a ring.
\end{lemma}

\begin{proof} Let $X_1$ and $X_2$ be two indecomposable objects in $\RepCt$, let $e_i\in E_{T_i}(k_i)$ be an idempotent such that $([k_i],e_i)\cong X_i$ for $i=1,2$. Then $e_1\otimes e_2$ is an idempotent in $E_{T_1+T_2}(k_1+k_2)$ such that $([k_1+k_2],e_1\otimes e_2)\cong X_1\otimes X_2$.
\end{proof}

We obtain the following analogue of \cite[Prop.~5.11]{De07}, a description of the associated graded of the Grothendieck ring for $\uRep(S_t)$.

\newcommand\gr{\operatorname{gr}}
\begin{proposition} \label{Grothendieck-ring} Let $\cC$ be a category of partitions and $t\in \CC \backslash \{0\}$. Then the mapping $\L$ induces a ring isomorphism between $R$ and the associated graded ring $\gr K(\cC,t)$.
\end{proposition}

\begin{proof} \Cref{thm::indecompsable_obj_by_A_k} means that $\L$ induces a bijection of abelian groups. 

Consider $V_i\in\Irr(S(p_i))$ for $i\in\{1,2\},p_i\in\ProjC$. Let $e_i$ be the primitive idempotents in $\CC S(p_i)$ corresponding to $V_i$. Then the tensor product of the objects corresponding to $V_i$ in $\RepCt$ are the image of the tensor product of the idempotent lifts of the $e_i$. Modulo lower order terms in the filtration, they correspond to the idempotent 
$$
e:= e_1\otimes e_2 \in \CC S(p_1)\otimes \CC S(p_2) \subset \CC S(p_1\otimes p_2)
.
$$
Let $(V_\lambda)_\lambda$ be a set of isomorphism classes of irreducible complex representations for $S(p_1\otimes p_2)$, with corresponding primitive idempotents $(e_\lambda)_\lambda$ in the group algebra. Then $e$ decomposes as a linear combination
$e = \sum_\lambda n_\lambda e_\lambda$
with multiplicities $(n_\lambda)_\lambda$, where 
\begin{align*}
    n_\lambda
 &= \dim \Hom_{S(p_1)\times S(p_2)}(\Res_{S(p_1)\times S(p_2)} V_\lambda, V_1\boxtimes V_2)
 \\
 &= \dim \Hom_{S(p_1\otimes p_2)}(V_\lambda, \Ind_{S(p_1)\times S(p_2)}^{S(p_1\otimes p_2)} V_1\boxtimes V_2)
.
\end{align*}
This shows that the structure constants of the multiplication coincide in the two rings considered.
\end{proof}

We note that the ring $R$ does not depend on $t$ and the Grothendieck ring of $\RepCt$ can be viewed as a filtered deformation of $R$ with deformation parameter $t$.

\begin{remark} We also note that the the operation $p\mapsto p\otimes\Paa$ for a projective partition $p$ defines a partial order and yields an embedding $S(p)\to S(p\otimes\Paa)$ which turns the groups $(S(p))_{p\in\cP}$ into an inverse system (whose underlying poset, however, might not be directed in general). For $\cC=P$, this is the system of all symmetric groups $S_0\subset S_1\subset S_2\subset \dots$.
\end{remark}

\newcommand\wRepCt{\widehat{\RepCt}}

\subsection{Semisimplification}
Let us consider now a group-theoretical category of partitions $\cC$, and let us recall (\Cref{thm-grouptheo-semisimple}) that $\RepCt$ is not semisimple if and only if $t\in\NN_0$. In the non-semisimple case, the semisimplification functor is given by \Cref{prop::fiber-functor}. For $t=0$, the semisimplification is trivial by \Cref{lem::semisimplification-t0}, so let us consider $t\geq 1$. In this case, we record some general observations about the semisimplification $\wRepCt$. For any $k\geq0$, $p\in\Projk$, and $V$ in $\Irr(S(p))$, let us denote the primitive idempotent in $\cC(k,k)$ corresponding to the indecomposable object $\L_p(V)$ according to \Cref{thm::indecompsable_obj_by_A_k} by $e_{k,p,V}$.

\begin{lemma} If $t\in\NN$, then $\L$ together with the quotient functor $\RepCt\to\wRepCt$ yields a bijection
 \begin{align*}
    \mathcal{V}\longleftrightarrow \left\{ \begin{matrix} \text{isomorphism classes of non-zero} \\ \text{indecomposable objects in } \widehat{\RepCt} \end{matrix} \right\},
    \end{align*}
where $\mathcal{V}$ is the set of isomorphism classes of those $V\in\Irr(S(p))$ for $k\geq0$, $[p]\in\Projk/{\sim}$, $p\not\in(\nu_k)$, whose associated idempotent $e_{k,p,V}$ decomposes into a sum of primitive idempotents $(e_i)_i$ in $P(k,k)\supset \cC(k,k)$ at least one of which has non-zero trace. 
\end{lemma}

\begin{proof} By general results on the semisimplification (see \cite[Thm.~2.6]{EO18} or \Cref{rem::negl_morphisms}), the quotient functor induces a bijection between the isomorphism classes of indecomposable objects of non-zero dimension in the original category and the isomorphism classes of non-zero indecomposable objects in the semisimplification. 

By \Cref{thm::indecompsable_obj_by_A_k}, the isomorphism classes of indecomposable objects in $\RepCt$ are given by the isomorphism classes of the objects $\L_p(V)$ for isomorphism classes $V$ of complex $S(p)$-representations, for $k\geq0$ and $[p]\in\Projk/{\sim}$ such that $p\neq(\nu_k)$. Hence, the indecomposable objects in $\wRepCt$ are given by the subset of those $V$ for which $\L_p(V)$ has non-zero dimension.

The dimension of $\L_p(V)$ in $\RepCt$ can be computed by decomposing the relevant idempotent $e_{k,p,V}$ as a sum of primitive idempotents in the containing category $\uRep(P,t)=\uRep(S_t)$,
$$
 e_{k,p,V} = \sum_i e_i .
$$
Then the dimension of $\L_p(V)$ is the trace of $e_{k,p,V}$ in $\RepCt$ or, equivalently, in $\uRep(P,t)$, but the latter is the sum of the traces of the idempotents $e_i$. 

However, if an idempotent $e_i$ has a non-zero trace, then the indecomposable object given by its image corresponds to an indecomposable object in $\Rep(S_t)$, which is the semisimplification of $\uRep(P,t)$, hence such a trace is the dimension of an irreducible complex $S_t$-module, and in particular, a positive integer. So the traces of the idempotents $e_i$ can only be non-negative integers.
\end{proof}

This allows us to describe at least a part of the semisimplification $\wRepCt$ uniformly for all group-theoretical $\cC$.

\begin{proposition} If $t\in\NN$, then there is a unique isomorphism class of non-zero indecomposable objects in $\wRepCt$ for each isomorphism class in $\Irr(S(p))$ for all $p\in\cP$ with $t(p)\leq t/2$, i.e.~$p$ has at most $t/2$ through-blocks.
\end{proposition}

\begin{proof} We record that if an idempotent $e$ in any ring lies in an ideal $I$, then any orthogonal decomposition consists of idempotents which are divisible by $e$, and hence, also contained in $I$.

Taking $I$ to be the ideal spanned by all partitions with at most $t/2$ through-blocks in $P(k,k)$ implies that decomposing the idempotent for some $V\in\Irr(S(p))$ in $P(k,k)$ results in a sum of primitive idempotents all of which have at most $t/2$ through-blocks. Such primitive idempotents have non-zero traces by the description of the negligible primitive idempotents in $\uRep(S_t)$ in \cite[Rem.~3.25]{CO11}. 
\end{proof}

\newcommand\SurC{\operatorname{Sur}_\cC}

\subsection{An alternative description} Instead of using projective partitions, we note that one could alternatively consider their ``upper halves'', that is, the partitions $p_0$ appearing in through-block factorisations $p=p_0^*p_0$ (\Cref{def::GroupsSp}) of projective partitions $p$. Let us explain how this yields an equivalent description of indecomposable objects in interpolation partition categories.

Let $\cC$ be any category of partitions.

\begin{definition} A partition $q\in P(k,l)$ is called \emph{surjective} if $t(q)=l$, i.e.~$q$ has exactly $l$ through-blocks. For $k\in\NN_0$, we set
$$
\SurC(k) := \{ q\in P(k,l): l\in\NN_0, t(q)=l, q^*q\in\cC \}
 .
$$
\end{definition}

Note that $k\geq t(q)=l$ for any surjective partition $q\in\cC(k,l)$.

Surjective partitions are helpful in generalising the concept of equivalence for partitions (see \Cref{def::equivprojpart}) even across endomorphism algebras $\End([k])$ for varying $k$. 

\begin{lemma} \label{lem::equivalent-equivalences}  Consider projective partitions $p,q\in\cC$ and surjective partitions $p_0,q_0\in P$ such that $t(p_0)=t(p)$, $t(q)=t(q_0)$, $p=p_0^*\cdot p_0$, and $q=q_0^*\cdot q_0$. Then the following are equivalent: \\
(1) There is a partition $r\in\cC$ such that $p=r\cdot r^*$ and $q=r^*\cdot r$. \\
(2) There is a partition $r\in\cC$ such that $p=r\cdot q\cdot r^*$ and $q=r^*\cdot p\cdot r$. \\
(3) There is a partition $r\in\cC$ and a permutation $s$ (regarded as a partition) such that $p_0 = s\cdot q_0\cdot r^*$ and $q_0 = s^*\cdot p_0\cdot r$.
\end{lemma}

\begin{proof} (1) $\Rightarrow$ (2): As $p$ is projective,
$$ p = p\cdot p = r\cdot r^*\cdot r\cdot r^* = r\cdot q\cdot r^* ,
$$
and similarly for $q$.

(2) $\Rightarrow$ (3): As the number of through-blocks of a product is at most that of any factor, $t(p)=t(q)$. Now
$$ p_0^* \cdot p_0 = p = r\cdot q\cdot r^* = (r\cdot q_0^*)\cdot (q_0\cdot r^*)
$$
are two different factorisations of the partition $p$ through $t(p)$ points. So there is a permutation $s$ such that $p_0 = s\cdot q_0\cdot r^*$. Similarly, there is a permutation $s'$ such that $q_0 = s'\cdot p_0\cdot r$. But then 
$$ p_0 = s\cdot s'\cdot p_0\cdot r\cdot r^*
\Rightarrow (s\cdot s')^{-1} = (p_0\cdot r)\cdot (p_0\cdot r)^*
$$
which shows that $s\cdot s'$ is the identity permutation, due to the horizontal symmetry of the right-hand side.

(3) $\Rightarrow$ (1): We have
$$
 p = p_0^*\cdot p_0 = r\cdot q_0^*\cdot s^*\cdot s\cdot q_0\cdot r^* = r\cdot q\cdot r^*,
$$
so
$$ (p\cdot r\cdot q)\cdot(p\cdot r\cdot q)^* 
= p\cdot r\cdot q\cdot q^*\cdot r^*\cdot p^*
= p\cdot p\cdot p = p ,
$$
and similarly $(p\cdot r\cdot q)^*\cdot(p\cdot r\cdot q)=q$.
\end{proof}

\begin{definition} A surjective partition $q\in\SurC(k)$ is called \emph{indecomposable surjective} if $q\notin \SurC(k')\cdot b$ for all $k'<k$ and all $b\in\cC(k,k')$. We set
$$
\cQ := \{ q\in\SurC(k): k\in\NN_0, q\text{ indecomposable}\}
 .
$$
\end{definition}

\begin{definition} \label{def::equivalence-surjectives} Two surjective partitions $q,q'\in P(k,l)\cap\SurC(k)$ are called \emph{equivalent} if there are partitions $r\in\cC(k,k)$, $s\in P(l,l)$ such that
$$
q = s\cdot q'\cdot  r^*
\quad\text{and}\quad
q' = s^*\cdot  q\cdot  r
 .
$$
\end{definition}

We observe that any $s$ as in the definition must have $t(s)=l$ through-blocks, so it must be a permutation and $s^*=s^{-1}$.

\begin{lemma} \Cref{def::equivalence-surjectives} defines equivalence relations $\sim$ on $\SurC(k)$ for all $k\in \NN_0$ and on $\cQ$.
\end{lemma}

\begin{proof} For the sets $\SurC(k)$, this can be verified directly. Moreover, for $q\in\SurC(k)\cap\cQ$ and an equivalent $q'\in\SurC(k)$, we see that $q'\in\cQ$, as well.
\end{proof}

\begin{definition} For any $q\in \SurC(k)$, we define the set
$$
S_{1/2}(q) := \{\sigma\in S_{t(q)}: q^* \sigma q\in\cC \}
 .
$$
\end{definition}

\begin{lemma} $S_{1/2}(q)$ is a subgroup of $S_l$ which only depends on the equivalence class of a surjective partition $q$.
\end{lemma}

\begin{proof} This can be checked as in \Cref{ssec::projectives}.
\end{proof}

\begin{proposition} \label{prop::equivalence-half} The mapping $q\mapsto q^*q$ induces a bijection between $\cQ/{\sim}$ and $\cP/{\sim}$, and $S_{1/2}(q) = S(q^*q)$ for each $q\in\cQ$.
\end{proposition}

\begin{proof} This follows from the definitions and \Cref{lem::equivalent-equivalences}.
\end{proof}

From \Cref{thm::indecompsable_obj_by_A_k} we obtain immediately:
\begin{corollary} The indecomposables in $\RepCt$ are parametrised by the irreducible complex representations of the system of finite groups $(S_{1/2}(q))_{q\in\cQ}$. 
\end{corollary}

Compared to the set of projective partitions $\cP$, the set $\cQ$ contains their possible (upper) halves, so the partitions in $\cQ$ are potentially smaller. However, various ``upper halves'' can produce the same projective partition, which is reflected in the slightly more complicated equivalence relation.

Beyond providing an alternative approach to the description of indecomposable objects in interpolating partition categories, the set $\cQ$ can be interpreted naturally in the more general framework of Knop's tensor envelopes (\cite{Kn07}). Such a generalisation is part of an ongoing research project.
\section{Indecomposable objects for some concrete examples} \label{sec-examples}

In this section, we compute concrete parameterisations for the indecomposable objects in $\RepCt$ up to isomorphism for all categories of partitions $\cC$ which either contain the partition $\cross$ or in which all partitions are non-crossing. Recall \Cref{ssec::categories-of-partitions} for the classification of these categories of partitions; for the corresponding easy quantum groups see for instance \cite[Sec.~2]{We13}. We conclude the section with a comparison of our results with the well-known theory of Temperley--Lieb categories.

\subsection{Indecomposable objects in 13 interpolating partition categories}

The categories of partitions $\cC$ with $\cross \in \cC$ are exactly the following six categories and the corresponding easy quantum groups $G_n(\cC)$, $n\in \NN_0$, are all given by compact matrix groups. Note that for these categories, we have an embedding of the symmetric group $S_k$, for all $k\geq0$, into the endomorphism algebra of the object $[k]\in\RepCt$ by sending the transpositions $(i, i+1)$ to $\Paa^{\otimes(i-1)}\otimes\cross\otimes\Paa^{\otimes(k-1-i)}$ for all $1\leq i<k$. Then for all $k,l\geq0$, the image of the permutation $(1, k+1)\dots(l, k+l)$ defines a symmetric braiding for the objects $[k]$ and $[l]$ in $\RepCt$, which extends uniquely to a symmetric braiding of $\RepCt$. 

\begin{enumerate}[label=(\roman*)]
    \item $P$ is the category of all partitions and corresponds to the symmetric groups $G_n(P)=S_n$.
    \item $P_\even$ is the category of partitions with even block size and corresponds to the hyperoctahedral groups $G_n(P_\even)=H_n=S_2 \wr S_n$.
    \item $P_2$ is the category of partitions with only blocks of size two and corresponds to the orthogonal groups $G_n(P_2)=O_n$. $\RepCt$ is the Karoubian envelope of the Brauer category with parameter $t$ in this case.
    \item $P':=\langle \cross, \singleton\ot \singleton, \vierpart \rangle$ is the category of partitions with an even number of blocks of odd size and corresponds to the modified symmetric groups $G_n(\langle \cross, \singleton\ot \singleton, \vierpart \rangle)=S_n \times S_2=:S_n'$.
    \item $P_b:=\langle \cross, \singleton \rangle$ is the category of partitions with blocks of size one or two and corresponds to the bistochastic groups $G_n(\langle \cross, \singleton \rangle)=B_n$.
    \item $P'_b:=\langle \cross, \singleton\ot \singleton \rangle = \langle \cross, \legpart \rangle$ is the category of partitions with an arbitrary number of blocks of size two and an even number of blocks of size one and corresponds to the modified bistochastic groups $G_n(\langle \cross, \singleton\ot \singleton \rangle)=B_n \times S_2=:B_n'$.
\end{enumerate}

The categories of partitions $\cC$ which contain only non-crossing partitions are precisely the following seven categories and all of them correspond to so-called free quantum groups.
\begin{enumerate}[label=(\roman*)]
    \item $NC$ is the category of all non-crossing partitions and corresponds to the free symmetric quantum groups $G_n(NC)=:S_n^+$.
    \item $NC_\even$ is the category of non-crossing partitions with even block size and corresponds to the hyperoctahedral quantum groups $G_n(NC_\even)=:H_n^+$.
    \item $NC_2$ is the category of non-crossing partitions with only blocks of size two and corresponds to the free orthogonal quantum groups $G_n(NC_2)=:O_n^+$.
    \item $NC':=\langle \singleton\ot \singleton, \vierpart \rangle$ is the category of non-crossing partitions with an even number of blocks of odd size and corresponds to the modified symmetric quantum groups $G_n(\langle \singleton\ot \singleton, \vierpart \rangle)=:S_n'^+$.
    \item $NC_b:=\langle \singleton \rangle$ is the category of non-crossing partitions with blocks of size one or two and corresponds to the bistochastic quantum groups $G_n(\langle \singleton \rangle)=B_n^+$,
    \item $NC^\#_b:=\langle \singleton\ot \singleton \rangle$ is the category of non-crossing partitions with an arbitrary number of blocks of size two and an even number of blocks of size one such that the number of points between any two connected points is even, considering all points naturally arranged in a circle, i.e. for instance the upper left point is next to the second left upper point and the lower left point. The corresponding quantum groups $G_n(\langle \legpart \rangle)=:B_n^{\#+}$ are called the freely modified bistochastic quantum groups.
    \item $NC'_b:=\langle \legpart \rangle$ is the category of non-crossing partitions with an arbitrary number of blocks of size two and an even number of block of size one and corresponds to the modified bistochastic quantum groups $G_n(\langle \legpart \rangle)=:B_n'^+$.
\end{enumerate}

In the following we will apply \Cref{thm::indecompsable_obj_by_A_k} to derive an explicit parametrisation of the indecomposable objects in $\RepCt$ up to isomorphism for all these categories. Recall that we have to determine equivalence classes of projective partitions in 
$$
\cP := \{ p\in\Projk: k\geq0, p\not\in b^* \Projl b \text{ for all } 0\leq l<k, b\in \cC(k, l) \} 
$$
and we have to describe the irreducible complex representations of the groups $S(p)$ associated to representatives of these equivalence classes. Then we obtain the indecomposable objects in $\RepCt$ by applying the mapping $\Irr(S(p))\ni V\mapsto\L_p(V)$ of \Cref{thm::indecompsable_obj_by_A_k}. For the remainder of this section, we denote the trivial complex representation $\CC$ of the trivial group by $V_{\mathrm{triv}}$.

We already considered the case $\cC=P$ in \Cref{ex::P}. The following lemma shows that some categories of partitions behave similarly. Recall that for $n\in \NN_0$, the inequivalent irreducible complex representations of the symmetric group $S_n$ can be indexed by Young diagrams of size $n$ (see for instance \cite[Sec.~4.2]{FH91}) and for any Young diagram $\lambda$ of size $|\lambda|=n$, we denote the corresponding equivalence class of irreducible complex representations by $V_\lambda \in \Irr (S_n)$.

\begin{lemma} \label{lem::ex0}
Let $\cC \in \{P,P_2,P_b, NC, NC_2, NC_b \}$ and $t\in \CC\backslash \{0\}$. Then $\cP = \{ \id_k \mid k\in \NN_0 \}$.

If $\cC \in \{P, P_2,P_b\}$, then $S(p)=S_{t(p)}$ for all $p\in \ProjC$ and the map
\begin{align*}
 \phi: \left\{ \begin{matrix} \text{Young diagrams } \lambda \\ \text{of arbitrary size} \end{matrix} \right\} \to \left\{ \begin{matrix} \text{isomorphism classes of non-zero} \\ \text{indecomposable objects in }  \RepCt \end{matrix} \right\}, ~ \lambda \mapsto \L_{\id_{|\lambda|}} (V_\lambda) 
\end{align*}
is a bijection.

If $\cC \in \{NC, NC_2, NC_b\}$, then $S(p)=\{\id\}$ for all $p\in \ProjC$ and the map
\begin{align*}
 \phi: \NN_0 \to \left\{ \begin{matrix} \text{isomorphism classes of non-zero} \\ \text{indecomposable objects in }  \RepCt \end{matrix} \right\}, ~ k \mapsto \L_{\id_k} (V_{\mathrm{triv}})
\end{align*}
is a bijection.
\end{lemma}

\begin{proof}
Consider a projective partition $p=p_0^*p_0\in \Projk$. It is easy to check that one can choose $p_0\in \cC$ and hence $p = p_0^* \id_{t(p)} p_0$ yields a decomposition of $p$ in $\cC$. Hence $p\in \cP$ if and only if $p=\id_k$ for some $k\in\NN_0$.

We described the structure of $S(p)$ for all $p\in \ProjC$ in \Cref{ex::FW}. If $\cross \in \cC$, then $S(p)=S_{t(p)}$ for all $p\in \ProjC$ and by \Cref{thm::indecompsable_obj_by_A_k} the indecomposables in $\RepCt$ up to isomorphism are in bijection with $\bigsqcup_{k\in \NN_0} \Irr (S_k)$ and hence with Young diagrams of arbitrary size. 

If $\cross \notin \cC$, then $S(p)=\{\id\}$ for all $p\in \ProjC$ and by \Cref{thm::indecompsable_obj_by_A_k} the indecomposables in $\RepCt$ up to isomorphism are in bijection with $\bigsqcup_{k\in \NN_0} \Irr (\{ \id \})$ and hence with $\NN_0$. 
\end{proof}

\begin{remark}
Our description reproduces the known results for $\uRep (O_t)=\uRep(P_2,t)$ from \cite{De07, CH17} (see also Wenzl's original article on the Brauer algebras \cite{Wenz87-brauer}).

Moreover, our description for the Temperley--Lieb category $\uRep(O^+_t)=\uRep(NC_2,t)$ reproduces known results as the indecomposable objects can be explicitly described using Jones--Wenzl idempotents. In \Cref{subs::Ot+andSt+} we will study this in more detail to obtain an explicit description of the indecomposables of $\uRep(S^+_t)=\uRep(NC,t)$.
\end{remark}

\medskip
Next we consider the categories $\uRep(S_t')=\uRep(P',t)$ and $\uRep (B_t')=\uRep(P'_b,t)$, and their free versions $\uRep(S_t'^+)=\uRep(NC',t)$ and $\uRep (B_t'^+)=\uRep(NC'_b,t)$.

\begin{lemma} \label{lem::ex1}
Let $\cC\in \{P', P_b', NC', NC_b' \}$. Then
\begin{align*}
\cP = \{p \in \Projk \mid~ &k\geq0, ~t(p)\geq k-1\} . 
\end{align*}
\end{lemma}

\begin{proof}
Note that any partition $p\in \cC$ has an even number of blocks of odd size.

Let $p\in \Projk$ with $t(p)<k-1$. Then there exist two upper points of $p$ none of which lies in a through-block of size two. Let $b\in P(k-2,k)$ be the partition that arises from $p$ by removing these points. If both considered points are in the same block of $p$, then removing them does not change the parity of the size of the block. Otherwise, removing them changes the parity of two blocks. In both cases we obtain a partition that still has an even number of blocks of odd size. Moreover, $b$ is a non-crossing partition if $p$ is non-crossing, and $b$ has only blocks of size one or two if $p$ does. Hence $b\in \cC$. As $p$ is projective, we have $p=b \cdot b^*$. We set $q:=b^* \cdot b\in \ProjC(k-2)$ and it follows that $p=p\cdot p=b\cdot b^*\cdot b\cdot b^*=b\cdot q\cdot b^* \notin \cP$.

Now let $p\in \Projk$ with $t(p)\geq k-1$. If $t(p)=k$, then $p=\id_k \in \cP$, so assume $t(p)=k-1$. Let us also assume that $p\notin \cP$. Then $p=b^* q b$ for some $0\leq l<k, b\in \cC(k, l), q\in \Projl$. Since the number of through-blocks of a composition of partitions is less than or equal to the number of through-blocks of each composed partition, it follows that $l=k-1$, $q=\id_{k-1}$, and $t(b)=k-1$. But then $b$ has a total of $k$ or $k-1$ blocks exactly one of which is of odd size, which is a contradiction.
\end{proof}

Both the modified symmetric groups $S_n'=S_n\times S_2$ and the modified bistochastic groups $B_n'=B_n\times S_2$ are direct products involving $S_2$. Hence their irreducible representations are of the form $\rho \times \varphi$, where $\rho$ is an irreducible representation of $S_n$ or $B_n$, respectively, and $\varphi$ is one of the two irreducible representations of $S_2$. In the following we will see that the indecomposables in the corresponding interpolating partition categories have an analogous structure.
 
\begin{lemma} 
Let $\cC\in \{P', P_b'\}$ and $t\in \CC\backslash \{0\}$. Then $S(\id_k)=S(\id_k \ot \bbar)=S_k$ for all $k\in \NN_0$ and the map
\begin{align*}
 &\phi: \left\{ \begin{matrix}  \text{Young diagram} \\
 \text{of arbitrary size} \end{matrix} \right\} \times \{1,-1\}
 \longrightarrow
 \left\{ \begin{matrix} \text{isomorphism classes of non-zero} \\ \text{indecomposable objects in }  \RepCt \end{matrix} \right\}\\
 &\text{with } \phi(\lambda,1) = \L_{\id_{|\lambda|}} (V_\lambda) \text{ and }  
 \phi(\lambda,-1) = \L_{\id_{|\lambda|} \ot \bbar} (V_\lambda) 
\end{align*}
is a bijection.
\end{lemma}

\begin{proof}
We show that $E:=\{\id_k \mid k\in\NN_0\} \cup \{ \id_k \ot \bbar \mid k\in \NN_0 \}$ is a set of representatives for all equivalence classes of projective partitions in $\cP$. Since equivalent projective partitions have the same number of through-blocks, $\id_k$ and $\id_{k-1} \ot \bbar$ are not equivalent for any $k\in \NN$ and hence all partitions of $E$ lie in different equivalence classes. 

Consider a partition $p\in \cP$. By \Cref{lem::ex1} we have $p=\id_k$ or $t(p)=k-1$, let us assume the latter. Since $\cross \in \cC$, it is easy to check that $p$ is either equivalent to $\id_{k-1} \ot \bbar$ or $\id_{k-2} \ot \vierpart$. In the latter case $\cC=P'$, but $\id_{k-1} \ot \bbar=r\cdot r^*$ and $\id_{k-2} \ot \vierpart=r^*\cdot r$ with $r=\id_{k-2} \ot \Paaab$, hence these partitions are equivalent. 

Thus $E$ is a set of representatives for all equivalence classes of projective partitions in $\cP$ and we have $S(\id_k)=S_k$ and $S(\id_k \ot \bbar)=S_k$ for all $k\in \NN_0$. Then by \Cref{thm::indecompsable_obj_by_A_k}, the indecomposables in $\RepCt$ up to isomorphism are in bijection with $\bigsqcup_{k\in \NN_0} (\Irr (S_k) \sqcup \Irr (S_k))$ and hence the claim follows.
\end{proof}

\begin{lemma} 
Let $\cC\in \{NC', NC_b' \}$ and $t\in \CC\backslash \{0\}$. Then the map
\begin{align*}
 &\phi: \NN_0 \times \{1,-1\} 
 \longrightarrow
 \left\{ \begin{matrix} \text{isomorphism classes of non-zero} \\ \text{indecomposable objects in }  \RepCt \end{matrix} \right\}\\
 &\text{with } \phi(k,1) = \L_{\id_k} (V_{\mathrm{triv}}) \text{ and }
 \phi(\lambda,-1) = \L_{\id_k \ot \bbar} (V_{\mathrm{triv}})
\end{align*}
is a bijection.
\end{lemma}

\begin{proof}
We show again that $E:=\{\id_k \mid k\in\NN_0 \} \cup \{ \id_k \ot \bbar \mid k\in \NN_0 \}$ is a set of representatives for all equivalence classes of projective partitions in $\cP$. As in the previous proof all partitions of $E$ lie in different equivalence classes. 

Let $k\in \NN$. We consider the partitions $\id_l \ot \bbar \ot \id_{k-1-l} \in \cP$ with $0\leq l\leq k-1$ and claim that they are all equivalent. Let $0\leq l<l'\leq k-1$ and we define $r:=\id_l \ot \La \ot \id_{l'-l} \ot \Ua \ot \id_{k-1-l'} \in \cC(k,k)$. Then we have $\id_l \ot \bbar \ot \id_{k-1-l}=rr^*$ and $\id_{l'} \ot \bbar \ot \id_{k-1-l'}=r^*r$ and hence these partitions are equivalent. 

Now, if $\cC=NC_b'=\langle \legpart \rangle$, then any partition in $\Projk \cap \cP$ is either $\id_k$ or $\id_l \ot \bbar \ot \id_{k-1-l}$ for some $0\leq l\leq k-1$ by \Cref{lem::ex1}. 
If $\cC=NC'=\langle \singleton\ot \singleton, \vierpart \rangle$, then $\Projk \cap \cP$ contains additionally the partitions $\id_l \ot \vierpart \ot \id_{k-2-l}$ with $0\leq l\leq k-2$. But as in the proof of the previous lemma they are all equivalent to the partitions $\id_l \ot \bbar \ot \id_{k-1-l}$ with $0\leq l\leq k-1$.
Thus, in both cases, $E$ is a set of representatives for all equivalence classes of projective partitions in $\cP$.

By \Cref{ex::FW} we have $S(p)=\{\id\}$ for all $p\in \ProjC$ and the claim follows from \Cref{thm::indecompsable_obj_by_A_k}.
\end{proof}

\medskip

Next, we consider the categories $\uRep(B_t^{\#+})=\uRep(NC_b^\#,t)$, which correspond to the freely modified bistochastic quantum groups.
Recall that $NC^\#_b$ is the category of non-crossing partitions with an arbitrary number of blocks of size two and an even number of blocks of size one such that the number of points between any two connected points is even, considering all points naturally arranged in a circle. 
We start by computing the set $\cP$ using the following definition.
\begin{definition}
For any $k\in \NN_0, l\in \NN_0$ and $m=(m_1,\ldots ,m_l)\in \underline{k}^l$ with $m_{i+1}-m_i\geq 2$ for all $i\in \underline{l-1}$, we define a partition $p_{k,m} \in NC_b^\#(k,k)$ such that any upper point is connected to the opposite lower point except for the $m_i$-th lower points and the $m_i$-th upper points, which are in blocks of size one, i.e. 
\[
p_{k,m} = \Partition{
\Pline (1,0) (1,1)
\Pline (4,0) (4,1)
\Pline (5,0) (5,0.25)
\Pline (5,0.75) (5,1)
\Pline (6,0) (6,1)
\Pline (9,0) (9,1)
\Pline (10,0) (10,0.25)
\Pline (10,0.75) (10,1)
\Pline (11,0) (11,1)
\Pline (16,0) (16,1)
\Pline (17,0) (17,0.25)
\Pline (17,0.75) (17,1)
\Pline (18,0) (18,1)
\Pline (21,0) (21,1)
\Ptext(2.6,0.5){$\cdots$}
\Ptext(7.6,0.5){$\cdots$}
\Ptext(13.5,0.1){$\cdots$}
\Ptext(19.6,0.5){$\cdots$}
\Ptext(5.1,-0.6){\footnotesize{$m_1$}}
\Ptext(10.1,-0.6){\footnotesize{$m_2$}}
\Ptext(17.1,-0.6){\footnotesize{$m_l$}}
} . \]
\end{definition}

\begin{lemma} \label{lem::ex2}
Let $\cC=NC_b^\#$. Then
\begin{align*}
\cP = \left\{ p_{k,m} ~ \middle\vert ~\begin{matrix} k\in \NN_0, l\in \NN_0,
m\in \underline{k}^l,  m_{i+1}-m_i\geq2 \end{matrix} \right\}. 
\end{align*}
\end{lemma}

\begin{proof}
Let $p\in \Projk$ be a projective partition which is not of the form $p_{k,m}$. Then $p$ has two upper or two lower points that are either in the same non-through-block, or adjacent points in blocks of size one. Replacing $p$ by $p^*$ if necessary, we may assume these points are upper points. Let $b\in P(k-2,k)$ be the partition that arises from $p$ by removing these points. By the above description of the category of partition $NC^\#_b$, we have $b\in NC^\#_b$ in both cases. As $p$ is projective, we have $p=b \cdot b^*$ and hence $p=p\cdot p=b\cdot b^*\cdot b\cdot b^*=b\cdot q\cdot b^* \notin \cP$ with $q:=b^* \cdot b\in \ProjC(k-2)$.

Now let $p=p_{k,m}$ be a partition in the set on the right-hand side of the assertion and we assume that $p\notin \cP$. Then $p=b^* q b$ for some $0\leq l<k, b\in \cC(k, l), q\in \Projl$. As $b\in NC^\#_b$ and $p=b^* q b$, $b$ has to arise from $p$ by removing some lower points of $p$ that are in blocks of size one. Consider the leftmost upper point $x$ of $b$ for which the opposite lower point has been removed from $p$. Then $p$ and $b$ are of the following form: 
$$p = q \ot 
\Partition{
\Pline (1,0) (1,1)
\Pline (2,0.65) (2,1)
\Pline (2,0.35) (2,0)
\Pline (3,0) (3,1) 
} 
\ot \ldots \quad \text{and} \quad
b = q \ot  
\Partition{
\Pline (1,0) (1,1)
\Pline (2,0.65) (2,1)
\Pline (2,0) (3,1) 
\Ptext(2,1.4){\footnotesize{$x$}}
} 
\ot \ldots \quad \text{for some } q\in NC^\#_b.$$
Now, consider the through-block of $b$ on the right of $x$. The number of points of $b$ on the left-hand side of this through-block is odd. Hence we have $b\notin NC^\#_b$, which is a contradiction.
\end{proof}

\begin{lemma} 
Let $t\in \CC\backslash \{0\}$. Then the map
\begin{align*}
 &\phi:  
\left\{ (k,m)~ \middle\vert ~\begin{matrix} k\in \NN_0, l\in \NN_0,
m\in \underline{k}^l, \\ m_{i+1}-m_i\geq 2 \end{matrix} \right\}
 \longrightarrow
 \left\{ \begin{matrix} \text{isomorphism classes of non-zero} \\ \text{indecomposable objects in } \uRep (B_t^{\#+}) \end{matrix} \right\} \\
 &\text{with } \phi((k,m)) = \L_{p_{k,m}} (V_{\mathrm{triv}})
\end{align*}
is a bijection.
\end{lemma}

\begin{proof}
We claim that the partitions in $\cP \cap \Projk$ are pairwise inequivalent for any $k\in \NN_0$. Assume that $p,q \in \cP\cap \Projk$ are equivalent projective partitions, i.e. there exists a partition $r\in NC^\#_b(k,k)$ with $p=r^*\cdot r$ and $q=r\cdot r^*$. By \Cref{lem::ex2}, $p$ and $q$ have the property that any block is either a block of size one or a through-block of size two. This completely determines the structure of $r$: Any block of $r$ is either a block of size one or a through-block of size two. Moreover, an upper point of $r$ is in a block of size one if and only if the corresponding upper point of $p$ is in a block of size one. Similarly, a lower point of $r$ is in a block of size one if and only if the corresponding lower point of $q$ is in a block of size one. Now, since $r\in NC^\#_b$, the number of points between any two connected points of $r$ is even. It is straightforward to check that this implies $p=q$.

Now, since $S(p)=\{\id\}$ for all $p\in \ProjC$ by \Cref{ex::FW}, the claim follows from \Cref{thm::indecompsable_obj_by_A_k}.
\end{proof}

\medskip

To study the structure of the indecomposable objects in $\uRep (H_t)=\uRep (P_\even,t)$ and $\uRep (H_t^+)=\uRep (NC_\even,t)$, again we first determine $\cP$.

\begin{lemma} \label{lem::Ht(p)_Q}
Let $\cC \in \{P_\even ,NC_\even \}$ and $k\in \NN_{\geq 2}$. Then
\begin{align*}
\cP = \{p \in \Projk \mid~ &k\geq0, ~p \text{ has only through-blocks and any block} \\
&\text{has at most $2$ upper and at most $2$ lower points}\} . 
\end{align*}
\end{lemma}

\begin{proof}
Let $p\in \Projk$. At first we assume that $p$ has a block with more than $3$ upper points. Let $b\in P(k-2,k)$ be the partition that arises from $p$ by removing two of the upper points in the considered block. As $p$ is projective and the block has more than $3$ upper points, it follows that $p=b \cdot b^*$. Moreover, as $\cC$ contains all partitions or all non-crossing partitions with even block size, $b$ lies also in $\cC$. We set $q:=b^* \cdot b\in \ProjC(k-2)$ and thus we have $p=p^2=b \cdot b^* \cdot b \cdot b^*=b  \cdot q \cdot b^* \notin \cP$.

Now we assume that $p$ has an upper non-through block, say of size $l\in \NN$. Let $b\in P(k-l,k)$ be the partition that arises from $p$ by removing the considered block. As $p$ is projective, it follows that $p=b \cdot b^*$ and again by the structure of $\cC$ the partition $b$ lies also in $\cC$. We set $q:=b^*  \cdot b\in \ProjC(k-2)$ and thus we have $p=p^2=b \cdot b^* \cdot b \cdot b^*=b  \cdot q \cdot b^* \notin \cP$.

It remains to show that any partition $p\in \Projk$ with only through-blocks that has only blocks with at most $2$ upper and at most $2$ lower points lies in $\cP$. Consider such a partition and assume that $p\notin \cP$. Then $p=b^*\cdot q\cdot b$ for some $0\leq l<k, b\in \cC(k, l), q\in \Projl$, and we assume that $ l$ is chosen minimally. By the above considerations this implies that also $q$ has only through-blocks and only blocks with at most $2$ upper and at most $2$ lower points. But as $l<k$ and as all blocks have even size, the composition of partitions $b^*\cdot q\cdot b$ has to have a block that contains at least two more upper points then the corresponding block of $q$. This contradicts the assumptions on $p=b^*\cdot q\cdot b$.
\end{proof}

For our next examples, let us record a useful observation regarding certain projective partitions in a class of group-theoretical categories of partitions.

\begin{definition} \label{def::hki}
For any sequence $k_1,\dots,k_s\in \NN_0$, let us define a projective partition $h_{k_1,k_2,\dots,k_s}$ by setting
$$
q := \{\{1,1'\}\}^{\ot k_1} \ot  \{\{1,2,1'\}\}^{\ot k_2} \ot \dots \ot \{\{1,\dots,s,1'\}\}^{\ot k_s}
\qquad \in P(k, T)
$$
and $h_{k_1,\dots,k_s}:= q^*q \in P_\even(k,k)$
\end{definition}

For instance, $h_{k_1,k_2}=\id_{k_1}\ot (\vierpart)^{\ot k_2}=\id_{k_1}\ot \vierpart\dots\vierpart$.

\begin{lemma} \label{lem::Sp-concrete}
Let $\cC \subseteq P_\even$ be a group-theoretical category of partitions and $k_1,\dots,k_s\in \NN_0$. Then $p:=h_{k_1,\dots,k_s}$ is equivalent to $h_{k_1,0,k_3,0,\dots}\ot h_{0,k_2,0,k_4,\dots}$, and $S(p)\cong S(\id_{k_1+k_3+\dots})\times S_{k_2+k_4+\dots}$.

In particular, for $\cC = P_\even$ we get $S(p)\cong S_{k_1+k_3+\ldots} \times S_{k_2+k_4+\ldots}$.
\end{lemma}

\begin{proof} Throughout the proof, we will use that every group-theoretical category is closed under coarsening (see \cite[Lem.~2.3]{RW14}).

\newcommand\s\sigma
Consider two integers $a,b\geq 1$ and assume one of them is even. We claim that the partition 
$$ r_{ab}:=
\{ \{ 1,\dots,a,(a+1)',\dots,(a+b)'\}, \{1',\dots,a',a+1,\dots,a+b\} \}
$$
lies in $\cC$. Indeed, we may assume $a$ is even, as $r_{ab}=r_{ba}^*$, and in this case, $r_{ab}$ is a coarsening of $(\Uaa)^{\ot (a/2)}\ot\id_b\ot(\Laa)^{\ot (a/2)}$, so in $\cC$. 
Now let $q$ be as in \Cref{def::hki}, so $p=q^*q$, and let us define
\begin{align*}
    q_1 := \{\{1,1'\}\}^{\ot k_1} \ot  \{\{1,2,3,1'\}\}^{\ot k_3} \ot \dots ,\quad p_1:=q_1^*q_1 ,
    \\
    q_2 := \{\{1,2,1'\}\}^{\ot k_2} \ot  \{\{1,2,3,4,1'\}\}^{\ot k_4} \ot \dots ,\quad p_2:=q_2^*q_2 .
\end{align*}
Then $q_1\ot q_2=q r$ and $q=(q_1\ot q_2)r'$, where $r,r'$ are compositions of morphisms of the form $r_{ab}$ as above. Hence by the claim, $r$ and $r'$ are in $\cC$, and $p=q^*q\sim (q_1\ot q_2)^*(q_1\ot q_2)=h_{k_1,0,k_3,0,\dots}\ot h_{0,k_2,0,k_4,\dots}$.

\newcommand\keven{{k_{\text{even}}}}
\newcommand\kodd{{k_{\text{odd}}}}
Set $\kodd:=k_1+k_3+\dots$ and $\keven:=k_2+k_4+\dots$. Now if $\s\in S_k$ is not in (the naturally embedded subgroup) $S_{\kodd}\times S_{\keven}$, then $(q_1\ot q_2)^*\s(q_1\ot q_2)$ has a block of odd size, so $\s\not\in S(p_1\ot p_2)$. Hence, $S(p_1\ot p_2)\cong S(p_1)\times S(p_2)$.

For any odd $m\geq 1$, the partition $\{\{1,\ldots ,m,1'\}\}$ is a coarsening of $(\Uaa)^{\ot(m-1)/2}\ot\id$, and hence, in $\cC$. This immediately implies $q_1\in\cC$, so $S(p_1)=S(\id_\kodd)$. Finally, as any partition $p_2^*\sigma p_2$ with $\sigma \in S_{\keven}$ is a coarsening of $(\twoblocks)^{\ot k'}$, for some $k'$, it follows that $S(p_2)=S_\keven$.
\end{proof}

Now we compute all indecomposable objects in $\uRep(H_t)$ up to isomorphism for $t\in \CC\backslash \{0\}$. It is well-known that for $n\in \NN_0$, inequivalent irreducible complex representations of the hyperoctahedral group $H_n$ can be indexed by bipartitions of size $n$, i.e.~pairs $(\lambda_1,\lambda_2)$ of partitions of some $n_1\leq n$ and $n_2\leq n$, respectively, with $n=n_1+n_2$ (\cite[Sec.~II]{GK76}, see also \cite{orellana}). We show that this description extends to a description of the non-isomorphic indecomposable objects in $\uRep(H_t)=\uRep(P_\even,t)$ by bipartitions of arbitrary size.

Recall the definition of the partitions 
$$
    h_{k_1,k_2}:= \text{id}_{k_1} \ot \vierpart \ot \ldots \ot \vierpart \in P_\even(k_1+2k_2,k_1+2k_2),
$$
for any $k_1,k_2\in \NN_0$ in \Cref{lem::Sp-concrete}.

\begin{proposition} \label{thm::indecomp_obj_Hn} \label{indecomposables_H}
Let $t\in \CC\backslash \{0\}$. Then $S(h_{k_1,k_2})=S_{k_1}\times S_{k_2}$ for all $k_1,k_2\in \NN_0$ and the map 
\begin{align*}
 &\phi: \left\{ \begin{matrix} \text{bipartitions } \lambda=(\lambda_1,\lambda_2) \\ \text{of arbitrary size} \end{matrix} \right\} \longrightarrow \left\{ \begin{matrix} \text{isomorphism classes of non-zero} \\ \text{indecomposable objects in }  \uRep(H_t) \end{matrix} \right\} \\
 &\text{with } \phi(\lambda) = \L_{h_{|\lambda_1|,|\lambda_2|}} (V_{\lambda_1} \ot V_{\lambda_2}) 
\end{align*}
is a bijection.
\end{proposition}

\begin{proof}
We start by showing that the set $E:=\{ h_{k_1,k_2} \mid k_1,k_2\in \NN_0\}$ is a set of representatives for all equivalence classes of projective partitions in $\cP$.

Let $k\in \NN_0$. Since $t(h_{k_1,k_2})=k_1+k_2$ is different for all $k_1,k_2\in \NN_0$ with $k_1+2k_2=k$, the partitions in $E\cap \Projk=\{ h_{k_1,k_2} \mid k_1,k_2\in \NN_0, k_1+2k_2=k\}$ are pairwise inequivalent. 

Let $p$ be projective partition in $\cP$. Then p has only through-blocks and any block has at most $2$ upper and at most $2$ lower points by \Cref{lem::Ht(p)_Q}. We denote by $k_1$ the number of blocks of $p$ of size $2$ and by $k_2$ the number of blocks of $p$ of size $4$. Since $\cross \in P_\even(k,k)$, it is easy to check that $p$ and $h_{k_1,k_2}$ are equivalent. 

By \Cref{lem::Sp-concrete} we have $S(h_{k_1,k_2})=S_{k_1}\times S_{k_2}$. Thus \Cref{thm::indecompsable_obj_by_A_k} yields a bijection between indecomposables in $\RepCt$ up to isomorphism and $\bigsqcup_{k_1,k_2\in \NN_0} \Irr (S_{k_1}\times S_{k_2})$ and hence with bipartitions of arbitrary size.
\end{proof}

We conclude our discussion by computing all indecomposable objects in the interpolation categories $\uRep(H_t^+)=\uRep(NC_\even,t)$ up to isomorphism for $t\in \CC\backslash \{0\}$. In \cite[Thm.~7.3]{BV09}, Banica and Vergnioux showed that for any $n\in \NN_0$, inequivalent irreducible representations of the free hyperoctahedral quantum group $H_n^+$ are indexed by finite binary sequences (of arbitrary length, independent of $n$). We show that also non-isomorphic indecomposable objects in $\uRep(H^+_t)$ are indexed by finite binary sequences.

For any $b\in \NN_0$ and binary sequence $a=(a_1,\ldots,a_b)\in \{0,1\}^b$, we define a partition $h_a \in NC_\even(k,k)$ with $k:=a_1+\dots+a_b$ by $h_a:=h_{a_1} \ot h_{a_2} \ot \cdots \ot h_{a_b}$ with $h_0=\id_1$ and $h_1=\vierpart$. 

\begin{proposition} \label{indecomposables_Hp}
Let $t\in \CC\backslash \{0\}$. Then the map
\[ \phi: \bigcup_{b\in \NN_0} \{0,1\}^b \longrightarrow \left\{ \begin{matrix} \text{isomorphism classes of non-zero} \\ \text{indecomposable objects in }  \uRep(H_t^+) \end{matrix} \right\}, ~a \mapsto \L_{h_a} (V_{\mathrm{triv}}) \]
is a bijection.
\end{proposition}

\begin{proof}
By \cite[Lem.~5.12]{FW16}, the set $E:=\{ h_a \mid b\in \NN_0, a\in \{1,2\}^b, \sum_{i=1}^b a_i=k\}$ is a set of representatives for all equivalence classes of projective partitions. Moreover, by \Cref{lem::Ht(p)_Q}, all of these partitions lie in $\cP$. 

Since $S(p)=\{ \id\}$ for all $p\in \ProjC$ by \Cref{ex::FW}, \Cref{thm::indecompsable_obj_by_A_k} yields a bijection between the indecomposables in $\RepCt$ up to isomorphism and $\bigsqcup_{e\in E} \Irr (\{ \id \})$ and hence the claim follows.
\end{proof}

\newcommand\cS{\mathcal{S}}

\subsection{Temperley--Lieb categories and other non-crossing partition categories} \label{subs::Ot+andSt+}
In \Cref{lem::ex0} we have seen that the indecomposable objects of the Temperley--Lieb category $\uRep(O^+_t)=\uRep(NC_2,t)$ as well as of the category $\uRep(S^+_t)=\uRep(NC,t)$ are indexed by the non-negative integers $\NN_0$. The indecomposable objects of the Temperley--Lieb category have been studied in various settings and can be described using Jones--Wenzl idempotents, discovered by Jones \cite{Jo83}. Even though this is probably known to experts (as the `fattening' procedure), we give a proof that $\uRep(S_{t^2}^+)$ is equivalent to a full subcategory of $\uRep(O_t^+)$ for any $t\in \CC\backslash \{0\}$. Using this, we can also specify the indecomposable objects in $\uRep(S_{t}^+)$. 

The following inductive definition is due to Wenzl \cite{We87}. Set 
\begin{equation} \label{eq::S}
    \cS :=  \{2\cdot \cos \left(\frac{j \pi}{l}\right) \mid l\in \NN_{\geq 2}, j\in \{1,\ldots,l-1\}\} , 
\end{equation} 
then for any $t\notin\cS$ and any $k\in \NN_0$ the Jones--Wenzl idempotent $e_k\in \End_{\uRep(O_t^+)}([k])$ is recursively defined via:
\begin{align*}
&~e_0 = \text{id}_0,~e_1 = \text{id}_1, \\
&\JWiTikz
\end{align*}
with $a_1=0$ and $a_k=(t-a_{k-1})^{-1}$ for all $k\geq 2$.

\begin{example} For instance, $e_2 = \Paa\Paa - \tfrac 1t \twoblocks$.
\end{example}

Using \Cref{thm::indecompsable_obj_by_A_k}, we recover a known result about the Temperley--Lieb categories.

\begin{proposition} \label{prop::Indec_Otp}
For any $t\in \CC\backslash \cS$ 
\[ \phi: \NN_0 \to \left\{ \begin{matrix} \text{isomorphism classes of non-zero} \\ \text{indecomposable objects in }  \uRep(O_t^+) \end{matrix} \right\},~ k\mapsto ([k],e_k) \]
is a bijection.
\end{proposition}

\begin{proof} We apply \Cref{lem::ex0}. Now the assertion follows from the observation that the Jones--Wenzl idempotents recursively defined as above are lifts of the idempotents $\id_k +I_k \in \End ([k]) / I_k$.
\end{proof}

\begin{remark} If $t\in\cS$, then only finitely many Jones--Wenzl idempotents are defined, and the last one of them generates the negligible morphisms in $\uRep_0(NC_2, t)$ (see \cite[Prop.~2.1.]{GW02}). Out of the infinitely many indecomposables in $\uRep(O^+_t)$, only finitely many are not isomorphic to the zero object in the semisimplification of $\uRep(O^+_t)$, the category obtained as a quotient by the tensor ideal of negligible morphisms; they correspond to the finitely many Jones--Wenzl idempotents, except the last one (see for instance, \cite[Thm.~5.4.2]{Ch14}). 
\end{remark}

In the following we show that $\uRep(S_{t^2}^+)$ is equivalent to a full subcategory of $\uRep(O_t^+)$ for any $t\in \CC\backslash \{0\}$.  
\begin{definition}
Let $t\in \CC$. We denote by $\mathcal{D}(t)$ the full subcategory of $\uRep(O_t^+)$ with objects
\[ \{ (A,e)\in \uRep(O_t^+) \mid A=\bigoplus_{i=1}^l [k_i], k_i\in \NN_0 \text{ even, for any } 1\leq i\leq l\} .\]
\end{definition}

Note that $\mathcal{D}(t)$ is the Karoubian envelope of the full subcategory of $\uRep_0(O_t^+)$ with objects $\{ [k]\mid k\in \NN_0 \text{ even} \}$.

\newcommand\relhat{\widehat{\sim}}
\begin{definition}[{\cite[Ex.~9.42]{NS06}}]
Let $k,l\in \NN$. To any partition $p\in NC_2(2k,2l)$ we associate a partition $\widehat{p}\in NC(k,l)$ as follows: First, we define a total order on the set of points $P:=\{1,\dots,2k,1',\dots,(2l)'\}$ of $p$ by setting $1<\dots< 2k< (2l)'\dots< 1'$. Now for any two odd points $x,y\in\widehat P:=\{1,3,\dots,2k-1,1',3',\dots,(2l-1)'\}$ we define the interval
$$
I_{xy} := \{ z\in P: x< z\leq y \text{ or }y<z \leq x\},
$$ and we set $x\relhat y$ if there are no two points $x'\in I_{xy},y'\in P\setminus I_{xy}$ which lie in the same block of $p$. It can be checked that this yields an equivalence relation $\relhat$ on $\widehat P$. We define $\widehat p$  as the partition of $\widehat P$ whose blocks are the equivalence classes of $\relhat$, where we identify $\widehat P$ with the set $\{1,\dots,k,1',\dots,l'\}$ naturally.
\end{definition}

$\widehat p$ can be determined graphically by choosing a suitable string diagram representation of $p$ and inserting a new point to the left of any odd point:

\begin{example}
The following diagram shows that, for $p=\Partition{\Pblock 0 to 0.25:3,4 \Pblock 1 to 0.75:2,3 \Pline (1,0) (1,1) \Pline (2,0) (4,1)}$, we have $\widehat{p}=\Partition{\Pblock 1 to 0.6:1,2 \Pline (1,0) (1,0.6) \Psingletons 0 to 0.25:2}$~:

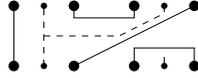
\begin{center} \scalebox{.8}{\begin{tikzpicture}
\coordinate (A1) at (0,0);
\coordinate (A2) at (1,0);
\coordinate (A3) at (2,0);
\coordinate (A4) at (3,0);
\coordinate (B1) at (0,1);
\coordinate (B2) at (1,1);
\coordinate (B3) at (2,1);
\coordinate (B4) at (3,1);
\coordinate (B5) at (0.5,0);
\coordinate (B6) at (2.5,0);
\coordinate (B7) at (0.5,1);
\coordinate (B8) at (2.5,1);

\fill (A1) circle (2.5pt);
\fill (A2) circle (2.5pt);
\fill (A3) circle (2.5pt);
\fill (A4) circle (2.5pt);
\fill (B1) circle (2.5pt);
\fill (B2) circle (2.5pt);
\fill (B3) circle (2.5pt);
\fill (B4) circle (2.5pt);
\fill (B5) circle (1.5pt);
\fill (B6) circle (1.5pt);
\fill (B7) circle (1.5pt);
\fill (B8) circle (1.5pt);

\draw (A1) -- (B1);
\draw (A3) -- (2,0.3) -- (3,0.3) -- (A4);
\draw (A2) -- (B4);
\draw (B2) -- (1,0.8) -- (2,0.8) -- (B3);

\draw[dashed] (B5) -- (B7);
\draw (B6) -- (2.5,0.15);
\draw[dashed] (0.5,0.5) -- (1.75,0.5) -- (2.5,0.875) -- (B8);
\end{tikzpicture}}\end{center}
\end{example}

It is well-known that the map $NC_2(2k,2l)\to NC(k,l),~p\mapsto \widehat{p}$, called \emph{fattening operation}, is a bijection, see \cite[Ex.~9.42]{NS06}. We will now show that, together with a suitable scaling, this map induces an equivalence of monoidal categories between $\mathcal{D}(t)$ and $\uRep(S_{t^2}^+)$.

\begin{definition} Let $t\in \CC \backslash \{0\}$ and let $\sqrt{t}\in \CC$ be any square root of $t$. We denote the trace in $\uRep(O_t^+)$ and $\uRep(S_{t^2}^+)$ by $\text{tr}$ and $\text{tr}_2$, respectively. We set
$$
  \mathcal{G}([2k]) := [k]  ,\quad
   \mathcal{G}(p) := a(p) \widehat{p} \quad \in NC(k,l)
  \qquad\text{for all }k,l\in\NN_0, p\in NC_2(2k,2l) ,
$$
where
\begin{align*}
 p' &:= \begin{cases}
    p\ot \Uaa \ot \cdots \ot \Uaa \in NC_2(2l,2l) & l\geq k , \\
    p\ot \Laa \ot \cdots \ot \Laa \in NC_2(2k,2k) & k\geq l ,
    \end{cases}
    \\
  a(p) &:= \left( \sqrt{t}\right)^{|k-l|} \frac{\text{tr}(p')}{\text{tr}_2(\widehat{p'})} .
\end{align*}
\end{definition}

\begin{lemma} \label{lem::Rechenregeln_ap}
We make the same assumptions as in the above lemma and let $p\in NC_2(2k,2l)$.
\begin{enumerate}[label=(\roman*)]
    \item We have $a(p) = \left( \sqrt{t}\right)^{k-l} a(p')$.
    \item We have $a(p \ot \id_{2m})=a(p)$ for all $m\in \NN_0$.
    \item If $k=l$, then $a(p \ot r) = \frac{1}{t^y} a(p)$ with $r=\twoblocks\ot \cdots \ot \twoblocks \in P(2y,2y)$ for all $y\in\NN_0$.
\end{enumerate}
\end{lemma}

\begin{proof}
\begin{enumerate}[label=(\roman*)]
\item The claim follows directly from the definition of $\mathcal{G}$, since                 
    $a(p')=\frac{\text{tr}(p')}{\text{tr}_2(\widehat{p'})}$.
\item Let $q=p \ot \id_{2m}$. If $k=l$, then we have $\widehat{q}=\widehat{p}\ot \id_2$ and hence
    \[ a(q)= \frac{\text{tr}(q)}{\text{tr}_2(\widehat{q})} 
    = \frac{\text{tr}(p) \cdot t^{2m}}{\text{tr}_2(\widehat{p}) \cdot (t^2)^{m} } = a(p).\] 
    Now, let $k>l$. The case $k<l$ follows analogously. Without loss of generality we assume that $m=2$. By $(i)$ we have to show that $a(q')=a(p')$. We have
    \begin{center}\scalebox{0.8}{\begin{tikzpicture}
        \coordinate [label=left:{\scalebox{1.25}{$\tr(q')=\tr($}}](O) at (0,1);
        \coordinate [label=right:{\scalebox{1.25}{$)=\tr(p')$.}}](O) at (7,1);
        \coordinate [label=left:{\scalebox{1.25}{$p$}}](O) at (1.5,1);
        \coordinate [label=left:{$\ldots$}](O) at (5.2,0);
        \coordinate (A1) at (0,2);
        \coordinate (A2) at (5,2);
        \coordinate (A3) at (5.5,2);
        \coordinate (A4) at (6,2);
        \coordinate (B1) at (0,0);
        \coordinate (B2) at (2,0);
        \coordinate (B3) at (2.5,0);
        \coordinate (B4) at (3,0);
        \coordinate (B5) at (3.5,0);
        \coordinate (B6) at (4,0);
        \coordinate (B7) at (5.5,0);
        \coordinate (B8) at (6,0);
        
        \fill (A1) circle (2.5pt);
        \fill (A2) circle (2.5pt);
        \fill (A3) circle (2.5pt);
        \fill (A4) circle (2.5pt);
        \fill (B1) circle (2.5pt);
        \fill (B2) circle (2.5pt);
        \fill (B3) circle (2.5pt);
        \fill (B4) circle (2.5pt);
        \fill (B5) circle (2.5pt);
        \fill (B6) circle (2.5pt);
        \fill (B7) circle (2.5pt);
        \fill (B8) circle (2.5pt);
        
        \draw[dashed] (A1) -- (A2) -- (B2) -- (B1) -- (A1);
        \draw (A3) to [bend left=90] (7,2) -- (7,0) to [bend left=90] (B7) ;
        \draw (A4) to [bend left=90] (6.5,2) -- (6.5,0) to [bend left=90] (B8) ;
        \draw (A3) -- (B3);
        \draw (A4) -- (B4);
        \draw (B5) -- (3.5,0.2) -- (4,0.2) -- (B6);
        \draw (B7) -- (5.5,0.2) -- (6,0.2) -- (B8);
    \end{tikzpicture}} \end{center}
    Analogously one can check that $\tr_2 (\widehat{q'}) = \tr_2(\widehat{p'})$ and hence 
    \[ a(q') = \frac{\text{tr}(q')}{\text{tr}_2(\widehat{q'})} = \frac{\text{tr}(q')}{\text{tr}_2(\widehat{q'})} a(p') .\]
\item Since $k=l$, we have $\widehat{p\ot r} = \widehat{p} \ot \widehat{r}$ and $\widehat{r}=\bbar \ot \cdots \ot \bbar \in P(y,y)$. It follows that
    \[ a(p\ot r)= \frac{\text{tr}(p\ot r)}{\text{tr}_2(\widehat{p\ot r})} 
    = \frac{\text{tr}(p) \cdot t^{y}}{\text{tr}_2(\widehat{p}) \cdot (t^2)^{y} } = \frac{1}{t^y} a(p).\] 
\end{enumerate}
\end{proof}

\begin{lemma} \label{lem::equivalence_St+_Ot+} $\mathcal{G}_t:=\mathcal{G}$ defines an equivalence of monoidal categories $\mathcal{D}(t)\to \uRep(S_{t^2}^+)$ for all $t\in \CC \backslash \{0\}$.
\end{lemma}

\begin{proof}
It suffices to show that $\mathcal{G}$ is a monoidal functor since $\mathcal{G}$ is full, faithful and essentially surjective, as $p\mapsto \widehat{p}$ is a bijection.

{Step 1:} We start by showing that $\mathcal{G}(q\circ p)=\mathcal{G}(q)\circ \mathcal{G}(p)$ for all $p\in NC_2(2k,2l)$ and $q\in NC_2(2l,2m)$. Comparing diagrams one can check that $\widehat{qp}=\widehat{q}\widehat{p}$. Together with
\begin{align*}
    &\mathcal{G}(q\circ p) = t^{\ell(q,p)} \mathcal{G}(qp) = t^{\ell(q,p)} a(qp) \widehat{qp},\\
    &\mathcal{G}(q)\circ \mathcal{G}(p) = a(p) q(p) (\widehat{q} \circ \widehat{p}) =  a(p) q(p) (t^2)^{\ell(\widehat{q},\widehat{p})} (\widehat{q}\widehat{p}),
\end{align*}
it follows that it suffices to show that $t^{\ell(q,p)}a(qp)=(t^2)^{\ell(\widehat{q},\widehat{p})} a(p)a(q)$. 

{Step 1.1:}
Kodiyalam and Sunder showed that for any $n\in \NN_0$ the map 
\[ \End_{\uRep(O_t^+)}([2n]) \to \End_{\uRep(S_{t^2}^+)}([n]), ~p\mapsto \mathcal{G}(p) \]
is an algebra isomorphism (see {\cite[Thm.~4.2]{KS08}}). Hence the claim follows for $k=l=m$.

{Step 1.2:}
For arbitrary $k,l,m\in \NN_0$ we set $x:=\text{max}(k,l,m)$ and extend $p$ and $q$ to partitions in $P(x,x)$ as follows:
\begin{align*}
&\Bar{p}:= p\ot \Uaa \ot \cdots \ot \Uaa \ot \Laa \ot \cdots \ot \Laa \in P(x,x), \\
&\Bar{q}:= q\ot \Uaa \ot \cdots \ot \Uaa \ot \Laa \ot \cdots \ot \Laa \in P(x,x).
\end{align*}
Step 1.1 implies that
\begin{equation*}
    t^{\ell(\Bar{q},\Bar{p})} a(\Bar{q}\Bar{p}) = (t^2)^{\ell(\widehat{\Bar{p}},\widehat{\Bar{q}})}~ a(\Bar{q}) a(\Bar{p})
\end{equation*}
 Moreover, by construction we have
\begin{align*}
    &t^{\ell(\Bar{q},\Bar{p})} = t^{\ell(q,p)} t^{x-l} \\
    &(t^2)^{\ell(\widehat{\Bar{p}},\widehat{\Bar{q}})} = (t^2)^{\ell(\widehat{p},\widehat{q})} (t^2)^{x-l}.
\end{align*}
and thus
\begin{equation}
    t^{\ell(q,p)} a(\Bar{q}\Bar{p}) = t^{x-l} (t^2)^{\ell(\widehat{p},\widehat{q})}~ a(\Bar{q}) a(\Bar{p}).
\end{equation}

{Step 1.3:}
We claim that 
\begin{align}
    &a(p) = \left(\sqrt{t}\right)^{x-k} \left(\sqrt{t}\right)^{x-l} a(\Bar{p}),\\
    &a(q) = \left(\sqrt{t}\right)^{x-l} \left(\sqrt{t}\right)^{x-m} a(\Bar{q}),\\
    &a(qp) = \left(\sqrt{t}\right)^{x-k} \left(\sqrt{t}\right)^{x-m} a(\Bar{q}\Bar{p}).
\end{align}
We prove the first equation since the others follow analogously. If $x\in \{k,l\}$, then we have $\Bar{p}=p'$ and hence \Cref{lem::Rechenregeln_ap}(i) implies $a(p)=\left( \sqrt{t}\right)^{|k-l|} a(p') =\left(\sqrt{t}\right)^{x-k} \left(\sqrt{t}\right)^{x-l} a(\Bar{p})$.\\
If $x=m$, then we have $\Bar{p}=p'\ot r$ with $r=\twoblocks\ot \cdots \ot \twoblocks \in P(2y,2y)$. \Cref{lem::Rechenregeln_ap}(iii) implies that $a(p')= t^{y} a(\Bar{p})$ and together with \Cref{lem::Rechenregeln_ap}(i) it follows that $a(p)=\left( \sqrt{t}\right)^{|k-l|} a(p') = \left( \sqrt{t}\right)^{|k-l|} t^{y} a(\Bar{p}) = \left(\sqrt{t}\right)^{x-k} \left(\sqrt{t}\right)^{x-l} a(\Bar{p})$.

{Step 1.4:}
We are ready to show that $t^{\ell(q,p)}a(qp)=(t^2)^{\ell(\widehat{q},\widehat{p})} a(p)a(q)$. We have
\begin{align*}
    &t^{\ell(q,p)} a(qp) \\
    \overset{(4)}{=}~~~& \left(\sqrt{t}\right)^{x-k} \left(\sqrt{t}\right)^{x-m} t^{\ell(q,p)}~ a(\Bar{q} \Bar{p}) \\
    \overset{(1)}{=}~~~& \left(\sqrt{t}\right)^{x-k} \left(\sqrt{t}\right)^{x-m}~t^{x-l}~ (t^2)^{\ell(\widehat{p},\widehat{q})}~ a(\Bar{q}) a(\Bar{p}) \\
    =~~~& (t^2)^{\ell(\widehat{p},\widehat{q})} \left(\left(\sqrt{t}\right)^{x-l} \left(\sqrt{t}\right)^{x-m}  a(\Bar{q})\right) \left(\left(\sqrt{t}\right)^{x-k} \left(\sqrt{t}\right)^{x-l} a(\Bar{p})\right) \\
    \overset{(2),(3)}{=}~& (t^2)^{\ell(\widehat{p},\widehat{q})} a(q) a(p). 
\end{align*}

{Step 2:}
It remains to show that $\mathcal{G}(p\ot q)=\mathcal{G}(p)\ot \mathcal{G}(q)$ for all $p\in NC_2(2k,2l), q\in NC_2(2m,2n)$. Again by comparing diagrams one can check that $\widehat{p\ot q}=\widehat{p}\ot \widehat{q}$ and thus we have to show that $a(p\ot q)=a(p)a(q)$. By Step 1 we have
\begin{align*}
    &a(p\ot q) \\
    =~& a( (\id_{2l} \ot q)(p \ot \id_{2m}) ) \\ 
    =~& a(\id_{2l} \ot q) a(p \ot \id_{2m}).
\end{align*}
and since $a(\id_{2l} \ot q)=a(q)$ and $a(p \ot \id_{2m})=a(p)$ by \Cref{lem::Rechenregeln_ap}(ii), the claim follows. 
\end{proof}

Since there are no non-zero morphisms in $\uRep(O_t^+)$ between subobjects of $[k_1]$ with $k_1$ even and subobjects of $[k_2]$ with $k_2$ odd, \Cref{prop::Indec_Otp} and \Cref{lem::equivalence_St+_Ot+} imply the following:

\begin{proposition} \label{lem::indecomposables_St+}
If $t\in \CC\setminus\{s^2\mid s\in\cS\} = \CC \backslash \{4\cdot \cos \left(\frac{j \pi}{l}\right)^2 \mid l\in \NN_{\geq 2}, j\in \{1,\ldots,l-1\}\}$, then 
\[ \phi: \NN_0 \to \left\{ \begin{matrix} \text{isomorphism classes of non-zero} \\ \text{indecomposable objects in }  \uRep(S_t^+) \end{matrix} \right\},~ k\mapsto ([k],\mathcal{G}(e_{2k})) \]
is a bijection.
\end{proposition}

\newcommand\cT{\mathcal{T}}
We can also use the functor $\cG$ to compute semisimplifications in the non-crossing case. Let us pick $q\in\CC$ such that $q^2$ is a primitive $l$-th root of unity for some $l\geq 2$, and let $\cT_q$ be the semisimplification of the category of tilting modules of the quantum group $U_q(\mathfrak{sl}_2)$, i.e., the fusion category studied by Andersen (see \cite{An}). It is known to be a hom-finite $C^*$-category (\cite[Thm.~3.7]{We-Cstar}, see also \cite{Ki}) in the sense of \Cref{semisimplicity-karoubian}, and there is an equivalence $*$-functor $\cQ\colon \uRep(O^+,-(q+q^{-1}))\to\cT_q$ sending the object $[1]$ to the standard representation (for instance, by \cite[Thm.~2.5.3]{NT13} or \cite[Thm.~2.4]{Os}). 

Let us assume $q=\exp(\tfrac{j\pi}{l})$ with $l\in\NN_{\geq2}$ and $j\in\{1,\dots,\l-1\}$, so $t:=q+q^{-1}=2\cos(\tfrac{j\pi}{l})\in\cS$ (see \Cref{eq::S}).

\begin{proposition} \label{prop::semisimplification-nc} For any category of non-crossing partitions $\cC$ and $t=q+q^{-1}\in\cS$, the semisimplification of $\uRep(\cC,t^2)$ is given by the image of the functor $\cQ\circ\cG_{-t}^{-1}$ in $\cT_q$.
\end{proposition}

\begin{proof} This follows from \Cref{lem-semisimplification} with $\cD=\cT_q$.
\end{proof}

This result can be seen as a non-crossing analogue to \Cref{prop::fiber-functor}.
In particular, the semisimplification of $\uRep(S^+,t^2)$ is equivalent to $\cT^{\text{even}}_q$, the Karoubian subcategory in $\cT_q$ generated by the tensor square of the standard representation. 

\printbibliography

\end{document}